\documentclass[12pt]{amsart}

\usepackage[latin2]{inputenc}
\usepackage[T1]{fontenc}
\usepackage{euscript}
\usepackage{amsmath}
\usepackage{amsthm}
\usepackage{amssymb}
\usepackage[matrix,arrow]{xy}
\usepackage{enumerate}
\usepackage{times}
\usepackage{amscd}
\usepackage{amsmath}
\usepackage{epsfig}
\usepackage{epic}
\usepackage{eufrak}

\numberwithin{equation}{section}
\numberwithin{equation}{subsection}

\theoremstyle{plain}
\newtheorem{theorem}[equation]{Theorem}
\newtheorem{lemma}[equation]{Lemma}
\newtheorem{proposition}[equation]{Proposition}
\newtheorem{corollary}[equation]{Corollary}

\theoremstyle{definition}
\newtheorem{definition}[equation]{Definition}
\newtheorem{example}[equation]{Example}
\newtheorem{remark}[equation]{Remark}

\newtheorem{bekezd}[equation]{}
\newtheorem{bekezdes}[equation]{}

\newcommand{\cO}{\mathcal{O}}
\newcommand{\cH}{\mathcal{H}}

\newcommand{\cV}{\mathcal{V}}

\newcommand{\cF}{\mathcal{F}}
\newcommand{\cP}{\mathcal{P}}

\newcommand{\cS}{\mathcal{S}}

\newcommand{\calN}{\mathcal{N}}

\newcommand{\calT}{\mathcal{T}}

\newcommand{\calR}{\mathcal{R}}

\newcommand{\calW}{{\mathcal W}}
\newcommand{\cale}{\mathcal{E}}\newcommand{\calK}{\mathcal{K}}
\newcommand{\calP}{\mathcal{P}}
\newcommand{\calF}{\mathcal{F}}

\newcommand{\bz}{{\bf z}}

\DeclareMathOperator{\Hom}{{\rm Hom}}

\newcommand{\setQ}{\mathbb{Q}}\newcommand{\Q}{\mathbb{Q}}
\newcommand{\R}{\mathbb{R}}
\newcommand{\setZ}{\mathbb{Z}}\newcommand{\Z}{\mathbb{Z}}

\newcommand{\hh}{\mathfrak{h}}
\newcommand{\bt}{{\bf t}}
\newcommand{\bl}{{\bf l}}\newcommand{\bll}{l}

\newcommand{\frsw}{\mathfrak{sw}}
\newcommand{\frd}{\mathfrak{d}}\newcommand{\frdd}{\mathfrak{d}}
\newcommand{\fro}{\mathfrak{o}}
\newcommand{\fra}{\mathfrak{a}}
\newcommand{\bc}{{\bf c}}
\newcommand{\tc}{\widetilde{c}}
\newcommand{\wtt}{\widetilde{t}}
\newcommand{\well}{\widetilde{\aalpha}}

\providecommand{\coloneqq}{\mathrel{:=}}

\newcommand{\labelpar}{\label}

\def\C{\mathbb C}
\def\Q{\mathbb Q}
\def\R{\mathbb R}
\def\Z{\mathbb Z}
\def\N{\mathbb N}

\newcommand{\calX}{{\mathcal X}}
\newcommand{\calv}{{\mathcal V}}
\newcommand{\calC}{{\mathcal C}}
\newcommand{\calL}{{\mathcal L}}
\newcommand{\calS}{{\mathcal S}}

\newcommand{\aalpha}{\ell}

\newcommand{\frv}{\mathfrak{v}}
\newcommand{\frl}{\mathfrak{l}}
\newcommand{\fH}{\mathfrak{H}}
\newcommand{\red}{\dagger}

\newcommand{\bms}{\mbox{\boldmath$s$}}

\title{Ehrhart theory of polytopes and Seiberg--Witten invariants of plumbed 3--manifolds}

\author{Tam\'as L\'aszl\'o}
\address{Central European University and A. R\'enyi Institute of Mathematics,
1053 Budapest,  Re\'altanoda u. 13-15, Hungary.}
\email{laszlo.tamas@renyi.mta.hu}
\thanks{T. L\'aszl\'o is supported by `Lend\"ulet' and ERC programs LTDBud at R\'enyi Institute,
A. N\'emethi is partially supported by OTKA Grant 100796.}

\author{Andr\'as N\'emethi}
\address{A. R\'enyi Institute of Mathematics, 1053 Budapest,  Re\'altanoda u. 13-15,  Hungary.}
\email{nemethi.andras@renyi.mta.hu}



\date{}

\begin{document}

\maketitle

\pagestyle{myheadings} \markboth{{\normalsize
T. L\'aszl\'o and A. N\'emethi}}{{\normalsize  Ehrhart theory and Seiberg--Witten invariants}}

\begin{abstract}
Let $M$  be a rational homology sphere
plumbed 3--manifold associated with a connected negative definite plumbing graph.
We show that its Seiberg--Witten invariants equal certain coefficients of
an equivariant multivariable Ehrhart polynomial. For this,
we construct the corresponding polytopes from the plumbing graphs together with an action of
$H_1(M,\Z)$ and we develop Ehrhart theory for them.
At an intermediate level we define the `periodic constant' of multivariable series and establish their properties. In this way, one identifies the Seiberg--Witten invariant of a plumbed 3--manifold,
the periodic constant of its `combinatorial zeta--function', and a coefficient of the
associated Ehrhart polynomial.
We make detailed presentations for graphs with at most two nodes. The two node case
 has surprising connections with the theory of affine monoids of rank two.
\end{abstract}

\section{Introduction}\label{sec:introduction}

\subsection{} The main motivation of the present article is the combinatorial computation
of the Seiberg--Witten invariants of negative definite plumbed 3--manifolds.
The final output is the identification of these invariants
with certain coefficients of a multivariable equivariant Ehrhart polynomial.

Let $\Gamma$ denote the  connected negative definite decorated plumbing graph with vertices $\cV$,
which determines the oriented  plumbed 3--manifold $M=M(\Gamma)$.
We assume that $\Gamma$ is a tree, and all the plumbed surfaces have genus zero, that is, $M$
is a rational homology sphere. We denote by  $\frsw_{\sigma}(M)$
the Seiberg--Witten invariants of $M$
indexed by the $spin^c$--structures $\sigma$ of $M$.

In the last years
several  combinatorial expressions were established regarding the Seiberg--Witten invariants.
In \cite{Nico5} Nicolaescu  proved (based on the surgery
formulas of \cite{MW}) that they are equivalent with Turaev's torsion
normalized by the Casson--Walker invariant. In terms of $\Gamma$,
 a combinatorial formula for the Casson--Walker invariant
 can be deduced from Lescop's book \cite{Lescop},
while  the Turaev's torsion  is determined in \cite{NN1} in terms of a Dedekind--Fourier sum.

For some special graphs,  when the Heegaard--Floer homology is determined,
we obtain the Seiberg--Witten invariant as the normalized Euler characteristic of the
Heegaard--Floer homology \cite{OSzP,OSz,OSz7}.
They can be determined inductively by surgery formulae as well, see e.g.
\cite{OSzP,NOSZ,Rus}. \cite{BN} provides a different type of surgery formula (which is not
induced by an exact triangle, but involves  the periodic constant of a series ---
more in the spirit of the present work).
In parallel, one can rely on the lattice cohomology too (introduced in \cite{NOSZ,NLC}):
in \cite{NJEMS} the second author proved that the Seiberg--Witten invariant is the normalized
Euler characteristic of the lattice cohomology of $M$. Hence, the surgery formulae \cite{NES}, and
closed formulae for specific families \cite{NR,trieste} provide further examples.

\subsection{} The starting point of the present article is the result of \cite{NJEMS}, when the
Seiberg--Witten invariant appears as the {\it periodic constant of a multivariable series}.
Next we provide some details.

Let us consider the plumbed 4--manifold $\widetilde{X}$ associated with
$\Gamma$. Its second
homology $L$ is freely generated by the 2--spheres $\{E_v\}_{v\in\cV}$, and its
second cohomology $L'$
by the (anti)dual classes $\{E^*_v\}_{v\in\cV}$; the intersection
form $I=(\,,\,)$ embeds $L$ into $L'$.
 Set $x^2:=(x,x)$.

Let $K\in L'$ be the canonical class (given by the adjunction relations),
$\widetilde{\sigma}_{can}$ the canonical
$spin^c$--structure on $\widetilde{X}$ with $c_1(\widetilde{\sigma}_{can})=-K$,
and $\sigma_{can}\in \mathrm{Spin}^c(M)$ its restriction on $M$. Set $H:=H_1(M,\Z)=L'/L$.
Then  $Spin^c(M)$ is an $H$--torsor, with action denoted by $*$.

Next,  consider the  multivariable
Taylor expansion $Z(\bt)=\sum p_{l'}\bt^{l'}$ at the  origin of
\begin{equation*}\prod_{v\in \cV} (1-\bt^{E^*_v})^{\delta_v-2},\end{equation*}
where for any $l'=\sum _vl_vE_v\in L'$ we write
$\bt^{l'}=\prod_vt_v^{l_v}$, and $\delta_v$ is the valency of $v$.
 This lives in $\setZ[[L']]$, the submodule of
formal power series $\setZ[[\bt^{\pm 1/d}]]$ in variables $\{t_v^{\pm 1/d}\}_v$, where $d=\det(-I)$.
It has a natural decomposition $Z(\bt)=\sum_{h\in H}Z_h(\bt)$, where $Z_h(\bt)=\sum_{[l']=h}p_{l'}\bt^{l'}$ (where $[l']$ is the class of $l'$).
 Then $\frsw_{-h*\sigma_{can}}(M)$ can be deduced from $Z_h$ as follows  \cite{NJEMS}.

Assume that  $l'=\sum_va_vE^*_v$ satisfies $a_v\geq -(E_v^2+1)$. Then
\begin{equation}\label{eq:INTR1}
\sum_{l\in L,\, l\not\geq 0}p_{l'+l }=
-\frac{(K+2l')^2+|\cV|}{8}-\frsw_{[-l']*\sigma_{can}}(M).\end{equation}
The left hand side appears as a {\it counting function}
of the coefficients of $Z_h$ associated with a special {\it truncation}, while the right hand side is a
multivariable quadratic polynomial whose {\it free term} is the  normalized Seiberg--Witten invariant.
In order to guarantee the validity of the formula, the vector $l'$ should sit in a special
{\it chamber} described by the inequalities of the assumption. This, after we
establish the necessary  bridges,
will read as follows: `the third degree' coefficient of a multivariable
Ehrhart polynomial associated with a certain polytope and specific chamber can be
identified with the SW invariant.

In fact, the way how one recovers the needed  information from the series can be done at several
levels. The first one is entirely at the level of series (or Taylor expansions of rational functions).
We develop a theory which associates with any series the counting function of its coefficients
(given by a truncation of the monomials) --- like the right hand side of (\ref{eq:INTR1}).
This, usually is a piecewise quasipolynomial. Once we fix a chamber, the free term of the counting
function is the so called `periodic constant' (denoted by ${\rm pc}$). In this terminology,
 the Seiberg--Witten invariant can be interpreted
as the {\it multivariable periodic constant} ${\rm pc}(Z)$
of the series $Z(\bt)$, where the chosen chamber is described by the inequalities of the
assumption (a part of the `Lipman cone'). The `periodicity'
is related with the quasipolynomial behavior of the counting function.)
The  `periodic constant' of one variable  series
was introduced in \cite{NO1,Opg}, and it had several applications (see e.g. \cite{NOz,NO1,NJEMS,BN}).
Here we create the general theory, which carries necessarily several
difficult technical ingredients  (e.g. one has to choose the `right'
truncation and summation procedure of the coefficients, which in the context of general series is not automatically motivated, and also it depends on the chamber decomposition of the space of exponents).
The theory has some similarities with the theory of vector partition functions.

On the other hand, there is a more sophisticated way to generalize the idenity (\ref{eq:INTR1}) too.

From any Taylor  expansion of a multivariable rational function
 with  denominator of type  $\prod_i(1-\bt^{a_i})$  we construct a polytope
situated in a lattice which carries also a representation of a finite abelian group $H$.
Associated with these data
 we consider the equivariant multivariable Ehrhart  piecewise quasipolynomials,
whose existence, main properties (like the Ehrhart--MacDonald--Stanley type reciprocity law
 or chamber decompositions)
 will also be established. This applied to the series $Z(\bt)$ above,
and to the quasipolynomial of those chambers which belong to the Lipman cone shows that
 the first three  top--degree coefficients (at least)
will carry gemometrical/topological meaning, including the SW invariants of the link.
(This coefficient identifications, and in fact
(\ref{eq:INTR1}) too, supplies an additional addendum to the intimate relationship between
lattice point counting  and the Riemann--Roch formula,   exploited in global algebraic geometry
by toric geometry.)

Here is a schematic picture of these connections and areas we target:

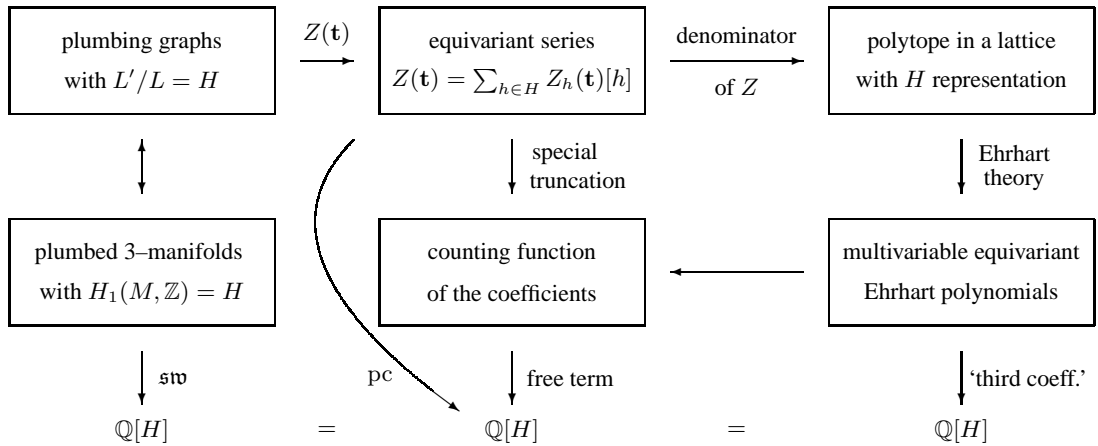
\begin{picture}(400,180)(0,60)
\put(10,180){\framebox(100,40){}}
\put(60,207){\makebox(0,0){\scriptsize{plumbing graphs}}}
\put(60,192){\makebox(0,0){\scriptsize{with $L'/L=H$}}}

\put(10,100){\framebox(100,40){}}
\put(60,127){\makebox(0,0){\scriptsize{plumbed 3--manifolds }}}
\put(60,112){\makebox(0,0){\scriptsize{with $H_1(M,\Z)=H$}}}

\put(60,150){\vector(0,1){20}}\put(60,170){\vector(0,-1){20}}
\put(60,90){\vector(0,-1){20}}
\put(60,60){\makebox(0,0){\scriptsize{$\Q[H]$}}}
\put(72,80){\makebox(0,0){\scriptsize{$\frsw$}}}
\put(120,200){\vector(1,0){20}}
\put(130,210){\makebox(0,0){\scriptsize{$Z(\bt)$}}}
\put(130,60){\makebox(0,0){\scriptsize{$=$}}}
\qbezier(140,170)(100,130)(170,75)
\put(170,75){\vector(4,-3){10}}
\put(150,80){\makebox(0,0){\scriptsize{${\rm  pc}$}}}
\put(150,180){\framebox(100,40){}}
\put(200,207){\makebox(0,0){\scriptsize{equivariant series}}}
\put(200,192){\makebox(0,0){\scriptsize{$Z(\bt)=\sum_{h\in H}Z_h(\bt)[h]$}}}

\put(150,100){\framebox(100,40){}}
\put(200,127){\makebox(0,0){\scriptsize{counting function}}}
\put(200,112){\makebox(0,0){\scriptsize{of the coefficients}}}

\put(220,165){\makebox(0,0){\scriptsize{special}}}
\put(225,155){\makebox(0,0){\scriptsize{truncation}}}
\put(222,80){\makebox(0,0){\scriptsize{free term}}}
\put(200,170){\vector(0,-1){20}}
\put(200,90){\vector(0,-1){20}}
\put(200,60){\makebox(0,0){\scriptsize{$\Q[H]$}}}

\put(260,200){\vector(1,0){50}}
\put(285,210){\makebox(0,0){\scriptsize{denominator}}}
\put(285,190){\makebox(0,0){\scriptsize{of $Z$}}}
\put(285,60){\makebox(0,0){\scriptsize{$=$}}}
\put(310,120){\vector(-1,0){50}}

\put(320,180){\framebox(100,40){}}
\put(370,207){\makebox(0,0){\scriptsize{polytope in a lattice}}}
\put(370,192){\makebox(0,0){\scriptsize{with $H$ representation}}}

\put(320,100){\framebox(100,40){}}

\put(370,127){\makebox(0,0){\scriptsize{multivariable equivariant}}}
\put(370,112){\makebox(0,0){\scriptsize{Ehrhart polynomials}}}

\put(390,165){\makebox(0,0){\scriptsize{Ehrhart}}}
\put(390,155){\makebox(0,0){\scriptsize{theory}}}
\put(395,80){\makebox(0,0){\scriptsize{`third coeff.'}}}
\put(370,170){\vector(0,-1){20}}
\put(370,90){\vector(0,-1){20}}
\put(370,60){\makebox(0,0){\scriptsize{$\Q[H]$}}}

\end{picture}

\subsection{} The number of terms in the denominator $\prod_i(1-\bt^{a_i})$
of the series equals the number of variables of the corresponding partition function
(associated with vectors $a_i$), and it is also
the rank of the lattice where the corresponding polytope sit.
In the case of the series $Z(\bt)$ associated with plumbing graph, this is the number of
{\it end vertices} of $\Gamma$.
On the other hand, the number of variables of $Z(\bt)$ is the number $|\calv|$
of vertices of $\Gamma$.
Furthermore, in the Ehrhart theoretical part, the associated (non--convex) polytope will be a
union of $|\calv|$ simplicial polytopes. Hence, with the number of vertices, the number of facets and the
complexity of the polytope  increases considerably as well.

Nevertheless,  the Reduction Theorem \ref{th:REST} eliminates a part of this abundance of
parameters: it says that from the periodic constant point of view, the number of variables of the series,
and also the number of simplicial polytopes in the union, can be reduce
 to the number of {\it nodes} of the graph.
Hence, in fact, the complexity level is measured  by the number of nodes.

In the body of the article, besides the general theory,  we make detailed computations for
graphs with less than two nodes.
Even in the special case of graphs without nodes (that is, the case of lens spaces)
 the  description of the equivariant Ehrhart quasipolynomials is new.
 In the one node case (start shaped graphs)
we provide a detailed presentation of all the involved (SW and Ehrhart) invariants, and
 we establish closed formulae in terms of the Seifert invariants. Here we make connection
 with already known topological results regarding the Seiberg--Witten invariants of Seifert
   3--manifolds, and also with analytic invariants of weighted homogeneous singularities.

In  the two node case  again we make complete presentations in terms of the analogs of the
Seifert invariants of the chains and star--shaped subgraphs,  including closed formulae for $\frsw(M)$.
But, this case has a very interesting additional surprise
in store. It turns out that the corresponding combinatorial series $Z(\bt)$ associated with
$\Gamma$, reduced to the two variables of the nodes, is the Hilbert (characteristic)
series of  an affine monoid of rank two (and some of its modules).
In particular, the Seiberg--Witten invariant appears as the periodic constant
of Hilbert series associated with affine monoids (and certain modules indexed by $H$), and,
in some sense, measures the non--normality of these monoids.

\subsection{} It is important to emphasize that the origin (and main motivation)
of the identity (\ref{eq:INTR1}) was an analytic identity.
Recall that the manifolds $M$  appear as links of complex normal surface singularities, and
several of the above  objects have their analytic counterparts.
For example, the analogue of the series $Z(\bt)$  is the
Hilbert series associated with the multivariable equivariant divisorial filtration of the
local ring of the singular germ, and its equivariant periodic constants are  the
 equivariant geometric genera. In the body of the paper we emphasize this parallelism
 as well, whenever the corresponding analytic invariants coincide with the topological ones.
 This happens e.g. in the case of star--shaped graphs and the weighted homogeneous
 analytic structures carried by them. For further relations with analytic structures
 (e.g. for the Seiberg--Witten
 Invariant Conjecture targeting these type of connections),  see  \cite{NN1,Line}.

The relevant terminology and additional connections with theory of complex
singularities can be found in \cite{AGV,CHR,CDG,EN,Nfive,CDGb,coho3}, its connection with 
 Seiberg Witten theory
in \cite{BN,BN07,NOSZ,trieste,NLC,NJEMS,NN1,NN2,Nico5}.  
For some results in Ehrhart theory,  relevant to the present work, see
\cite{Bar,BP,Beck_c,Beck_m,Beck_p,BR1,BR2,BDR,CL,DR}, while for  partition
functions, see \cite{BV,SZV,Str}.

\subsection{} The titles of the sections are the following; they show also the organization of the paper.

2. \ Normal surface singularities. The main motivation

3. \ Equivariant multivariable Ehrhart theory

4. \ Multivariable rational functions and their periodic constants

5. \ The case of rational functions associated with plumbing graphs

6. \ The one--node case, star--shaped plumbing graphs

7. \ The two--node case

8. \ Ehrhart theoretical interpretation of the SW invariant (the general case).

\section{Normal surface singularities. The main Motivation}
\labelpar{sec:main-results}

\subsection{Surface singularities and their links and
graphs}\labelpar{ss:11} \

\vspace{2mm}

\noindent Let \((X,o)\) be a complex normal surface singularity
whose {\it link $M$ is a rational homology sphere}.  Let
\(\pi:\widetilde{X}\to X\) be a good resolution with dual graph
\(\Gamma\) whose vertices are denoted by $\cV$. Hence  \(\Gamma\)
is a tree and all the irreducible exceptional divisors have genus
\(0\). We will write  $s$, or $|\cV|$, for the number of vertices
and  $H:=H_1(M,\Z)$.

Set \(L \coloneqq H_2 ( \widetilde{X},\setZ )\). It is freely
generated by the classes of the irreducible exceptional curves
\(\{E_v\}_{v\in\cV}\). $L$ will also be identified with the group of
integral cycles supported on $E=\pi^{-1}(o)$. We set
$I_{vw}=(E_v,E_w)$. The vertex $v$ of the graph is decorated by $I_{vv}$.
The intersection matrix $I=\{I_{vw}\}$
is negative definite, and any connected plumbing graph
with negative definite intersection form appears in this way for some
singularity. The graph may also serve as the plumbing graph of the link
$M=\partial \widetilde{X}$. In this case $\widetilde{X}$ is the plumbed
4--manifold associated with $\Gamma$, and one  might consider
this topological starting setup instead of the analytic one.


If  \(L'\) denotes
\(H^2( \widetilde{X}, \setZ )\), then the intersection form
 provides an embedding \(L \hookrightarrow
L'\) with factor $H^2(\partial
\widetilde{X},\setZ)\simeq H$; $[l']$ denotes the class of $l'$.
The form $(\,,\,)$ extends to
$L'$ (since $L'\subset L\otimes \setQ$).
The module $L'$ over $\setZ$ is freely generated by the (anti-)duals $\{E_v^*\}_v$, where we
prefer the convention $ ( E_v^*, E_w) =  -1 $ for $v = w$, and
$0$ otherwise.
We write $\det(\Gamma):=\det(-I)$. The inverse of $I$ has entries
 $(I^{-1})_{vw}=(E_v^*,E^*_w)$, all of them are negative. Furthermore,
 cf. \cite[page 83 and \S 20]{EN},
 \begin{equation}\label{eq:DETsgr}
  \begin{split}
 \mbox{ $-|H|\cdot (E_v^*,E^*_w)$
equals the determinant of the subgraph obtained}\\
\mbox{from $\Gamma$ by eliminating the
shortest path connecting $v$ and $w$.}\end{split}\end{equation}
The {\it canonical class} $K\in L'$ is defined by the
{\it adjunction formulae}
\begin{equation}\label{eq:adjun}
(K+E_v,E_v)+2=0 \ \ \ \mbox{ for all $v\in\cV$.}
\end{equation}
We set $\chi(l'):=-(l',l'+K)/2$ for any $l'\in L'$. If $l\in L$ is
effective then by  Riemann-Roch theorem
$\chi(l)=h^0(\cO_l)-h^1(\cO_l)$. The integer $\chi(l')$ has
similar analytic interpretation via line bundles of
$\widetilde{X}$, cf. \cite[2.2.8]{trieste}. The expression
$K^2+|\cV|$ will appear in several formulae.  One has the following combinatorial
expression  in terms of the graph, cf. \cite{NN1}:
\begin{equation}\label{eq:K2}
K^2+|\cV|=\sum_{v\in\cV}(E_v,E_v)+3|\cV|+2+\sum_{v,w\in\cV}\, (2-\delta_v)(2-\delta_w)I^{-1}_{vw},
\end{equation}
where $\delta_v$ is the valency of the vertex $v$.

For $l_1,l_2\in L\otimes \Q$ one writes $l_1\geq l_2$ if
$l_1-l_2=\sum r_vE_v$ with all $r_v\in\setQ_{\geq 0}$. Denote by $\cS'$  the Lipman
 cone $\{l'\in L'\,:\, (l',E_v)\leq 0 \ \mbox{for all
$v$}\}$.  It is generated over $\setZ_{\geq 0}$ by the
elements $E_v^*$. Since  {\em all the entries } of $E_v^*$
are {\em strict} positive, for any fixed $a\in L'$ one has:
\begin{equation}\label{eq:finite}
\{l'\in \cS'\,:\, l'\ngeq a\} \ \ \mbox{is finite}.
\end{equation}
For any class $h\in H$ there exists a unique minimal element of
$\{l'\in L'\,:\,[l']=h\}\cap \calS'$, cf. \cite[5.4]{NOSZ}, it
will be denoted by $s_h$. Furthermore, we set \,$\square=\{\sum_v
l'_vE_v\in L'\,:\, 0\leq l'_v <1\}$ for the `semi-open cube', and
for any $h\in H=L'/L$ we consider the unique representative
$r_h\in \square$ with $[r_h]=h$. One has  $s_h\geq r_h$, and
usually $s_h\not=r_h$ (see e.g. \cite[4.5]{trieste}). Moreover,
using the generalized Laufer computation sequence of
\cite[4.3.3]{trieste} connecting $-r_h$ with $-s_h$ one gets
\begin{equation}\label{chiineq}
\chi(s_h)\leq \chi(r_h).\end{equation}

Denote by $\theta:H\to\widehat{H}$ the isomorphism $[l']\mapsto e^{2\pi i (l',\cdot )}$ of $H$
with its Pontrjagin dual $\widehat{H}$.

For more details on the resolution graphs see e.g. \cite{Nfive,NOSZ,trieste}.

\bekezd\labelpar{ss:SW}{\bf $\mathbf{Spin^c}$--structures and the Seiberg--Witten invariant of $M$.}
Let $\widetilde{\sigma}_{can}$ be the {\it canonical $spin^c$--structure} on $\widetilde{X}$; its
first Chern class $c_1( \widetilde{\sigma}_{can})$ is $-K\in L'$, cf. \cite[p.\,415]{GS}.
The set of $spin^c$--structures $\mathrm{Spin}^c(\widetilde{X})$ of $\widetilde{X}$ is an $L'$--torsor; if we denote
the $L'$--action by $l'*\widetilde{\sigma}$, then $c_1(l'*\widetilde{\sigma})=c_1(\widetilde{\sigma})+2l'$.
Furthermore,  all the $spin^c$--structures of $M$ are obtained by restrictions from $\widetilde{X}$.
$\mathrm{Spin}^c(M)$ is an $H$--torsor, compatible with the restriction and the projection $L'\to H$.
The {\it canonical $spin^c$--structure} $\sigma_{can}$ of $M$ is the restriction of $\widetilde{\sigma}_{can}$.

We denote the Seiberg--Witten invariant by $\frsw:\mathrm{Spin}^c(M)\to\setQ$, $\sigma\mapsto \frsw_\sigma$.

\subsection{Motivation: $\frsw_\sigma(M)$  as the constant term of a
`combinatorial Hilbert series'.}\labelpar{SW} \

\vspace{2mm}

\noindent Consider the  multivariable Taylor expansion
$Z(\bt)=\sum p_{l'}\bt^{l'}$ at the  origin of
\begin{equation}\label{eq:INTR}\prod_{v\in \cV} (1-\bt^{E^*_v})^{\delta_v-2},\end{equation}
where for any $l'=\sum _vl_vE_v\in L'$ we write
$\bt^{l'}=\prod_vt_v^{l_v}$ and $\delta_v$ is the valency of $v$ as above.
 This lives in $\setZ[[L']]$, the submodule of
formal power series $\setZ[[\bt^{\pm 1/|H|}]]$ in variables $\{t_v^{\pm 1/|H|}\}_v$.
\begin{theorem}\labelpar{th:JEMS} \cite{NJEMS} \ Fix some $l'\in L'$.
Assume that for any  $v\in\cV$ the  $E^*_v$--coordinate  of $l'$
is larger than or equal to $-(E_v^2+1)$. Then
\begin{equation}\label{eq:SUM}\sum_{l\in L,\, l\not\geq 0}p_{l'+l }=
-\frsw_{[-l']*\sigma_{can}}(M)-\frac{(K+2l')^2+|\cV|}{8},
\end{equation}
where
$*$ denotes the torsor action of $H$ on $\mathrm{Spin}^c(M)$.
In particular,
$$-\frsw_{[-l']*\sigma_{can}}(M)-\frac{K^2+|\cV|}{8}$$
appears as the constant term of
a `combinatorial multivariable Hilbert polynomial' (the right hand side of (\ref{eq:SUM})).
\end{theorem}

Since $Z(\bt)$ is supported on the Lipman cone, by (\ref{eq:finite}) the sum
(\ref{eq:SUM}) is finite.

Note also that the series $Z(\bt)$  decomposes in several series indexed by elements of
$H$. Indeed,  $Z(\bt)=\sum_hZ_h(\bt)$, where $Z_h(\bt)=\sum_{l':[l']=h}p_{l'}\bt^{l'}$.
The identity (\ref{eq:SUM}) involves only $Z_{[l']}$.

\vspace{2mm}

In fact, the above topological
theorem \ref{th:JEMS} was motivated by a similar theorem which
targets the analytic invariants of the singularity. In order to
have  a complete picture and possibility to interpret  the subsequent
results via analytic invariants, we recall briefly this setup as
well.

\subsection{The analytic motivation: multivariable Hilbert series of divisorial
filtrations.}\labelpar{FM} \

\vspace{2mm}

\noindent One of the strongest analytic invariants of $(X,o)$ is
its {\it equivariant divisorial Hilbert series} $\cH(\bt)$. This
is defined as follows (for more details, see  \cite{coho3,CDGb}).

Fix a resolution $\pi$ of $(X,o)$ as in (\ref{ss:11}), let $c:(Y,o)\to (X,o)$ be the
universal abelian cover of $(X,o)$ with Galois group $H=H_1(M,\Z)$,  $\pi_Y :\widetilde{Y}\to Y$ the normalized
pullback of $\pi$ by $c$, and $\widetilde{c}:\widetilde{Y}\to \widetilde{X}$ the morphism which covers $c$.
Then $\cO_{Y,o}$ inherits  the {\em divisorial multi-filtration}:
\begin{equation*}\label{eq:03}
\cF(l'):=\{ f\in \cO_{Y,o}\,|\, {\rm div}(f\circ\pi_Y)\geq \widetilde{c}^*(l')\}.
\end{equation*}
Let $\hh(l') $ be the dimension of the $\theta([l'])$--eigenspace
of $\cO_{Y,o}/\cF(l')$. Then  the {\em equivariant
divisorial Hilbert series} is
\begin{equation*}\label{eq:04}
\cH(\bt)=\sum _{l'=\sum l_vE_v\in L'}
\hh(l')t_1^{l_1}\cdots t_s^{l_s}=\sum_{l'\in L'}\hh(l')\bt^{l'}\in
\setZ[[L']].
\end{equation*}
In $\cH(\bt)$ the exponents $l'$ of the terms $\bt^{l'}$  reflect the $L'/L\simeq H$ eigenspace
decomposition too.  E.g., $\sum_{l\in L}\hh(l)\bt^{l}$ corresponds
to the $H$--invariants, hence it is the {\em Hilbert series}  of
$\cO_{X,o}$ associated with the $\pi^{-1}(o)$-divisorial
multi-filtration  (see  e.g. \cite{CHR,CDG}).

If $l'$ is in the special `Kodaira vanishing zone' $l'\in -K+\cS'$,
then by vanishing (of a certain  first cohomology), and by
Riemann-Roch, one obtains (see \cite{coho3}) that the expression
 \begin{equation}\label{eq:KV}
 \hh(l')+\frac{(K+2l')^2+|\cV|}{8}
\end{equation}
{\it depends only on the class $[l']\in L'/L$ of $l'$}. The key bridge
connecting  $\cH(\bt)$ with the
topology of the link  and with  $\Gamma$   is done by the
series (cf.  \cite{CDG,CDGEq,CDGb,coho3}):
\begin{equation*}\label{eq:06}
\cP(\bt)=-\cH(\bt) \cdot \prod_v(1-t_v^{-1})\in \setZ[[L']].
\end{equation*}
Moreover, this identity
can be   `inverted' (cf. \cite[(3.2.6)]{coho3}):
\begin{equation*} \label{eq:inv}
\hh(l')=\sum_{l\in L,\, l\not\geq 0} \bar{p}_{l'+l}, \ \
\mbox{where} \ \ \cP(\bt)=\sum_{l'}\bar{p}_{l'}\bt^{l'}.
\end{equation*}
$\cP$ is supported on
$\cS'$, cf. \cite[(3.2.2)]{coho3}, hence the sum is finite, cf. (\ref{eq:finite}).
In particular, 
\begin{equation}\label{eq:KV2}
\sum_{l\in L,\, l\not\geq 0} \bar{p}_{l'+l}=-\mathrm{const}_{[-l']} -
\frac{(K+2l')^2+|\cV|}{8}
\end{equation}
for any $l'\in-K+\cS'$, where $\mathrm{const}_{[-l']}$ depends only on the class $[-l']$ of $-l'$.
The right hand side can be interpreted as a `multivariable Hilbert polynomial' of degree 2 associated with
the series $\cH(\bt)$, or with $\cP(\bt)$. Its constant term is
the 9normalized)
equivariant geometric genera of the universal abelian cover $Y$, that  is (cf. \cite{coho3})
\begin{equation}\label{eq:KV2b}
\mathrm{ dim}(H^1(\widetilde{Y},\cO_{\widetilde{Y}}\,)_{\theta(h)})=-\mathrm{const}_{[-r_h]} -
\frac{(K+2r_h)^2+|\cV|}{8}.
\end{equation}
The point is that the {\it topological candidate} of   $\cP(\bt)$ is exactly
$Z(\bt)$ from the previous subsection; they agree
for several singularities, see e.g. \cite{CDGEq,CDGb,coho3}. The identification of their
constant terms (for `nice' analytic structures) is the subject of the
`Seiberg--Witten Invariant Conjecture', cf. \cite{NN1,Line,trieste}.
Hence, when $\calP(\bt)=Z(\bt)$,  then $\mathrm{const}_{[-l']}=
\frsw_{[-l']*\sigma_{can}}(M)$ too, and (\ref{eq:KV2b}) creates the bridge between the combinatorial/
topological Seiberg--Witten theory of the analytic counterpart.
The identity $\calP(\bt)=Z(\bt)$ is valid e.g. for splice quotient
singularities \cite{coho3}, which include all the rational singularities (when the links $M$
are $L$-spaces), minimally elliptic singularities, or weighted homogeneous singularities.

\section{Equivariant multivariable Ehrhart theory}\labelpar{ss:PPET}

\subsection{Preparatory results on Ehrhart theory}\

\vspace{2mm}

\noindent
In this section we generalize the classical Ehrhart theory to the equivariant multivariable version,
involving non-convex polytopes,
which will fit with our comparison with the equivariant multivariable series provided by plumbing graphs.

Let us start with a $d$--dimensional lattice $\calX\subset \R^d$ and a group homomorphism $\rho:\calX\to
\fH$
to a finite abelian group $\fH$. We consider a {\it rational  vector--dilated
polytope} with parameter  $\bl=(\bl_1,\ldots, \bl_r)$, $\bl_v\in \Z^{m_v}$,
\begin{equation}\label{eq:POL}P^{(\bl)}=\bigcup_{v=1}^rP^{({\bl}_v)}_v, \ \ \mbox{where} \ \
P^{({\bl}_v)}_v=\{{\bf x}\in\R^d\,:\, {\bf A}_v{\bf x}\leq
\bl_v\},\end{equation} with ${\bf A}_v\in M_{m_v,d}(\Z)$ (integral
$m_v\times d$ matrices). If $\{A_{v,\lambda i}\}_{\lambda i}$ and
$\{\bll_{v,\lambda}\}_\lambda$  are the entries of ${\bf A}_v$ and
$\bl_v$, then the inequality ${\bf A}_v{\bf x}\leq \bl_v$ in
(\ref{eq:POL}) reads as $\sum_{i=1}^dx_iA_{v,\lambda i}\leq
\bll_{v,\lambda}$ for any $\lambda=1,\ldots,m_v$.

We will vary the parameter $\bl$ in some `chambers' (described
below for the needed cases) such that the polytopes $P^{(\bl)} $
remain  combinatorial equivalent when $\bl$ runs in the same chamber.
This means that  there is a bijection
between their faces that preserves the inclusion relation. (This
implies that they are connected by homeomorphisms, which preserve
the stratification of the faces.) We also suppose that $P^{(\bl)}$
is homeomorphic to a $d$--dimensional manifold. Denote the set of
all closed facets of $P^{(\bl)}$ by $\calF$ and let $\calT$ be a
subset of $\calF$, such that $\cup_{F^{(\bl)}\in \calT}F^{(\bl)}$
is homeomorphic to a $(d-1)$--manifold. Then  we have the
following generalization to the {\it equivariant version} of
results of Stanley \cite{S74}, McMullen \cite{M78} and Beck
\cite{Beck_c,Beck_m}.
\begin{theorem}\labelpar{th:REC}
For any $h\in \fH$ and $\calT\subset \calF$ let
\begin{equation}\label{eq:REC}
\calL_h({\bf A}, \calT,\bl)
:=\mbox{cardinality of}\ \big(\big(P^{(\bl)}\setminus \cup_{F^{(\bl)}\in \calT}
F^{(\bl)}\big)\cap \rho^{-1}(h)\big).\end{equation}

(a) If $\bl$ moves in some region in such a way that
$P^{(\bl)} $ stays  combinatorially stable then the expression
$\calL_h({\bf A},\calT,\bl)$ is a quasipolynomial in $\bl\in \Z^{\sum m_v}$.

(b) For a fixed combinatorial type of $P^{(\bl)} $ and for a fixed
$\calT$, the quasipolynomials $\calL_h({\bf A},\calT,\bl)$ and
$\calL_{-h}({\bf A},\calF\setminus \calT,\bl)$ satisfy the
Ehrhart--MacDonald--Stanley reciprocity law
\begin{equation}\label{eq:EMDS}
\calL_h({\bf A},\calT,\bl) =(-1)^d\cdot \calL_{-h}({\bf
A},\calF\setminus \calT,\bl)|_{\mbox{\tiny{\rm{replace  $\bl$ by
$-\bl$}}}}.\end{equation}
\end{theorem}
\begin{proof} The statements for $\fH=0$
 are identical with those of Beck from \cite{Beck_m}.
Part (a) above can be proved identically as in \cite{Beck_m}. Or,
we notice that via standard additivity formulae, cf.
\cite[\S\,2]{Beck_m}, it is enough to prove the statement for each
convex $P_v^{(\bl_v)}$. But, considering  $P_v^{(\bl_v)}$ and $K:=
\ker(\rho)$, for any ${\bf r}\in\calX$ one has the isomorphism
$$\{{\bf x}\in K+{\bf r}\,:\, {\bf A}_v{\bf x}\leq \bl_v\}\simeq
 \{{\bf y}\in K\,:\, {\bf A}_v{\bf y}\leq \bl_v-{\bf A}_v{\bf r}\}.$$
Hence \cite[Theorem 2]{CL} can  be applied, which shows (a). Next, part (b) can also be reduced
to \cite{Beck_m}. Indeed, we can reduce the discussion again to $P_v^{(\bl_v)}$. We drop
the index $v$,  we  choose ${\bf r}_h\in \calX$ with
$\rho({\bf r}_h)=h$, and  we fix some $\bl_0$.
 Then for ${\bf x}\in K\pm {\bf r}_h $ with ${\bf A}{\bf x}\leq \bl_0$ we take
${\bf y}:={\bf x}\mp {\bf r}_y$ and ${\bf k}:=\bl_0\mp {\bf A}{\bf r}_h$, which satisfy
${\bf y}\in K$ and ${\bf A}{\bf y}\leq {\bf k}$. Therefore, using \cite{Beck_m}
for this polytope and the lattice $K$,  and the natural notations:
$$\calL_h({\bf A},\calT,\bl)|_{\bl=\bl_0}=
\calL({\bf A}{\bf y}\leq {\bf m},\,\calT,{\bf y}\in K)|_{{\bf m}={\bf k}}=
$$
$$ (-1)^d\cdot \calL({\bf A}{\bf y}\leq {\bf m},\,\calF\setminus\calT,{\bf y}\in K)|_{{\bf m}=-{\bf k}}=
 (-1)^d\cdot \calL_{-h}({\bf A},\calF\setminus \calT,\bl)|_{\bl=-\bl_0}.
$$
\end{proof}
\begin{definition}
The quasipolynomial $\calL_h({\bf A}, \calT,\bl)$  considered in Theorem \ref{th:REC}, associated with a fixed
combinatorial type of $P^{({\bf l})}$, is called the {\it equivariant multivariable quasipolynomial }
associated with the corresponding data.

If we vary ${\bf l}$ in $\Z^{\sum m_v}$ (hence we allow the variation of the combinatorial type)
we obtain the  {\it equivariant multivariable piecewise quasipolynomial } $\calL_h({\bf A}, \calT,\bl)$
(see also Theorem \ref{th:PQP} and Corollary \ref{cor:Taylor} below).
\end{definition}

\begin{remark}
Parallel to the collection $\{\calL_h\}_h$ defined in
(\ref{eq:REC}) one can consider their Fourier transforms as well:
for any character $\xi\in \widehat{\fH}=\Hom (\fH,S^1)$, one defines
\begin{equation}\label{eq:REC2}
\calL_\xi({\bf A}, \calT,\bl) :=\sum \xi^{-1}(\rho({\bf x})), \ \
\mbox{sum over} \ \ {\bf x}\in (P^{(\bl)}\setminus
\cup_{F^{(\bl)}\in \calT} F^{(\bl)}\big),\end{equation} which
satisfies $\calL_\xi=\sum_h\calL_h\cdot \xi^{-1}(h), \ \ \mbox{and
} \ \ |\fH|\cdot \calL_h=\sum_\xi\calL_\xi\cdot \xi(h)$. Hence, the
above properties of $\calL_h$ can be obtained from
similar properties of $\calL_\xi$ as well. Hence, Theorem \ref{th:REC} can be
deduced from \cite[\S\,4.3]{BV} too.
\end{remark}

\begin{remark} (a) \ In the sequel we will not consider polytopes with this
high generality: our polytopes will be special ones associated with the
denominators of type $\prod_i (1-\bt^{a_i})$ of multivariable rational functions, or their
Taylor series.
In order to avoid unnecessary technical details,
the stability of the combinatorial type of $P^{({\bf l})}$, and the corresponding chamber decomposition of
$\R^{\sum m_v}$  will also be treated for this special polytopes, see \ref{bek:combtypes}.

(b) \ To avoid any confusion regarding the expression of (\ref{eq:EMDS}) we note:
the  two quasipolynomials in (\ref{eq:EMDS}) are associated with that domain of definition
(chamber) which corresponds to the fixed combinatorial type.
Usually for $-\bl$ the combinatorial type  of $P^{(\bl)} $ is
different, hence the right hand side of (\ref{eq:EMDS}) {\it does
not equal} $(-1)^d\cdot \calL_{-h}({\bf A},\calF\setminus
\calT,-\bl)$. This last expression is the value at $-\bl$ of the
quasipolynomial associated with the chamber which contains $-\bl$.
\end{remark}

\section{Multivariable rational functions and their Periodic Constants}\labelpar{s:PC}

\subsection{Historical remark: the one--variable case
 \cite[3.9]{NO1}, \cite[4.8(1)]{Opg}.}\labelpar{PC} \

 \vspace{2mm}

\noindent Let  $S(t) = \sum_{l\geq 0} c_l t^l\in \Z[[t]]$  be a
formal power series. Suppose that for some positive integer $p$,
the expression $\sum_{l=0}^{pn-1} c_l$ is a polynomial $P_p(n)$ in
the variable $n$.  Then the constant term $P_p(0)$ of $P_p(n)$ is
independent of the `period' $p$. We call $P_p(0)$ the
\emph{periodic constant} of $S$ and denote it by $\mathrm{pc}(S)$.
For example, if $l\mapsto Q(l)$ is a quasipolynomial and
$S(t):=\sum_{l\geq 0}Q(l)t^l$, then one can take for $p$ the period
of $Q$,  and one shows that $\mathrm{pc}(\sum_{l\geq
0}Q(l)t^l)=0$.

Assume that $S(t)$ is the Hilbert series associated with a graded
algebra/vector space $A=\oplus_{l\geq 0}A_l$ (i.e. $c_l=\dim
A_l$), and the series $S$ admits a Hilbert quasipolynomial $Q(l)$
(that is, $c_l=Q(l)$ for $l\gg 0$). Since the periodic constant of
$\sum_lQ(l)t^l$ is zero, the periodic constant of $S(t)$ measures
exactly the difference between $S(t)$ and its `regularized series'
$S_{reg}(t):=\sum_{l\geq 0}Q(l)t^l$. That is:
$\mathrm{pc}(S)=(S(t)-S_{reg}(t))|_{t=1}$ collecting all the
anomalies of the starting elements of $S$.

Note that $S_{reg}(t)$ can be represented by a rational function
of negative degree with denominator of type
$A(t)=\prod_i(1-t^{a_i})$, and $S(t)-S_{reg}(t)$ is a polynomial.
Conversely,  one has the following reinterpretation of the
periodic constant \cite[7.0.2]{BN}. If $\sum_lc_lt^l$ is a
rational function $B(t)/A(t)$ with $A(t)= \prod_i(1-t^{a_i})$, and
one rewrites it as $C(t)+D(t)/A(t)$ with $C$ and $D$  polynomials
and $D(t)/A(t)$ of negative degree, then $\mathrm{pc}(S)=C(1)$.
From this fact one also gets that
$\mathrm{pc}(S(t))=\mathrm{pc}(S(t^N))$ for any $N\in \Z_{>0}$. We
will refer to $C(t)$ as the {\it polynomial part} of $S$.

As an  example, consider a subset $\calS\subset\Z_{\geq 0}$ with finite complement.  Then
 $S(t)=\sum_{s\in \calS}t^s$ rewritten is $1/(1-t)-\sum_{s\not\in \calS}t^s$, hence
 $\mathrm{pc}(S)=-\#(\Z_{\geq 0}\setminus \calS)$. In particular, if $\calS$ is the
 semigroup of a local irreducible complex plane curve singularity, then $-\mathrm{pc}(S)$ is the
 delta--invariant of that germ. Our study  below   includes the generalization of  this fact to
  surface singularities.

\subsection{The setup for the multivariable  generalization}\labelpar{ss:GENRF}\

\bekezdes\label{bek:LL'}
We wish  to extend the definition of the periodic constant to the case
  of Taylor expansions at the origin  of multivariable rational functions
of type
\begin{equation}\label{eq:func}
f(\bt)=\frac{\sum_{k=1}^r\iota_k\bt^{b_k}}{\prod_{i=1}^d
(1-\bt^{a_i})} \ \ \ \ (\iota_k\in\setZ).
\end{equation}
Let us explain the notation.
Let $L$ be a lattice of rank $s$ with fixed bases  $\{E_v\}_{v=1}^s$.
Let $L'$ be an overlattice of it with same rank, $L\subset L'\subset L\otimes \Q$ with $|L'/L|=\frdd$.
Then, in  (\ref{eq:func}),
$\{b_k\}_{k=1}^r, \ \{a_i\}_{i=1}^d\in L'$ and for any
$l'=\sum_vl'_vE_v\in L'$ we write $\bt^{l'}=t_1^{l'_1}\dots
t_{s}^{l'_{s}}$.  We also assume that
{\it all the coordinates $a_{i,v}$ of $a_i$  are strict positive}.
Hence, in general, the coefficients $l'_v$ are not integral, and the
Laurent   expansion $Tf({\bt})$ of $f({\bt})$ at the origin is
$$Tf(\bt)=\sum_{l'}p_{l'}\bt^{l'}\in \Z[[t_1^{1/\frdd},\ldots, t_s^{1/\frdd}]]
[t_1^{-1/\frdd},\ldots, t_s^{-1/\frdd}]:=\Z[[\bt^{1/\frdd}]][\bt^{-1/\frdd}].$$
We also consider the natural partial ordering of $L\otimes \Q$ (defined as in  \ref{ss:11}).
If all vectors $b_k\geq 0$ then $Tf(\bt)$ is in
$\sum_{l'}p_{l'}\bt^{l'}\in \Z[[\bt_1^{1/\frdd}]]$.
Sometimes we will not make difference between $f$ and $Tf$.

\bekezdes\label{bek:LL'2} This will be extended to the following equivariant case. We fix a
finite abelian group $G$, and for each $g\in G$ a series (or rational function)
 $Tf_g\in \Z[[\bt^{1/\frdd}]][\bt^{-1/\frdd}]$ as in
\ref{bek:LL'}, and we set
$$Tf^e(\bt):=\sum_{g\in G}\, Tf_g(\bt)\cdot [g]\in \Z[[\bt^{1/\frdd}]][\bt^{-1/\frdd}][G].$$
Sometimes this equivariant extension is given automatically in the context of \ref{bek:LL'}.
Indeed, if  in \ref{bek:LL'} we set  $H:=L'/L$, and for
\begin{equation}\label{eq:fh}
Tf=\sum_{l'}p_{l'}\bt^{l'} \ \ \ \mbox{we define} \ \ \
Tf_h:=\sum_{[l']=h}p_{l'}\bt^{l'},\end{equation}
we obtain a decomposition of $Tf$ as a sum $\sum_hTf_h\in \Z[[\bt^{1/\frdd}]][\bt^{-1/\frdd}][H]$
(with $\frdd=|H|$).

In our cases we always start  with this  group $L'/L=H$ (hence
$f$ determines its decomposition $\sum_hf_h$). Nevertheless, some alterations will appear.
First, we might consider the nonequivariant case, hence we can forget the decomposition over $H$.
Another case appears as follows.  In order to simplify the rational function we will eliminate
some of its  variables (e.g., we substitute $t_i=1$ for certain indices $i$),
or we restrict $f$ to a linear subspace $V$. Then,
after this substitution,  the restricted function
$f|_{t_i=1}$ will not determine anymore the restrictions $(f_h)|_{t_i=1}$  of the `old' components $f_h$.
That is, the new pair of lattices $(L_V,L'_V)=(L\cap V,L'\cap V)$
and the `old group' $H=L'/L$ become rather independent.  In such cases we will
 keep the old group $H=L'/L$ (and the `old' decomposition  $f_h$)
without asking any compatibility with $L'_V/L_V$.

\bekezdes\label{bek:LL'3} Since all the coordinates $a_{i,v}$ of $a_i$ are strict positive,
for any $Tf(\bt)=\sum_{l'}p_{l'}\bt^{l'}$  we get a well defined  {\it counting function} of the coefficients,
$$l'\mapsto Q(l'):=\sum_{l''\not\geq l'} \, p_{l''}.$$
If $Tf=\sum_hTf_h$, then each $Tf_h$ determines a counting function  $Q_h$ defined in the same way.

If $H=L'/L$ and $Tf$ decomposes into $\sum_hTf_h$ under the law from (\ref{eq:fh}), then
\begin{equation}\label{eq:PCDEFa}
\sum_{l''\not\geq l'} \, p_{l''}\cdot [l'']=\sum _{h\in H}\,Q_h(l')[h].
\end{equation}
The definitions are motivated by formulae (\ref{eq:SUM}) and (\ref{eq:KV2}).
The functions $Q_h(l')$ will be studied in the next subsections via Ehrhart theory.

\subsection{Ehrhart quasipolynomials associated with denominators of rational functions}\labelpar{ss:EP} \

\vspace{2mm}

\noindent First we consider the
case $d>0$, the special  case  $d=0$ will  be treated  in
\ref{ss:d0}.

\bekezdes\labelpar{bek:pol} {\bf The polytope associated with $\{a_i\}_{i=1}^d$.}\label{constr:pol}
In order to run the Ehrhart theory we have first to fix the lattice $\calX$ and the representation
$\rho:\calX\to \fH$, cf. section \ref{ss:PPET}.
First, we set  $\calX=\Z^d$ and  $\alpha:\calX\to L'$ given by
$\alpha({\bf x})=\sum_{i=1}^dx_ia_i\in L'$. For $(\fH,\rho)$ there are two possibilities:

(a) $\fH=H=L'/L$ and $\rho$ is the composition
 $\calX\stackrel{\alpha}{\longrightarrow} L'\to L'/L$.

(b) $\fH=0$ and $\rho=0$.

This choice has an effect on the equivariant decomposition $f^e=\sum_gf_g[g]$ of $f$ too. In case (a)
usually we have  $G=H$ and the decomposition is given by \ref{eq:fh}. In case (b) we can take either
$G=0$ (this can happen e.g. when we forget the decomposition in case (a), and
we sum up all the components),
or we can take any $G$ (by specifying each $f_g$). In this latter case each fixed $f_g$
behaves like a function in the nonequivariant case $G=0$, hence can be treated in the same way.

Since the case (b) follows from case (a) (by forgetting the extra information from $\fH$), in the sequel
we treat the case (a), hence $\fH=G=L'/L$.

 Consider the
matrix ${\bf A}$ with column vectors $|H|a_i$ and write ${\bf A}_v$ for its rows. Then the construction of subsection \ref{ss:PPET}
can be repeated (eventually completing each ${\bf A}_v$ to assure the inequalities $x_i\geq 0$ as well).
For  $l\in \sum_vl_vE_v\in L$ consider
\begin{equation}\label{eq:Pv}
P_v^{\triangleleft} :=\{{\bf x}\in (\R_{\geq 0})^d\,:\, |H|\cdot
\sum_ix_ia_{i,v} < l_v\} \ \ \mbox{and} \ \
P^\triangleleft :=\bigcup_{v=1}^sP_v^\triangleleft.\end{equation}  The closure $P_v$ of $P_v^\triangleleft$
 is a dilated
convex (simplicial)  polytope depending on the one-dimensional parameter $l_v$.
Moreover,  $P^\triangleleft$
is described  via the partial ordering  of $L\otimes\Q$ as the set
$\sum_ix_ia_i\not\geq l/|H|$. Since $L'\subset L/|H|$, we can
restrict ourself to the lattice $L'$ (preserving all the general
results of section \ref{ss:PPET}). Hence for any $l'\in L'$ we
set
\begin{equation}\label{eq:PL}
P^{(l'),\triangleleft }:=\{{\bf x}\in (\R_{\geq 0})^d\,:\, \sum_ix_ia_i\not\geq l'\}, \ \ \
P^{(l')}=\mbox{closure of }(
P^{(l'),\triangleleft }).\end{equation}
The combinatorial type of $P^{(l')}$ might vary with $l'$. Nevertheless, by definition, the facets will be
grouped for all different combinatorial types  by the same principle: we consider the
coordinate facets $F_i:=P^{(l')}\cap \{x_i=0\}$, $1\leq i\leq d$, and
we denote by $\calT$ the collection of all other facets. Hence $P^{(l'),\triangleleft }=
P^{(l')}\setminus  \cup_{F^{(l')}\in \calT}F^{(l')}$. The construction is
motivated by the summation from (\ref{eq:SUM})
(although in the general statements the choice of $\calT$ is irrelevant).

 Then \ref{th:JEMS} and \ref{PC} lead to the next counting function defined in the
 group ring $\Z[H]$ of $H$:
\begin{equation}\label{eq:calP}
\calL^e({\bf A},\calT,l'):=\sum_{h\in H} \calL_h({\bf A},\calT,l')\cdot [h]
:=\sum 1\cdot [l'']\in \Z[H],
\end{equation}
where the last sum runs over $l''\in \big(P^{(l')}\setminus \cup_{F^{(l')}\in \calT}
F^{(l')}\big)\cap L'=P^{(l'),\triangleleft}\cap L'$.

The corresponding  nonequivariant counting function, corresponding to $G=0$ is denoted by
$$ \calL_{ne}({\bf A},\calT,l'):=\sum_{h\in H} \calL_h({\bf A},\calT,l')\in \Z.$$

Similarly, we set $\calL^e({\bf A},\calF\setminus \calT,l')$ too. For both of them
Theorem \ref{th:REC} applies.

By the very construction, we have the following identity. Consider the
 equivariant Taylor expansion at the origin of the function determined by the {\it denominator of $f$},  namely
\begin{equation}\label{eq:Taylorden}
\widetilde{f}^e(\bt)=\frac{1}{\prod_{i=1}^d (1-[a_i]\bt^{a_i})}=
\sum_{l''}\widetilde{p}_{l''}\bt^{l''}\cdot [l'']\in \Z[[\bt^{1/|H|}]]\, [H].
\end{equation}
Note that since all the $\{E_v\}$--coefficients of each $a_i$ are strict positive, for any
$l'\in L'$ the   set $\{l''\,:\, \widetilde{p}_{l''}\not=0, l''\not\geq l'\}$ is finite. Then, by the above construction,
\begin{equation}\label{eq:PCDEFaden}
\sum_{l''\not\geq l'} \, \widetilde{p}_{l''}\cdot [l'']=\calL^e({\bf A},\calT,l').
\end{equation}

\bekezdes\labelpar{bek:combtypes}{\bf Combinatorial types,
chambers.}  Next, we wish to make precise the {\it
combinatorial stability} condition. The result of Sturmfels
\cite{Str}, Brion--Vergne \cite{BV}, Clauss--Loechner \cite{CL} and
Szenes--Vergne \cite{SZV} implies that $\calL^e$ from (\ref{eq:PCDEFaden})
(that is, each $\calL_h$)
is a {\it piecewise
quasipolynomial} on $L'$: the parameter space $L\otimes \R$
decomposes into several chambers, the restriction of $\calL^e$ on
each chamber is a quasipolynomial, and $\calL^e$ is continuous. The chambers
are described as follows.

Notice that the combinatorial type of $P^{(l')}$ in (\ref{eq:PL})
vary in the same way as the closure of its {\it convex} complement in $\R_{\geq
0}^d$, namely
\begin{equation}\label{eq:PCOMP}
\{ {\bf x}\in (\R_{\geq 0})^d\,:\, \sum_ix_ia_i\geq l'\},
\end{equation}
since both are  determined by their common boundary $\calT$. The
inequalities  of (\ref{eq:PCOMP}) can be viewed as a {\it vector
partition}  $\sum_ix_ia_i+ \sum_v y_v(-E_v)=l'$, with $x_i\geq 0$
and $y_v\geq 0$.  Hence, according to the above references, we
have the following chamber decomposition of $L\otimes \R$.

Let ${\bf M}$  be the matrix with column vectors $\{a_i\}_{i=1}^d$
and $\{-E_v\}_{v=1}^s$. A subset $\sigma$ of indices of columns is
called {\it basis} if the corresponding columns form a basis of
$L\otimes \R$; in this case we write   $Cone({\bf M}_{\sigma})$
for the positive closed cone generated by them.  Then the chamber
decomposition is the polyhedral subdivision of $L\otimes \R$
provided by the common refinement of the cones $Cone({\bf
M}_{\sigma})$, where $\sigma$ runs all over the basis. {\it A chamber
is a closed cone of the subdivision whose interior is non--empty.}
Usually we denote them by $\calC$, let their
index set (collection) be $\mathfrak{C}$.

We will need the associated {\it disjoint} decomposition of $L\otimes
\R$ with relative open cones as well. A typical element  of this disjoint decomposition
is the {\it
relative interior of an intersection of type $\cap_{\calC\in
\mathfrak{C}'}\calC$}, where $\mathfrak{C}'$ runs over the subsets
of $\mathfrak{C}$. For these cones we use the notation
$\calC_{op}$.

Each chamber $\calC$ determines an open cone, namely its interior.
And, conversely, each top dimensional open cone determines a
chamber $\calC$, namely its closure.

 The next theorem is the
direct consequence of \cite[4.4]{BV}, \cite[0.2]{SZV} and
(\ref{th:REC}) using the additivity of the Ehrhart quasipolynomial
on the suitable convex parts of $P^{(l')}$. (We state it  for our
specific facet--collection $\calT$, the case which will be used
later, but it is true for any other facet--decomposition of the
boundary whenever  $\cup_{F^{(l')}\in \calT}F^{(l')}$ is
homeomorphic to a $(d-1)$--manifold.)
\begin{theorem}\labelpar{th:PQP}  (a) For each relative open cone  $\calC_{op}$ of $L\otimes \R$,
 $P^{(l')}$ is combinatorially stable, that is, the polytopes
  $\{P^{(l')}\}_{l'\in \calC_{op}}$ have the same combinatorial type.
 Therefore, for any fixed $h\in H$, the restrictions
 $\calL^{\calC_{op}}_h({\bf A},\calT)$ and  $\calL^{\calC_{op}}_h({\bf A},\calF\setminus \calT)$
to $\calC_{op}$  of
  $\calL_h({\bf A},\calT)$ and  $\calL_h({\bf A},\calF\setminus \calT)$
 respectively  are quasipolynomials.

(b) These quasipolynomials have a continuous extension to the
closure of $\calC_{op}$. Namely, if $\calC_{op}'$ is in the
closure of $\calC_{op}$, then
 $\calL^{\calC_{op}'}_h({\bf A},\calT)$ is the restriction to
 $\calC_{op}'$ of the (abstract) quasipolynomial  $\calL^{\calC_{op}}_h({\bf
 A},\calT)$. (Similarly for $\calL^{\calC_{op}}_h({\bf A},\calF\setminus
 \calT)$.)

In particular, for any chamber $\calC$ one has a well defined
quasipolynomial  $\calL^{\calC}_h({\bf A},\calT)$, defined as
$\calL^{\calC_{op}}_h({\bf A},\calT)$, where $\calC_{op}$ is the
interior of $\calC$, which equals $\calL_h({\bf A},\calT)$ for
all points of $\calC$.

This also shows that for any two chambers $\calC_1$ and $\calC_2$
one has the continuity property
\begin{equation}\label{eq:CONTINa}
\calL_h^{{\calC_1}}({\bf A},\calT)|_{\calC_1\cap \calC_2}=
\calL_h^{{\calC_2}}({\bf A},\calT)|_{\calC_1\cap
\calC_2}.\end{equation}

(c) $\calL^{\calC}_h({\bf A},\calT)$ and $\calL^{\calC}_{-h}({\bf A},\calF\setminus \calT)$,
as abstract quasipolynomials associated with a fixed chamber $\calC$,   satisfy the reciprocity
 $$\calL^{\calC}_h({\bf A},\calT,l')=(-1)^d\cdot\calL^{\calC}_{-h}({\bf A},\calF\setminus \calT,-l').$$

\end{theorem}

\noindent
We have the following consequences regarding the counting function $l'\mapsto Q_h(l')$ of $f^e(\bt)$
defined in (\ref{eq:PCDEFa}):

\begin{corollary}\labelpar{cor:Taylor}
(a)  $Q_h$ is a piecewise quasipolynomial.
Indeed, for any $h\in H$ and $l'\in L'$
\begin{equation}\label{eq:IND}
Q_h(l')=\sum_k\iota_k\cdot\calL_{h-[b_k]}({\bf A},\calT,l'-b_k).
\end{equation}   
In particular, the right hand side of (\ref{eq:IND}) is independent
of the representation of $f$ as in (\ref{eq:func}) (that is, of the
choice of $\{b_k,\,a_i\}_{k,i}$), it depends only on the rational function $f$.

(b) Fix a chamber $\calC$  of $L\otimes \R$, cf. \ref{th:PQP}, and for any $h\in H$ define
the quasipolynomial
\begin{equation}\label{eq:qbar}
\overline{Q}_h^\calC(l'):=
\sum_k\iota_k\cdot\calL^\calC_{h-[b_k]}({\bf A},\calT,l'-b_k).
\end{equation}
Then  the restriction of
$Q_h(l')$ to $\cap_k(b_k+\calC)$ is a quasipolynomial, namely
\begin{equation}\label{ex:qbar2}
Q_h(l')=\overline{Q}_h^\calC(l') \ \mbox{ on \
$\bigcap_k(b_k+\calC)$.}\end{equation} Moreover,
there exists $l'_*\in\calC$ such that $l'_*+\calC\subset \cap_k(b_k+\calC)$.

 [Warning:
$\calL^\calC_{h-[b_k]}({\bf A},\calT,l'-b_k)\not=\calL_{h-[b_k]}({\bf A},\calT,l'-b_k)$
unless $l'-[b_k]\in \calC$.]

(c) For any fixed $h\in H$,
the quasipolynomial  $\overline{Q}_h^\calC(l')$ satisfies  the following property:
for any $l'\in L'$ with $[l']=h$, and any $q\in\square$ (the semi-open unit cube), one has
\begin{equation}\label{eq:cccc}\overline{Q}_h^\calC(l'-q)=\overline{Q}_h^\calC(l').\end{equation}
In particular, by taking $l'=q=r_h$:
\begin{equation}\label{eq:cccc2}\overline{Q}_h^\calC(r_h)=\overline{Q}_h^\calC(0).\end{equation}
\end{corollary}
\begin{proof}
For (a) use (\ref{eq:PL}) and the fact that $b_k+\sum x_ia_i\not\geq l'$ if and only
if $\sum x_ia_i\not\geq l'-b_k$. Since the coefficients of the
Taylor expansion depend only on $f$, the second sentence  follows too.

For (b) use part (a)  and
the fact that $\calC\cap(\cap_k(b_k+\calC))$ contains a set of type $l_*'+\calC$.

(c) Consider those values $l'$ in some $l'_*+\calC$ for which  all elements of type
$l'-b_k$ and $l'-q-b_k$ are in $\calC$. For these values $l'$, (\ref{eq:cccc})
follows from the identity
 $P^{(l'),\triangleleft}\cap\rho^{-1}(h)= P^{(l'-q),\triangleleft}\cap\rho^{-1}(h)$
whenever $[l']=h$.
This is true since  for any $l''$ with
$[l'']=[l']$, $l''\geq l'$ is equivalent with $l''\geq l'-q$. Indeed, taking $y=l''-l'$,
this reads as follows: for any $y\in L$, $y\geq 0$ if and only if $y\geq -q$.

Now, if two quasipolynomials agree on $l'_0+\calC$ then they are equal.
\end{proof}

\begin{remark}\labelpar{re:Szenes} Thanks to  \cite[Theorem 0.2]{SZV},
the continuity property \ref{eq:CONTINa} has the following
extension (coincidence of the quasipolynomials on neighboring strips).
Set $\Box({\bf A}):=\sum_i [0,1)a_i$. Then for any two chambers $\calC_1$ and $\calC_2$,
and  $S:=(-\Box({\bf A})+\calC_1) \cap (-\Box({\bf
A})+\calC_2)$
\begin{equation}\label{eq:CONTINb}
\calL_h^{{\calC_1}}({\bf A},\calT)|_{S}= \calL_h^{{\calC_2}}({\bf
A},\calT)|_{S}.\end{equation}
\end{remark}

\bekezdes\labelpar{ss:d0} {\bf The $d=0$ case.} All the above
properties can be extended for $d=0$ as well. Although the
polytope constructed in \ref{eq:PL} does not exist,  we can
look at the polynomial  $f(\bt)=\sum_k \iota_k
\bt^{b_k}$ itself. Then using notation of (\ref{eq:PCDEFa}) we set
$$\sum _{h\in H}\,Q_h(l')[h]=
\sum_{l''\not\geq l'} \, p_{l''}\cdot [l'']=\sum_{\{k \,:\, b_k\not\geq l'\}}\iota_k [b_k].$$
Moreover, we have  the chamber decomposition of $L\otimes\R$ defined by
$\{-E_v\}_{v=1}^s$ via the same principle as above. This means two chambers:
 $\calC_0:=\R_{\geq 0}\langle -E_v
\rangle$ and $\calC_1$,  the closure of the  complement of $\calC_0$ in $\R^s$.
Then $Q_h(l')=\sum_{\{k \,:\, [b_k]=h\}}\iota_k$ on
$\cap_k(b_k+\calC_1)$ and $0$ on $\cap_k(b_k+\calC_0)$.

\subsection{The definition of the multivariable equivariant periodic constant of a rational function}\

\vspace{2mm}

\noindent We consider the situation of \ref{bek:LL'} and \ref{bek:pol}(a).
For each $h\in H$ define  $r_h\in L'$ as in \ref{ss:11}.
\begin{definition}\labelpar{def:pc}
Let  $\calK\subset L'\otimes \R$ be a
closed real cone whose affine closure aff$(\calK)$ has positive dimension.
For any $h\in H$ we assume
that there exist

$\bullet$ $l'_*\in \calK$

$\bullet$ a sublattice $\widetilde{L}\subset L$ of finite index, and

$\bullet$ a quasipolynomial $l'\mapsto \widetilde{Q}_h(l')$,
defined on  $\widetilde{L}\cap{\mathrm{aff}}(\calK)$
such that
\begin{equation}\label{eq:PCDEFx}
Q_h(l')=\widetilde{Q}_h(l') \ \ \mbox{
for any $\widetilde{L}\cap(l'_*+\calK)$.}
\end{equation}
Then we define the {\it equivariant periodic constant} of $f$ associated with   $\calK$  by
\begin{equation}\label{eq:PCDEF}\mathrm{pc}^{e,\calK}(f)=\sum_{h\in H}\,
\mathrm{pc}^{\calK}_h(f)\cdot [h]:= \sum_{h\in H}\,
\widetilde{Q}_h(0)\cdot [h]\in \Z[H],\end{equation} and  we say
that {\it $f$ admits a periodic constant in $\calK$}. (Sometimes
we will use the same notation for the real  cone $\calK$ and for
its lattice points  $\calK\cap L'$ in $L'$.)
\end{definition}
\begin{remark}
The above definition is independent of the choice of the
sublattice $\widetilde{L}$: it can be replaced by any
sublattice  of finite index. The advantage of such sublattices is
that convenient restrictions of $Q_h$ might  have nicer forms which are easier to
compute. The choice of  $\widetilde{L}$ corresponds to
the choice of $p$ in \ref{PC}, and it is responsible for the name
`periodic' in the name of $\mathrm{pc}^{e,\calK}(f)$.
\end{remark}
\begin{proposition}\labelpar{prop:pc}
(a)  Consider the chamber decomposition of $L\otimes \R$
given by the denominator $\prod_i(1-\bt^{a_i})$ of $f$
as in Theorem~\ref{th:PQP}.
Then $f$ admits  a periodic constant in each chamber $\calC$ and
\begin{equation}\label{eq:PCC}
{\mathrm{pc}}_h^\calC(f)\, = \,
\overline{Q}_h^\calC(r_h)=\overline{Q}_h^\calC(0).
\end{equation}
(b) If two functions $f_1$ and $f_2$ admit periodic constant in  some cone  $\calK$,
then the same is true for  $\alpha_1f_1+\alpha_2f_2$  and
$$\mathrm{pc}^\calK (\alpha_1f_1+\alpha_2f_2)=\alpha_1
\mathrm{pc}^\calK (f_1)+\alpha_2\mathrm{pc}^\calK (f_2) \ \ \ \ \
(\alpha_1,\,\alpha_2\in \C).$$ (c) If $f$ admits periodic
constants in two (top dimensional) cones $\calK_1$ and $\calK_2$,
and the interior $int(\calK_1\cap\calK_2)$ of the intersection
$\calK_1\cap\calK_2$ is non-empty, then $\mathrm{pc}^{\calK_1}(f)=
\mathrm{pc}^{\calK_2}(f)$.

In particular, if $\{C_i\}_{i=1,2}$ are two chambers as in (a), and $f$ admits a periodic constant
in $\calK$,  and $int(\calC_i\cap\calK)\not=\emptyset$ ($i=1,2$), then
$\mathrm{pc}^{\calC_1}(f)= \mathrm{pc}^{\calC_2}(f)$.
\end{proposition}
\begin{proof}
For (a) use Corollary \ref{cor:Taylor};   (b) is clear. For (c) we
can assume that $\calK_2\subset \calK_1$ (by considering $\calK_i$
and $\calK_1\cap\calK_2$). Then if $Q_h$ is quasipolynomial on
$l'_1+\calK_1$  (with $l'_1\in\calK_1$), then $(l'_1+\calK_2)\cap
\calK_2$ contains a set of type $l_2'+\calK_2$ with
$l'_2\in\calK_2$, on which one can take the restriction of the
previous quasipolynomial.
\end{proof}

\begin{remark}\labelpar{rem:fh}
  Note that in the rational presentation of $f$ we
might assume that  $a_i \in L$ for all $i$. Indeed, take  $o_i\in
\Z_{>0}$ such that $o_ia_i\in L$, and amplify the fraction by
 $\prod_i(1-\bt^{o_ia_i})/(1-\bt^{a_i})$. Therefore, for each $h$ we can
write $f_h(\bt)$ in the form
$$f_h(\bt)=\bt^{r_h}\sum_k \iota_k \cdot\frac{\bt^{\overline{b}_k}}{\prod_i(1-\bt^{a_i})},$$
where $a_i,\ \overline{b}_k \in L$, hence $f_h(\bt)/\bt^{r_h}\in \Z[[\bt]][\bt^{-1}]$.
Then if we consider the
nonequivariant periodic constant ${\rm pc}^\calC$ of $f_h(\bt)/\bt^{r_h}$,
\ref{eq:PCDEFa}, \ref{ex:qbar2} and \ref{eq:PCC} imply that
${\rm pc}_h^\calC(f(\bt))={\rm pc}^\calC(f_h(\bt)/\bt^{r_h})$ for all chambers
$\calC$ associated with $\{a_i\}_i$.
\end{remark}

 \begin{example}\labelpar{ex:1} \
Assume that $L=L'=\Z$ and $\calK=\R_{\geq 0}$, and consider
$S(t)$ as in \ref{PC}. If $S(t)$ admits  a periodic constant in
$\calK$, then the definition of $\mathrm{pc}(S)$ from
\ref{def:pc} is compatible with the statements  from \ref{PC}.
\end{example}

\begin{example}\labelpar{ex:2}
(a)  (The $d=0$ case) \ Assume that
$f(\bt)=\sum_{k=1}^r\iota_k\bt^{b_k}$. Then, using \ref{ss:d0}
(and its notation), $\mathrm{pc}^{e,\calC_0}(f)=0$ and
$\mathrm{pc}^{e,\calC_1}(f)=\sum_{k=1}^r\iota_k[b_k]\in\Z[H]$.

(b) Assume that the rank is $s=2$ and $f(\bt)=\bt^b/(1-\bt^a)$,
with both the entries $(a_1,a_2)$ of $a$ positive.
We assume that $a\in L$ while $b\in L'$. Again, for $h\not=[b]$ the counting function,
hence its periodic constant too, is zero. Assume $h=[b]$, and write
$b=(b_1,b_2)$. Then the denominator provides three chambers:
$\calC_0:=\Z_{\geq  0}\langle -E_1,-E_2\rangle$,
$\calC_1:=\Z_{\geq  0}\langle a,-E_2\rangle$, $\calC_2:=\Z_{\geq
0}\langle a,-E_1\rangle$. Then the three quasipolynomials for
$1/(1-\bt^a)$ are $\calL^{\calC_0}_h=0$, \
$\calL^{\calC_i}_h(n_1,n_2)=\lceil n_i/ a_i\rceil$; hence
 ${\mathrm{pc}}^{\calC_0}_h(f)=0$, \
${\mathrm{pc}}^{\calC_i}_h(f)=\lceil -b_i/a_i\rceil$  ($i=1,2$).
In particular, ${\mathrm{pc}}^{\calC}_h(f)$, in general, depends on
the choice of $\calC$.

(c) Assume that $L=L'$ and
$f(t)=\frac{t_1^{b_1}t_2^{b_2}}{(1-t_1t_2)(1-t_1^2t_2)}$.  Then
the  chambers associated with the denominator are:
$\calC_0:=\R_{\geq 0}\langle -E_1,-E_2\rangle$, $\calC_2:=\R_{\geq
0}\langle -E_1,(1,1)\rangle$, $\calC:=\R_{\geq 0}\langle
(1,1),(2,1)\rangle$ and $\calC_1:=\R_{\geq 0}\langle
(2,1),-E_2\rangle$. Then, by a computation,
\begin{equation}
 \begin{array}[l]{ll}
    \calL^{\calC_0}=0; & \calL^{\calC_2}(l_1,l_2)=\frac{l_2^2}{2}+\frac{l_2}{2};\\
    \calL^{\calC}(l_1,l_2)=\frac{l_1^2}{2}+l_2^2+\frac{l_1}{2}-l_1
    l_2; &
    \calL^{\calC_1}(l_1,l_2)=\frac{l_1^2}{4}+\frac{l_1}{2}+\frac{1+(-1)^{l_1+1}}{8}.\\
 \end{array}
\end{equation}
Hence, by Proposition \ref{prop:pc} and (\ref{eq:qbar}), one has
$\mathrm{pc}^{\calC_*}(f)= \calL^{\calC_*}(-b_1,-b_2)$.
\end{example}

\begin{example}\labelpar{ex:NAM} {\bf Normal affine monoids.} \
Consider the following objects (cf.  \ref{bek:LL'}): a lattice $L$
with fixed bases $\{E_v\}_{v=1}^d$ (hence $s=d$) and with induced
partial ordering $\leq$, $L'\subset L\otimes \Q$ an overlattice
with finite abelian quotient $H:=L'/L$ and projection $\rho:L'\to
H$. Furthermore,  let  $\{a_i\}_{i=1}^d$ be linearly independent
vectors in $L'$ with all their $\{E_v\}$--coordinates positive.
Let $\calK$ be the positive real cone generated by the vectors
$\{a_i\}_i$, and consider the Hilbert series of $\calK$
$$f(\bt ):=\sum_{l'\in \calK\cap L'}\bt^{l'}.$$
Since $\calK$ depends only on the rays generated by the vectors
$a_i$, we can assume that $a_i\in L$ for all $i$.

Set  $\Box({\bf A})= \sum_{i=1}^d [0,1) a_i$ as above, and
consider the monoid $M:=\Z_{\geq 0}\langle a_i \rangle$ (cf.  e.g.
\cite[2.C]{BG}). Then the normal affine monoid
$\calK\cap L'$ is a module over  $M$ and if we set $B:=\Box({\bf
A})\cap L'$, \cite[Prop. 2.43]{BG} implies that
$$\calK\cap L'=\bigsqcup_{b\in B}b+M.$$
In particular, $f(\bt)$ equals $\sum_{b\in B}\bt^b/\prod_{i=1}^d (1-\bt^{a_i})$ and has the
form considered in \ref{ss:GENRF}.

If the rank $d$ is $\geq 3$ then $\calK$ usually is cut
in more chambers. Indeed, take e.g. $d=3$, $a_i=(1,1,1)+E_i$ for $i=1,2,3$. Then $\calK$ is cut in its
 barycentric subdivision.  Nevertheless,  if $d= 2$ then $\calK$ consists of  a unique chamber and $f$ admits a periodic
 constant in $\calK$. Indeed, one has:
\begin{lemma}\labelpar{lem:d2}
 If $d=2$ then ${\mathrm pc}^{\calK}_h(f)=0$ \,for all $h \in H$.
\end{lemma}
\begin{proof}
It is elementary to see that $\calK$ is one of the chambers (use the
construction from  \ref{bek:combtypes}). Take
$B=\{b_k\}_k$, and write $f=\sum_kf_k$, where
$f_k=\bt^{b_k}/(1-\bt^{a_1})(1-\bt^{a_2})$. The only relevant
classes $h\in H$ are given by $\{[b_i]: b_i\in B\}$, otherwise
already the Ehrhart quasipolynomials are zero (since $a_i\in L$). Fix such a class
$h=[b_i]$. Let $\calL^{\calK}_{h}(\calT)$ be the quasipolynomial
associated with the chamber $\calK$ and the denominator of $f$.
Then, by (\ref{eq:PCC}) and (\ref{eq:qbar}),
$\mathrm{pc}_h^{\calK}(f_k)=\calL^{\calK}_{[b_i-b_k]}(\calT,-b_k)$.
This, by the Reciprocity   Law \ref{th:PQP}(c) equals
$\calL^{\calK}_{[b_k-b_i]}(\calF\setminus \calT,b_k)$. Again,
since the denominator is a series in $L$, for $[b_k-b_i]\not=0$
the series is zero; so we may assume $[b_k-b_i]=0$.  But, since
$b_k\in \calK$, the value $\calL^{\calK}_0(\calF\setminus
\calT,b_k)$  of the quasipolynomial carries its geometric
meaning, it is the cardinality of the set  $\{m=n_1a_1+n_2a_2\,:\,
n_1>0,\, n_2>0,\, m\not>b_k\}$. But since for any such $m$ one has
$m\geq a_1+a_2>b_k$, contradicting $m\not>b_k$, this set is empty.
\end{proof}
\end{example}

\begin{example}\labelpar{ex:AM} {\bf General  affine monoids of rank $d=2$.} \
 Consider the situation of Example \ref{ex:NAM} with $d=2$, and let $N$ be a submonoid of
$\widehat{N}=\calK\cap L'$ of rank 2, and we also assume that
$\widehat{N}$ is the normalization of $N$. Set $$f(\bt
):=\sum_{l'\in N}\bt^{l'}.$$ Then $f(\bt)$ is again of type
(\ref{eq:func}). Indeed, by \cite[Prop. 2.35]{BG},
$\widehat{N}\setminus N$ is a union of finite family of sets of
type (I) $b\in \widehat{N}$, or (II) $b+\Z \ell a_i$, where $b\in
\widehat{N}$, $\ell\in\Z_{\geq 0}$, $i=1$ or 2.  Obviously, two sets
of type (II) with different $i$-values might have an intersection
point of type (I). In particular,
$$f(\bt)=\sum_{l'\in \widehat{N}}\bt^{l'}-\sum_i\frac{\bt^{b_{i,1}}}{1-\bt^{k_{i,1}a_1}}-
\sum_j\frac{\bt^{b_{j,2}}}{1-\bt^{k_{j,2}a_2}}+\sum_k(\pm
\bt^{b_k}).$$ Note that the periodic constant of the first sum is
zero by Lemma \ref{lem:d2}, and the others can easily be computed
(even with closed formulae) via Example \ref{ex:2}, parts (a) and
(b).

The computation shows that the periodic constant carries
information about the failure of normality of $N$ (compare with
the delta-invariant computation from the end of \ref{PC}).

The situation is similar when we consider a {\it semigroup}
of $\widehat{N}$, that is, when we eliminate the
neutral element of the above $N$  (or, when we consider a module over the
submonoid $N\subset \widehat{N}$).
\end{example}

\begin{example} {\bf Reduction of variables.}
The next statement is an example when the number of variables of the function $f$ can be reduced
in the procedure of the periodic constant computation. (For another reduction result, see Theorem \ref{th:REST}.)
 For simplicity we assume  $L'=L$.\end{example}
\begin{proposition}\labelpar{prop:pc2}
 Let $f(\bt)=\frac{\bt^{b}}{\prod_{i=1}^d (1-\bt^{a_i})}$ and
 assume that $b=\sum_{v=1}^s b_v E_v \in \calC$, where
 $\calC$ is  a chamber associated with the denominator.

 We consider  the
subset $Pos:=\{v \, : \, b_v>0\}$ with
cardinality $p$, and the projection $\R^s\to \R^p$,
defined by $(r_v)_{v=1}^s\mapsto (r_v)_{v\in Pos}$ and denoted by
$v\mapsto v^\red $. Accordingly,  we set a new function
$f^\red(\bz):=\frac{\bz^{b^\red }} {\prod_{i=1}^d (1-\bz^{a_i^\red })}$ in
$p$ variables, and a new chamber $\calC^{\red} :=\R_{\geq 0}\langle
\{w_j^\red \}_j\rangle$, where $w_j$ are the generators of
\,$\calC=\R_{\geq 0}\langle \{w_j\}_j\rangle$. Then
$\mathrm{pc}^{\calC}(f)=\mathrm{pc}^{\calC^\red }(f^\red )$.
\end{proposition}
\begin{proof}
This is a direct application of  Theorem \ref{th:REC}(b).
 Indeed, by the Ehrhart--MacDonald--Stanley
reciprocity law,
 we get $\mathrm{pc}^{\calC}(f)=\calL^{\calC}({\bf A},
\calT,-b)=(-1)^d\cdot \calL^{\calC}({\bf A},\calF\setminus
\calT,b)$. Since $b \in \calC$, by the very definition of
$\calL^{\calC}({\bf A},\calF\setminus \calT)$,
this (modulo the sign) equals
the number of integral points of $P^{(b)}\setminus
\cup_{F^{(b)}\in \calF\setminus \calT} F^{(b)} \subset \R^d$.
But, if  $v\notin Pos$, or $b_v\leq 0$, then in (\ref{eq:POL}) $P_v^{(b_v)}$  has only
non-positive integral points. Therefore we can
omit these polytopes without affecting the periodic constant. Then, this fact and
$b^\red \in \calC^\red$ imply that  $\mathrm{pc}^{\calC}$ can be computed as
$(-1)^d\calL^{\calC^\red}({\bf A^\red},\calF^\red\setminus \calT^\red,b^\red)$.
\end{proof}

\begin{remark}\label{rem:SZenes}
Under the conditions of Proposition \ref{prop:pc2} we have the
following application of the statement from Remark \ref{re:Szenes}
(based on \cite{SZV}): {\it Assume that  $b \in \Box(A)-\calC$ and
$b\geq 0$.  Then $\mathrm{pc}^{\calC}(f)=0$.} Indeed,
$\mathrm{pc}^{\calC}(f)=\calL^{\calC}({\bf
A},\calT,-b)=\calL^{\calC(-b)}({\bf A},\calT,-b),$ where $\calC(-b)$ is a
chamber containing $-b$. But  since $-b\leq 0$ one gets
$\calL^{C(-b)}({\bf A},\calT,-b)=0$ by \ref{prop:pc2}.

\vspace{2mm}

One of the key messages of the above examples (starting from
 \ref{ex:2}) is the following: `if $b$ is small compared with the $a_i$'s, then the periodic
constant is zero' (compare with \ref{PC} too).
\end{remark}

\subsection{The polynomial part of rational functions with
$d=s=2$} \labelpar{ss:POLPART} \

In this case $\mathrm{rank}(L)=2$, and we have two vectors in the
denominator of $f$, namely $a_i=(a_{i,1},a_{i,2})$, $i=1,2$. We
will order them in such a way that $a_2$ sits in the cone of $a_1$
and $E_1$, that is, $\det {a_{1,1}\ a_{1,2}\choose a_{2,1}\
a_{2,2}}<0$. The chamber decomposition will be the following:
$\calC_0:=\R_{\geq 0}\langle -E_1,-E_2\rangle$, $\calC_2:=\R_{\geq
0}\langle -E_1, a_1\rangle$, $\calC:=\R_{\geq 0}\langle
a_1,a_2\rangle$ and $\calC_1:=\R_{\geq 0}\langle a_2,-E_2\rangle$
(the index choice is motivated by the formulae from
\ref{ex:2}(b)).

Our goal is to write any rational function (with denominator
$(1-\bt^{a_1})(1-\bt^{a_2})$) as a sum of $f^+(\bt)$ and
$f^-(\bt)$, such that $f^+\in\Z[L']$ (the  `polynomial part of
$f$'), and $\mathrm {pc}^{e,\calC}(f^-)=0$. This is a generalization
of the decomposition in the one--variable case discussed in
\ref{PC}, and will be a major tool in the computation of the
periodic constant is section \ref{s:TN} for graphs with two nodes.
The specific  form of the decomposition  is
motivated  by Examples \ref{ex:2}(b) and \ref{ex:NAM}.

As above, we set $\Box({\bf A})=[0,1)a_1+[0,1)a_2$ and for $i=1,2$
we also consider the strips $$\Xi_i:= \{b=(b_1,b_2)\in L\otimes
\R\ |\ 0\leq b_i<a_{i,i}\}.$$
\begin{theorem}\labelpar{th:POLPART}
(1) \ Any function
$f(\bt)=\big(\sum_{k=1}^r\iota_k\bt^{b_k}\big)/\prod_{i=1}^2
(1-\bt^{a_i})$ can be written as a sum $f(\bt)=f^+(\bt)+f^-(\bt)$,
where

\vspace{1mm}

(a) \ $f^+(\bt)$ is a finite sum $\sum_{\ell} \kappa_\ell
\bt^{\beta_\ell}$, with $\kappa_\ell\in\Z$ and $\beta_\ell\in L'$;

(b) \ $f^-(\bt)$ has the form
\begin{equation}\label{eq:POLPART}
f^-(\bt)=\frac{\sum_{k=1}^r\iota_k\bt^{b'_k}}{\prod_{i=1}^2
(1-\bt^{a_i})}+
\frac{\sum_{i=1}^{n_1}\iota_{i,1}\bt^{b_{i,1}}}{1-\bt^{a_1}}+
\frac{\sum_{i=1}^{n_2}\iota_{i,2}\bt^{b_{i,2}}}{1-\bt^{a_2}},
\end{equation}
with $b'_k\in L'\cap \Box({\bf A})$ for all $k$, and $b_{i,j}\in
L'\cap \Xi_j$ for any $i$ and $j=1,2$.

\vspace{1mm}

(2) Consider a sum
\begin{equation}\label{eq:SUMPOL} \Sigma(\bt):=
\frac{Q(\bt)}{\prod_{i=1}^2 (1-\bt^{a_i})}+
\frac{Q_1(\bt)}{(1-\bt^{a_1})}+
\frac{Q_2(\bt)}{(1-\bt^{a_2})}+f^+(\bt),
\end{equation}
where $ Q(\bt):=\sum_{k=1}^r\iota_k\bt^{b'_k}$ with $b'_k\in
L'\cap \Box({\bf A})$ for all $k$;
$Q_j(\bt)=\sum_{i=1}^{n_1}\iota_{i,j}\bt^{b_{i,j}}$  with
 $b_{i,j}\in
L'\cap \Xi_j$ for any $i$ and $j=1,2$;  and finally $f^+\in\Z[L']$
is a polynomial as in part (a) above.

Then, if $\Sigma(\bt)=0$, then
$Q(\bt)=Q_1(\bt)=Q_2(\bt)=f^+(\bt)=0$.

In particular, the decomposition in part (1) is unique.

\vspace{1mm}

 (3) The periodic constant of $f^-(\bt)$ associated
with the chamber $\calC$ is zero. Hence, in the decomposition (1)
one also has
$\mathrm{pc}^{e,\calC}(f)=\mathrm{pc}^{e,\calC}(f^+)=\sum_{\ell}
\kappa_\ell [\beta_\ell]\in\Z[H]$.
\end{theorem}

\begin{remark}\label{rem:POL} (a)
In the expression of $f$ and $f^-$ above we wished to emphasize
that even the summation $\sum_{k=1}^r$ and coefficients $\iota_k$
are preserved when we decompose  $f$ into $f^-$ and $f^+$. In
fact, for every $b_k\in L'$ we have a {\it unique} $b_k'\in L'\cap
\Box({\bf A})$ such that $b_k-b_k'\in \Z\langle a_i\rangle$. This
is how we get from the expression of $f$ the first fraction in
(\ref{eq:POLPART}).

(b) This decomposition is associated with a choice of a chamber:
here the chamber is $\calC$, and the decomposition satisfies
property (3) of the chamber $\calC$, where this choice is hidden.

Although in the present work we will not use any other
decomposition, we note that in general any chamber $\calC_*$
provides a decomposition with $f^+\in\Z[H]$ and
$\mathrm{pc}^{\calC_*}(f^-)=0$   (and usually these decompositions
are different). For example, for $\calC_0$,
$\mathrm{pc}^{\calC_0}(f)=0$, hence we can take $f^+=0$.
\end{remark}
\begin{proof}
First one determines the first fraction of $f^-$ as it is
explained in Remark \ref{rem:POL}(a). Then one has to decompose
fractions of type $(1-\bt^{k_1a_1+k_2a_2})/\prod_{i=1}^2
(1-\bt^{a_i})$ which is again  elementary.

Part (2) is again elementary algebra. Or, proceed as follows. The
vanishing of $Q$ follows again by the unicity of the choice of
$b_k'$ in \ref{rem:POL}(a). For the others, take a convenient
filtration of $\Z[L']$ (e.g. by integral multiples of $\Xi_1$,
resp. of $\Xi_2$).

The vanishing of the periodic constant of the first fraction of
$f^-$ follows from the proof of Lemma \ref{lem:d2}. The vanishing
of $\mathrm{pc}^{e,\calC}$ of the other two fractions follows from
Example  \ref{ex:2}(b). For the last expression see Example
\ref{ex:2}(a).
\end{proof}

\section{The case of rational functions associated with  plumbing graphs }\labelpar{s:PLRF}

\subsection{A `classical' connection between  polytopes and
gauge invariants (and its limits).}\labelpar{ss:CLASSI} \

\vspace{2mm}

\noindent In the literature of normal surface singularities there
is a sequence  of results which connect the topology of the link
with the number of lattice points in a certain polytope. Here are
some details.

The first step is based on the theory of hypersurface
singularities with Newton nondegenerate principal part, see e.g.
\cite{AGV}. According to this, for such a germ one defines the
Newton polytope $\Gamma^-_N$ using the nontrivial monomials of the
defining equation of the germ, and one proves that several
invariants of the germ can be recovered from $\Gamma^-_N$, see
e.g. \cite{BN07}. E.g., by a result of Merle and Teissier
\cite{MT}, the geometric genus $p_g$ equals with the number of
lattice points in $((\Z_{>0})^3\cap \Gamma^-_N)$. The second step
is provided by Laufer--Durfee formula, which determines the
signature of the Milnor fiber $\sigma$ as $-8p_g-K^2-|\cV|$
\cite{D78}.  Finally, there is a conjecture of Neumann and Wahl,
formulated for hypersurfaces with integral homology sphere links
\cite{NW91}, and proved e.g. for Brieskorn, suspension \cite{NW91}
and splice quotient  \cite{NO1} singularities, according to which
$\sigma/8=\lambda(M)$, the Casson invariant of the link.
Therefore, if all these steps run, e.g. in the Brieskorn case,
then the Casson invariant of the link, normalized by $K^2+|\cV|$,
 can be expressed as the number of lattice points
of a polytope associated with the equation of the germ.

[For the computations of the lattice points in the case
simplicial polytopes in terms of Dedekind sums see e.g. \cite{BP,Beck_c,BR1,DR} and the
citations therein, for its relation with the 
Riemann Roch formulae see e.g. \cite{CS,KK92,P} or literature of classical toric geometry, while 
for the relation of Dedekind sums with the Casson invariant see the classical
book \cite{Lescop}.]

 The above  correspondence has several deficiencies. First, even in simple cases,
we do not know how to extend the correspondence to the equivariant
case (that is, how to express the equivariant geometric genus from
$\Gamma^-_N$). Second, the expected generalization, the
Seiberg--Witten Invariant Conjecture (see \ref{FM}), which aims to
identifies    the Seiberg--Witten invariant of the link with $p_g$
(or $\sigma$) is still open for Newton nondegenerate germs. 
Finally, this  family of germs is rather
 restrictive. [Additionally, as a general fact about
 lattice point computations, in the literature there very few explicit formulae for
 the Ehrhart polynomial of non--simplicial polyhedrons.]

 The present article defines another  polytope,  which carries an action of the group $H$, and
its Ehrhart invariants
determine the Seiberg--Witten invariant {\it in any case}. It is not described from
the equations of the germ, but from its multivariable `zeta-function' $Z(\bt)$.
Furthermore, the polytope is a union of several simplices, and those coefficients 
of the Ehrhart polynomail which carry the information about the Seiberg--Witten invariant 
will be determined. 

\subsection{The new construction. Applications of Section \ref{s:PC}.}\labelpar{ss:cor} \

\vspace{2mm}

Consider the topological setup of a surface singularity, as in
subsection \ref{ss:11}. The lattice $L$ has  a canonical basis
$\{E_v\}_{v\in \calv}$ corresponding to the vertices of the graph $\Gamma$.
We investigate  the periodic constant of the rational function $Z(\bt)$, defined  in
\ref{SW} from $\Gamma$. Since  $Z(\bt)$ has the form
(\ref{eq:func}), all the results of section \ref{s:PC} can be applied.
In particular, if $\cale=\{v\in\calv\,:\, \delta_v=1\}$ denotes the set of {\em ends} of the
graph, then ${\bf A}$ has column vectors $a_v=E_v^*$ for $v\in\cale$. Hence,
the rank of the lattice/space where the polytopes $P^{(l')}=\cup_vP_v$ sit is $d=|\cale|$,
and the convex polytopes $\{P_v\}$ are indexed by $\calv$. Furthermore, the dilation parameter $l'$
of the polytopes runs in a $|\cV|$--dimensional space.   In the sequel we will drop
the symbol ${\bf A}$ from $\calL^\calC_h({\bf A},\calT,l')$.

[The construction has some analogies with the construction of the splice quotient singularities
\cite{nw-CIuac}: in that case  the equations of the universal abelian cover of the
singularity are written in $\C^d$, together with an action of $H$. Nevertheless, in the present situation,
we are not obstructed with the semigroup and congruence relations present in that theory.]

In this new construction, a crucial additional ingredient comes
from singularity theory, it is Theorem \ref{th:JEMS} (in fact,
this is the main starting point and motivation of the whole
approach). This combined with facts from Section \ref{s:PC} give:
 \begin{corollary}\labelpar{cor:4.1}
 Let $\calS=\calS_{\R}$ be the (real) Lipman cone $\{x\in \R^{|\cV|}: (x,E_v)\leq 0 \ \mbox{for all $v$}\}$.

  (a) The rational function
  $Z(\bt)$ admits a periodic constant in the cone $\calS$, which  equals the
   normalized Seiberg--Witten invariant
 \begin{equation}\label{eq:4.2}
 \mathrm{pc}^{\calS}_h(Z)=
-\frac{(K+2r_h)^2+|\cV|}{8}-\frsw_{-h*\sigma_{can}}(M).\end{equation}

(b)  Consider the chamber decomposition associated with the
 denominator  of $Z(\bt)$ as in Theorem \ref{th:PQP}, and let $\calC$ be a chamber
such that $int(\calC\cap \calS)\not=\emptyset$.
Then $Z(\bt)$ admits a periodic constant in $\calC$,
which  equals both  $\mathrm{pc}^{\calS}_h(Z)$ (satisfying  (\ref{eq:4.2})) and also
\begin{equation}\label{eq:PCC2}
{\mathrm{pc}}_h^\calC(Z)\, = \,
\sum_k\iota_k\cdot\calL^\calC_{h-[b_k]}(\calT,-b_k)=
\sum_k\iota_k\cdot\calL^\calC_{[b_k]-h}(\calF\setminus \calT,b_k).
\end{equation}
In particular, ${\mathrm{pc}}_h^\calC(Z)$ does not depend on the choice of \,$\calC$
(under the above assumption).
\end{corollary}
\begin{proof}
 Write  $l'=\tilde{l}+r_h$ with $\tilde{l}\in L$ in  (\ref{eq:SUM}).
Since
$\sum_{l\in L,\, l\not\geq 0}p_{l'+l}=
\sum_{l''\not\geq \tilde{l},\, [l'']=h}p_{l''}$,
(a)  follows from Theorem \ref{th:JEMS}. For  (b) use Corollary \ref{cor:Taylor} and
 Proposition \ref{prop:pc}.
\end{proof}

We note that the Lipman cone $\calS$ can indeed be cut in several
chambers (of the denominator of $Z$). This can happen even in the
simple case of Brieskorn germs.  Below we provide such an example
together with several exemplifying details of the construction.

\begin{example}\labelpar{ex:tref} {\bf Lipman cone cut in several chambers.} \
Consider the 3--manifold $S^3_{-1}(T_{2,3})$ (where $T_{2,3}$
is the right-handed trefoil knot), or, equivalently,
 the link of the
hypersurface singularity $z_1^2+z_2^3+z_3^7=0$. Its
plumbing graph $\Gamma$ and matrix $-I^{-1}$ are:

\begin{picture}(200,75)(60,10)
\put(150,55){\circle*{3}}
\put(180,55){\circle*{3}}
\put(120,55){\circle*{3}}
\put(150,25){\circle*{3}}
\put(150,55){\line(-1,0){30}}
\put(150,55){\line(1,0){30}}
\put(150,55){\line(0,-1){30}}
\put(150,70){\makebox(0,0){$E_0$}}
\put(180,70){\makebox(0,0){$E_3$}}
\put(120,70){\makebox(0,0){$E_1$}}
\put(135,25){\makebox(0,0){$E_2$}}
\put(160,45){\makebox(0,0){$-1$}}
\put(105,55){\makebox(0,0){$-2$}}
\put(195,55){\makebox(0,0){$-7$}}
\put(165,25){\makebox(0,0){$-3$}}
\put(345,45){\makebox(0,0){$-I^{-1}=\begin{pmatrix}
42&21&14&6\\
21&11&7&3\\
14&7&5&2\\
6&3&2&1
\end{pmatrix}$}}
\end{picture}

\noindent
 where the row/column vectors of $-I^{-1}$ are $E_0^*$,
$E_1^*$, $E_2^*$ and $E_3^*$ in the $\{E_v\}$ basis. The polytopes defined in (\ref{eq:POL}), or in (\ref{eq:Pv}),
with parameter $l=(l_0,l_1,l_2,l_3)\subset \Z^4$, sit in $\R^3$. Let $u_1,u_2,u_3$ be the basis of
$\R^3$. Then the polytopes  are the following convex closures:
\begin{eqnarray*}
 P_0^{(l)}& = & conv\left(0,\left(l_0/21\right)u_1,\left(l_0/14\right)u_2,
\left(l_0/6\right)u_3\right)\\
 P_1^{(l)}& = & conv\left(0,\left(l_1/11\right)u_1,\left(l_1/7\right)u_2,
\left(l_1/3\right)u_3\right)\\
 P_2^{(l)}& = & conv\left(0,\left(l_2/7\right)u_1,\left(l_2/5\right)u_2,
\left(l_2/2\right)u_3\right)\\
 P_3^{(l)}& = & conv\left(0,\left(l_3/3\right)u_1,\left(l_3/2\right)u_2,
\left(l_3\right)u_3\right).
\end{eqnarray*}
Since $E_0^*+\varepsilon(-E_0)$ is in the interior of the (real) Lipman cone
for $0<\varepsilon \ll 1$, we get that the Lipman cone is cut in several chambers.
The periodic constant can be computed with any of them.
In fact, by the continuity of the quasipolynomials associated with the chambers, any quasipolynomial
associated with a chamber which contains
any ray in the Lipman cone, even if it is situated at its boundary, provides the
periodic constant.
One such degenerated polytope provided by a ray on the boundary of $\calS$ is of special interest. Namely,
if we take $l=\lambda E_0^* \in \calS$ for $\lambda>0$, then $P^{(l)}=\bigcup_{v=0}^3 P_v^{(l)}$ is
the same as $P_0^{(l)}=conv(0,2\lambda u_1,3\lambda u_2,7\lambda u_3)$.
%
Moreover, if  $\calC$ is any chamber which contains the ray
$\R_{\geq 0} E_0^*$ at its boundary, then for any $l=\lambda
E^*_0$ one has
$\calL^{\calC}({\bf
A},\calT,l)=\calL(\widetilde{P}_0,\calT,\lambda)$, 
 where the last
is the classical Ehrhart polynomial of the tethrahedron
$\widetilde{P}_0:=conv(0,2 u_1, 3u_2,7u_3)$. Here we witness an
additional  coincidence of $\widetilde{P}_0$ with the Newton
polytope $\Gamma^-_N$ of the equation $z_1^2+z_2^3+z_3^7$.

We compute $ \calL(\widetilde{P}_0,\calT,\lambda)$  as follows.
From (\ref{eq:KV2})--(\ref{eq:KV2b}) and Corollary
\ref{cor:Taylor}, we get that
\begin{equation}\label{eq:NEW}
\chi(\lambda E_0^*)+\mbox{geometric genus of }
\{z_1^2+z_2^3+z_3^7=0\}=\calL(\widetilde{P}_0,\calT,\lambda)-\calL(\widetilde{P}_0,\calT,\lambda-1).
\end{equation}
Since this geometric genus is 1, and the free term of
$\calL(\widetilde{P}_0,\calT,\lambda)$ is zero (since for $\lambda
=0$ the zero polytope with boundary conditions contains no lattice
point), and $-K=2E_0+E_1+E_2+E_3$, we get that
$\calL(\widetilde{P}_0,\calT,
\lambda)=7\lambda^3+10\lambda^2+4\lambda$.
We emphasize that a formula as in (\ref{eq:NEW}), realizing a bridge between
the Riemann--Roch expression $\chi$ (supplemented with the geometric genus)
and the Ehrhart polynomial of the Newton diagram, was not known
for Newton nondegenerate germs.

\vspace{2mm}

 In the sequel we will provide several examples, when
the Newton polytope is not even defined.
\end{example}

\subsection{Example. The case of lens spaces}\labelpar{ex:lens} \

\vspace{2mm}

\bekezdes \ As we will see in  Theorem \ref{th:REST},
the complexity of the problem depends basically  on the number of nodes
$\calN=\{v\in \cV:\delta_v\geq 3\}$ of  $\Gamma$.
 In this subsection we treat the case when there are no nodes at all, that is $M$ is a
lens space. In this case the numerator of the rational function $f(\bt)$ is 1, hence
everything is described by the 2--dimensional polytopes determined by the denominator.
In the literature there are several results about lens spaces fitting in the present program,
here we collect the relevant ones completing with the new interpretations.
This subsection also serves as a preparatory part, or model,  for the study of chains of arbitrary graphs.

\bekezdes {\bf The setup.}
Assume that the plumbing graph is
\begin{picture}(170,20)(80,15)
\put(100,20){\circle*{3}}
\put(130,20){\circle*{3}}
\put(200,20){\circle*{3}}
\put(230,20){\circle*{3}}
\put(100,20){\line(1,0){50}}
\put(230,20){\line(-1,0){50}}
\put(165,20){\makebox(0,0){$\cdots$}}
\put(100,30){\makebox(0,0){$-k_1$}}
\put(130,30){\makebox(0,0){$-k_2$}}
\put(230,30){\makebox(0,0){$-k_s$}}
\put(200,30){\makebox(0,0){$-k_{s-1}$}}
\end{picture}
with all $k_v\geq 2$, and
$p/q$ is expressed via the (Hirzebruch, or negative) continued fraction
\begin{equation}\label{eq:HCF}
[k_1,\ldots, k_s]=k_1-1/(k_2-1/(\cdots -1/k_s)\cdots ).
\end{equation}
Then $M$ is the lens space $L(p,q)$. We also define $q'$ by
$q'q\equiv 1$ mod $p$, and $0\leq q'<p$. Furthermore, we set $g_v=[E^*_v]\in H$.
Then $g_s$ generates $H=\Z_p$, and any element
of $H$ can be written as $ag_s$ for some $0\leq a<p$. Recall the definitions of
$r_h$ and $s_h$ from \ref{ss:11} as well.

From analytic point of view $(X,o) $  is a cyclic quotient singularity
$(\C^2,o)/\Z_p$, where the action is $\xi*(x,y)=(\xi x,\xi^q y)$ (here $\xi$ runs over $p$--roots of unity).

\bekezdes {\bf The Seiberg--Witten invariant.} Since $(X,o)$ is rational,
in this case $Z(\bt)=P(\bt)$ (cf. subsection \ref{FM}). Moreover, in (\ref{eq:KV2b})
 $H^1(\cO_{\widetilde{Y}})=0$, hence
\begin{equation}\label{eq:lens2}
 \frsw_{-h*\sigma_{can}}(M)=
-\frac{(K+2r_h)^2+|\cV|}{8}=-\frac{K^2+|\cV|}{8}+\chi(r_h).
\end{equation}
On the other hand, in  \cite{NOSZ,NLC} a similar formula is proved for the SW-invariant:
one only has to replace in (\ref{eq:lens2}) $\chi(r_h)$ by $\chi(s_h)$.
In particular, for lens spaces, and for any $h\in H$ one has
\begin{equation}\label{eq:lens3}
\chi(r_h)=\chi(s_h).
\end{equation}
[Note that, in general, for other links, $\chi(r_h)>\chi(s_h)$ might happen, see Example~\ref{ex:sh}.
Here, (\ref{eq:lens3})
follows from the vanishing of the geometric genus of the universal abelian cover of ($X,o)$.]

In general,  the coefficients of
the representatives $s_{ag_s}$ and $r_{ag_s}$ ($0\leq a<p$)
are rather complicated arithmetical expressions;
for $s_{ag_s}$ see \cite[10.3]{NOSZ} (where $g_s$ is defined with opposite sign).
The value  $\chi(s_{ag_s})$ is computed in
\cite[10.5.1]{NOSZ} as
\begin{equation}\label{eq:lens1}
\chi(s_{ag_s}) =\frac{a(1-p)}{2p} +\sum_{j=1}^a \Big\{\frac{jq'}{p}\Big\}.
\end{equation}

For completeness of the discussion we also recall that
$K=E_1^*+E_s^*-\sum_vE_v$ and
\begin{equation}\label{eq:lens4}
(K^2+|\cV|)/4=(p-1)/(2p)-3\cdot \bms(q,p),
\end{equation}
cf. \cite[10.5]{NOSZ}, where $\bms(q,p)$ denotes the Dedekind sum
\[
\bms(q,p)=\sum_{l=0}^{p-1}\Big(\Big( \frac{l}{p}\Big)\Big)
\Big(\Big( \frac{ ql }{p} \Big)\Big), \ \mbox{where} \ \
((x))=\left\{
\begin{array}{ccl}
\{x\} -1/2 & {\rm if} & x\in {\R}\setminus {\Z}\\
0 & {\rm if} & x\in {\Z}.
\end{array}
\right.
\]
In particular,  $\frsw_{-h*\sigma_{can}}(M)$ is determined via the formulae
(\ref{eq:lens2}) -- (\ref{eq:lens4}).

The non-equivariant picture looks as follows:
 $\sum_h \frsw_{-h*\sigma_{can}}=\lambda$,
the Casson--Walker invariant of $M$, hence (\ref{eq:lens2}) gives
$$\lambda=-p(K^2+|\cV|)/8+\textstyle{\sum_h}\chi(r_h).$$
This is compatible with (\ref{eq:lens4}) and formulae
$\lambda(L(p,q))=p\cdot\bms(q,p)/2$ and $\sum_h\chi(r_h)=(p-1)/4-p\cdot \bms(q,p)$, cf.
\cite[10.8]{NOSZ}.

\bekezdes {\bf The polytope and its quasipolynomial.}
We compare the above data with Ehrhart theory.
In this case $Z(\bt)=(1-\bt^{E^*_1})^{-1}(1-\bt^{E^*_s})^{-1}$.
The vectors $a_1=E^*_1$ and $a_s=E^*_s$ determine the polytopes
$P^{(l')}$ and a chamber decomposition.

For $1\leq v\leq w\leq s$ let $n_{vw}$ denote by the numerator of the continued fraction
$[k_v,\ldots,k_w]$ (or, the determinant of the corresponding bamboo subgraph).
For example, $n_{1s}=p$, $n_{2s}=q$ and $n_{1,s-1}=q'$. We also set
$n_{v+1,v}:=1$. Then $pE^*_1=\sum_{v=1}^sn_{v+1,s}E_v$ and $pE^*_s=\sum_{v=1}^sn_{1,v-1}E_v$.

In particular, for any $l'=\sum_vl'_vE_v\in\calS'$, the (non-convex) polytopes are
\begin{equation}\label{eq:lens5}
P^{(l')}=\bigcup_{v=1}^s\Big\{\,(x_1,x_s)\in\R_{\geq 0}^2\,:\, x_1n_{v+1,s}+x_sn_{1,v-1}\leq pl'_v\Big\}
\subset \R_{\geq 0}^2.
\end{equation}
The representation $\Z^2\stackrel{\rho}{\longrightarrow}\Z_p$ is $(x_1,x_s)\mapsto (qx_1+x_s)g_s$.

Though $P^{(l')}$ is a plane polytope, the direct computation of its equivariant Ehrhart multivariable polynomial
(associated with a chamber, or just with the Lipman cone)
 is still highly non-trivial. Here  we will rely again on  Theorem \ref{th:JEMS}.
On a subset of type $l'_0+\calS'$ the identity (\ref{eq:SUM}) provides the counting function. The right hand side
of (\ref{eq:SUM})  depends on all the
coordinates of $l'$, hence all the  triangles $P_v$ contribute in $P^{(l')}$. Since this can happen only in
a unique combinatorial way, we get that there is a chamber $\calC$ which contains the Lipman cone.
Let $\calL^{e,\calC}$ be its quasipolynomial, and  $\calL^{e,\calS}$ its restriction on $\calS$.
Since the numerator of $Z(\bt)$ is 1, $\overline{Q}^{\calC}_h=\calL^{\calC}_h$.
Since this agrees with the right hand side of (\ref{eq:SUM}) on a cone of type
$l'_0+\calS'$, and the Lipman cone is in $\calC$,  we get that
\begin{equation}\label{eq:lens6}
Q_h(l')=\overline{Q}^{\calC}_h(l')=\calL^{\calS}_h(l')=-\frsw_{-h*\sigma_{can}}(M)-\frac{(K+2l')^2+|\cV|}{8}
\end{equation}
for any $l'\in (r_h+L)\cap \calS'$ and  $h\in H$.
Using the identity (\ref{eq:lens2}), this reads as
\begin{equation}\label{eq:lens7}
\calL^{\calS}_h(\calT,l')=\chi(l')-\chi(r_h),
\ \ \ l'\in (r_h+L)\cap \calS'.
\end{equation}
Note that for any fixed $h$ and any $l'$ there exists a unique $q=q_{l',h}\in \square$ such that
$l'+q:=l''\in r_h+L$.
Indeed, take for $q$ the representative of $r_h-l'$ in $\square$. Then (\ref{eq:cccc})
and (\ref{eq:lens7}) imply
\begin{equation}\label{eq:Lg}
\calL^{\calS}_h(\calT,l')
=\calL^{\calS}_h(\calT,l'')=\chi(l'+q_{l',h})-\chi(r_h).\end{equation}
This formula emphasizes the periodic behavior of  $\calL^{\calS}_h(\calT,l')$ as well.

If $l'$ is an element of  $L$
then $q_{l',h}=r_h$, hence (\ref{eq:Lg}) gives in this case
\begin{equation}\label{eq:Lg2}
\calL^{\calS}_h(\calT,l)
=\chi(l+r_h)-\chi(r_h)=\chi(l)-(l,r_h) \ \ \ \mbox{for $l\in L\cap \calS$}.\end{equation}
In particular,
$\mathrm{pc}(\calL_h^\calS(\calT))=\chi(r_h)-\chi(r_h)=0$ (a fact compatible with  $H^1(\cO_{\widetilde{Y}})=0$).



\vspace{2mm}

Even the non-equivariant case looks rather interesting.
Let $\calL^{\calS}_{ne}(\calT)=\sum_{h\in H}\calL^{\calS}_h(\calT)$ be the Ehrhart polynomial
of $P^{(l')}$ (with boundary condition $\calT$), where we count all the lattice points
independently of their class in $H$.
Then,  (\ref{eq:Lg2}) gives
for $ l\in L\cap \calS$
\begin{equation}\label{eq:lens8}
\calL^{\calS}_{ne}(\calT,l)=p\cdot \chi(l)-(l,\textstyle{\sum_h}r_h)=
-p\cdot
(l,l)/2-p\cdot (l,K)/2-(l,\textstyle{\sum_h}r_h).
\end{equation}
In fact, $\sum_hr_h$ can  explicitly be computed. Indeed, set
 $d_v={\mathrm{gcd}}(p,n_{1,v-1})$ and  $p_v=p/d_v$.  Then one
checks that $aE^*_s=\sum_vn_{1,v-1}\frac{a}{p}E_v$,
$r_h=\sum_v\big\{n_{1,v-1}\frac{a}{p}\big\}E_v$ and $\sum_hr_h=\sum_vd_v\frac{p_v-1}{2}E_v$.

The coefficients of the polynomial $\calL^{\calS}_{ne}(\calT,l)$
can be compared with the coefficients given by general theory
of Ehrhart polynomials applied for $P^{(l)}$. E.g., the leading coefficient gives
 \begin{equation*}\label{eq:lens9}
-p \cdot(l,l)/2= \ \mbox{Euclidian area of }\ P^{(l)}.
\end{equation*}
Knowing that in $P^{(l)}$ all the $P_v$'s contribute, and it depends on $s$ parameters, and the
intersection of their boundary is messy, the simplicity and conceptual form of
(\ref{eq:lens8}) is rather remarkable.

\subsection{Reduction of the variables of $Z(\bt)$}\labelpar{ss:REST} \

\vspace{2mm}

Let $\calN$ denote the set of nodes $\{v\in\calv\,:\, \delta_v\geq 3\}$. Let
$\calS_\calN$ be the positive cone
$\R_{\geq 0}\langle E^*_n\rangle_{n\in \calN}$  generated by the dual base elements indexed by $\calN$, and  $V_\calN:= \R\langle E^*_n\rangle_{n\in \calN}$
be its supporting linear subspace in $L\otimes \R$. Clearly  $\calS_\calN\subset \calS$.
Furthermore, consider  $L_\calN:=\Z\langle E_n\rangle_{n\in \calN}$ generated by the node
base elements, and
the projection $pr_\calN:L\otimes \R\to L_\calN\otimes \R$ on the node coordinates.

\begin{lemma}\label{lem:RES}
The restriction of $pr_\calN$ to $V_\calN$, namely $pr_\calN:V_\calN\to L_\calN\otimes \R$, is an isomorphism.
\end{lemma}
\begin{proof}
Follows from the negative definiteness of the intersection form of the plumbing, which guarantees that
any minor situated centrally on the diagonal is nondegenerate.
\end{proof}
Our goal is to prove that restricting the counting function to the subspace $V_\calN$, the non-node variables of $Z(\bt)$ and $Q(l')$ became non-visible, hence they can be eliminated. This fact will
provide a remarkable simplification in the periodic constant computation.
But, {\it before} any elimination-substitution,
we have first to decompose our series $Z(\bt)$
into $\sum_{h\in H}Z_h(\bt)[h]$
if we wish to preserve the information about all the $H$ invariants,
cf. the comment at the end of \ref{bek:LL'2}.

\begin{theorem}\labelpar{th:REST}
(a) The restriction of $\calL_h({\bf A}, \calT,l')$ to $V_\calN$
 depends only on those coordinates which are indexed by the nodes (that is, it
depends only on $pr_\calN(l')$ whenever $l'\in V_\calN$).

(b) The same is true for the counting function $Q_h$ associated with $Z_h(\bt)$ as well.
In other words, if we consider the restriction
$$Z_h(\bt_\calN):=Z_h(\bt)|_{t_v=1\ \mbox{\tiny{for all $v\not\in\calN$}}}$$
then for any $l'\in L_\calN$, the counting functions $\sum_{l''\not\geq l'} p_{l''}[l'']$
of  $Z_h(\bt)$ and $Z_h(\bt_\calN)$ are the same.

(c) Consider the chamber decomposition of $\calS_\calN$ by intersections of type
$\calC_\calN:=\calC\cap\calS_\calN$, where $\calC$ denotes a chamber (of $Z$)
such that $int(\calC\cap \calS)\not=\emptyset$, and
the intersection of $\calC$ with the relative interior of $\calS_\calN$ is also
non-empty.
Then  \begin{equation}\label{eq:REDPC}
{\mathrm{pc}}^\calC(Z_h(\bt))=
 {\mathrm{pc}}^{\calC_\calN}(Z_h(\bt_\calN)).
 \end{equation}
\end{theorem}
The theorem applies as follows. Assume that we are interested in the computation of
$\mathrm{pc}^\calC_h(Z(\bt))$ for some chamber $\calC$
(e.g. when $\calC\subset \calS$, cf. Corollary \ref{cor:4.1}). Assume that
$\calC$ intersects the relative interior of $\calS_\calN$. Then,
the restriction to  $\calC\cap \calS_\calN$ of the quasipolynomial associated with
$\calC$ has two properties: it still preserves sufficiently information to
determine $\mathrm{pc}^\calC_h(Z(\bt))$
(via the periodic constant of the restriction, see (\ref{eq:REDPC})),
but it also has the advantage that
for these dilatation parameters $l'$ the  union $\cup_{v\in\cV} P^{(l'),\triangleleft}_v$
equals the union of essentially much less polytopes, namely $\cup_{n\in\calN} P^{(l'),\triangleleft}_v$.

For example, when we have only one node, one has to handle only one convex simplicial polytope
instead of a union of $|\cV|$ simplices.

\begin{proof} (a) We show that for any $l'\in  V_\calN$ one has the inclusions
\begin{equation}\label{eq:4.13}
P_v^{(l'),\triangleleft}\,\subset\,
\bigcup_{n\in \calN}P_n^{(l'),\triangleleft} \ \mbox{for any $v\not\in \calN$}.
\end{equation}
We  consider two cases. First we assume that $v$ is on  a leg (chain) connecting an
end $e(v)\in\cale$ with a node $n(v)$ (where $e(v)=v$ is also possible).
Then, clearly, (\ref{eq:4.13}) follows from
\begin{equation}\label{eq:4.13b}
P_v^{(l'),\triangleleft}\,\subset\,
P_{n(v)}^{(l'),\triangleleft} \ \ \mbox{\ \ \ for any $l'\in \calS_\calN$}.
\end{equation}
Let  $E^*_{uv}=(E^*_u)_v=-(E^*_u,E^*_v)$ be the $v$--cordinate of $E^*_u$. Note that
$E^*_{uv}=E^*_{vu}$.  Using the definition of the polytopes,
(\ref{eq:4.13b}) is equivalent with the implication (cf. \ref{bek:pol})
\begin{equation}\label{eq:4.14}
\big(\ \sum_{e\in\cale}x_eE^*_{ve} < l'_{v}\ \big)
\Longrightarrow
\big(\ \sum_{e\in\cale}x_eE^*_{n(v)e} < l'_{n(v)}\ \big) \ \
\mbox{\ \ for any $l'\in \calS_\calN$ and $x_e\geq 0$}.
\end{equation}
Let $\calW$ be the set of vertices
on this leg (including $e(v)$ but not $n(v)$). Then, one verifies that there exist
positive rational numbers $\alpha$ and $\{\alpha_w\}_{w\in \calW}$, such that
\begin{equation}\label{eq:4.15}
E^*_v=\alpha\, E^*_{n(v)}+\sum_{w\in\calW}\alpha_wE_w.
\end{equation}
The numbers $\alpha$ and $\{\alpha_w\}_{w\in \calW}$ can be determined from the linear
system obtained by  intersecting the identity (\ref{eq:4.15}) by $\{E_w\}_w$ and $E_{n(v)}$.
Intersecting (\ref{eq:4.15}) by $E^*_e$ ($e\in\cale$), we get that $E^*_{ve}=\alpha E^*_{n(v)e}$ for any
$e\not=e(v)$,  and $E^*_{v,e(v)}=\alpha E^*_{n(v),e(v)}+\alpha_{e(v)}$. Hence
\begin{equation}\label{eq:4.16}
 \sum_{e\in\cale}x_eE^*_{ve} =\alpha
 \sum_{e\in\cale}x_eE^*_{n(v)e}+ x_{e(v)}\alpha_{e(v)}.
\end{equation}
On the other hand, intersecting (\ref{eq:4.15}) with $E^*_n$, for $n\in\calN$, we get
$E^*_{vn}=\alpha E^*_{n(v)n}$. Since $l'$ is a linear combination of $E^*_{n}$'s,
we get that
\begin{equation}\label{eq:4.17}
-l'_v=(l',E^*_v)=\alpha (l',E^*_{n(v)})=-\alpha l'_{n(v)}.
\end{equation}
Since $x_{e(v)}\alpha_{e(v)}\geq 0$, (\ref{eq:4.16}) and (\ref{eq:4.17}) imply (\ref{eq:4.14}).
This ends the proof of this case.

Next, we assume that $v$ is on a chain  connecting two nodes $n(v)$ and $m(v)$.
Let $\calW$ be the set of vertices
on this bamboo (not including $n(v)$ and  $m(v)$).
Then we will show that
\begin{equation}\label{eq:4.13c}
P_v^{(l'),\triangleleft}\,\subset\,
P_{n(v)}^{(l'),\triangleleft}\cup P_{m(v)}^{(l'),\triangleleft} \ \ \mbox{for any $l'\in \calS_\calN$}.
\end{equation}
This follows  as above from the
existence of positive rational numbers $\alpha$, $\beta$ and $\{\alpha_w\}_{w\in \calW}$ with
\begin{equation}\label{eq:4.18}
E^*_v=\alpha\, E^*_{n(v)}+\beta\, E^*_{m(v)}+\sum_{w\in\calW}\alpha_wE_w.
\end{equation}

(b) follows from (a) and from the fact that
 all $b_k$  entries in the numerator of $Z(\bt)$  belong to $\calS_\calN$.

 (c) If ${\mathrm{pc}}^{\calC}(Z_h(\bt))$ is computed as $\widetilde{Q}_h(0)$ for some
 quasipolynomial $\widetilde{Q}_h$ defined on $\widetilde{L}\subset L$, then part (b) shows that
 ${\mathrm{pc}}^{\calC_\calN}(Z_h(\bt_\calN))$ can be  computed as
 $(\widetilde{Q}_h|_{\widetilde{L}\cap \calS_\calN})(0)$, which equals $\widetilde{Q}_h(0)$.
\end{proof}

\begin{example}
Consider the following graph $\Gamma$:
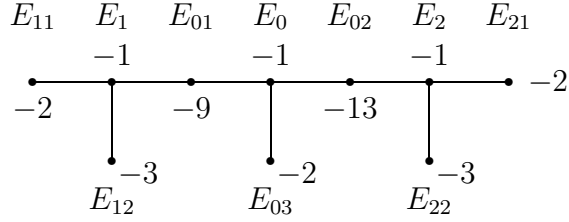
\begin{center}
\begin{picture}(150,80)(80,15)
\put(150,55){\circle*{3}}
\put(180,55){\circle*{3}}
\put(210,55){\circle*{3}}
\put(240,55){\circle*{3}}
\put(120,55){\circle*{3}}
\put(90,55){\circle*{3}}
\put(60,55){\circle*{3}}
\put(150,25){\circle*{3}}
\put(90,25){\circle*{3}}
\put(210,25){\circle*{3}}
\put(150,55){\line(-1,0){90}}
\put(150,55){\line(1,0){90}}
\put(150,55){\line(0,-1){30}}
\put(210,55){\line(0,-1){30}}
\put(90,55){\line(0,-1){30}}
\put(150,80){\makebox(0,0){$E_0$}}
\put(180,80){\makebox(0,0){$E_{02}$}}
\put(210,80){\makebox(0,0){$E_2$}}
\put(240,80){\makebox(0,0){$E_{21}$}}
\put(120,80){\makebox(0,0){$E_{01}$}}
\put(90,80){\makebox(0,0){$E_1$}}
\put(60,80){\makebox(0,0){$E_{11}$}}
\put(150,10){\makebox(0,0){$E_{03}$}}
\put(90,10){\makebox(0,0){$E_{12}$}}
\put(210,10){\makebox(0,0){$E_{22}$}}
\put(150,65){\makebox(0,0){$-1$}}
\put(180,45){\makebox(0,0){$-13$}}
\put(210,65){\makebox(0,0){$-1$}}
\put(255,55){\makebox(0,0){$-2$}}
\put(220,20){\makebox(0,0){$-3$}}
\put(160,20){\makebox(0,0){$-2$}}
\put(100,20){\makebox(0,0){$-3$}}
\put(60,45){\makebox(0,0){$-2$}}
\put(90,65){\makebox(0,0){$-1$}}
\put(120,45){\makebox(0,0){$-9$}}
\end{picture}
\end{center}
\vspace{0.3cm}
By the Theorem \ref{th:REST} we are interested only in those polytopes $P_v\subset \R^5$
which are associated to the nodes $E_1$, $E_2$ and $E_0$. Let $l\in \calS_{\calN}$, i.e.
$l=\lambda_1 E_1^*+\lambda_2 E_2^*+\lambda_0 E_0^*$.
Then one can  verify that  $\calS_{\calN}$ is divided by  the plane
$\lambda_1= (13/9)\lambda_2$. Hence, in general   $\calS_{\calN}$
too can be divided into several chambers.
[On the other hand, if the graph has at most two nodes this does not happen.]

\end{example}

\section{The one--node case, star--shaped plumbing graphs}\labelpar{s:seifert}

\subsection{Seifert invariants and other notations}\labelpar{ss:seiferjel}
Assume that the graph is star--shaped with $d$ legs.
Each leg is a chain  with normalized  Seifert  invariant $(\alpha_i,\omega_i)$,
where $0<\omega_i <\alpha_i$, gcd$(\alpha_i,\omega_i)=1$.
We also use $\omega_i'$ satisfying $\omega_i\omega_i'\equiv 1$ (mod $\alpha_i$), $0< \omega_i'<\alpha_i$.

If we consider the
Hirzebruch/negative continued fraction expansion, cf. (\ref{eq:HCF})
$$ \alpha_i/\omega_i=[b_{i1},\ldots, b_{i\nu_i}]=
b_{i1}-1/(b_{i2}-1/(\cdots -1/b_{i\nu_i})\cdots )\ \  \ \ (b_{ij}\geq 2),$$
then the $i^{\mathrm{th}}$ leg has $\nu_i$ vertices, say $v_{i1},\ldots, v_{i\nu_i}$,
 with decorations $-b_{i1},\ldots, -b_{i\nu_i}$, where
 $v_{i1}$ is connected by the central vertex. The corresponding base elements in $L$ are
$\{E_{ij}\}_{j=1}^{\nu_i}$.
Let  $b$ be the decoration of the  central vertex; this vertex also defines $E_0\in L$.
The plumbed 3--manifold $M$ associated with such a star--shaped graph has a Seifert structure. It is
a rational homology sphere if and only if  the central vertex has genus zero;
this fact will be assumed in the sequel.

The classes
in $H=L'/L$ of the dual base elements are denoted by
$g_{ij}=[E^*_{ij}]$ and  $g_0=[E^*_0]$. For simplicity we also write
$E_i:=E_{i\nu_i}$ and $g_i:=g_{i\nu_i}$. A possible presentation of  $H$ is
\begin{equation}\label{eq:sei1}
H={\mathrm{ab}}\langle \, g_0,g_1,\ldots, g_d\,|\, -b\cdot g_0=\sum_{i=1}^d \omega_i\cdot g_i;
\, g_0=\alpha_i\cdot g_i\ (1\leq i\leq d)\rangle,
\end{equation}
cf. \cite{neumann.abel} (or use $\sum_k I_{ik}g_k$ repeatedly, see also  (\ref{eq:DET})).
The orbifold Euler number of $M$ is defined as $e=b+\sum_i\omega_i/\alpha_i$. The negative definiteness of
the intersection form implies $e<0$. We write $\alpha:=\mathrm{lcm}(\alpha_1,\ldots,\alpha_d)$,
$\frd:=|H|$ and
$\fro$ for the order of $g_0$ in $H$. One has (see e.g. \cite{neumann.abel})
\begin{equation}\label{eq:sei2}
\frd=\alpha_1\cdots\alpha_d|e|, \ \ \ \ \fro=\alpha|e|.
\end{equation}
Each leg has similar invariants as the graph of a lens space, cf. Example \ref{ex:lens}, and we can
introduce similar notation. For example, the determinant of the  $i^{\mathrm{th}}$ leg is $\alpha_i$.
We write $n^i_{j_1j_2}$ for  the determinant of the sub-chain  of the   $i^{\mathrm{th}}$ leg
connecting the vertices $v_{ij_1}$ and $v_{ij_2}$ (including these vertices too). Then, using the correspondence between intersection
pairing of the dual base elements and the determinants of the subgraphs, cf. (\ref{eq:DETsgr})
or \cite[11.1]{NOSZ}, one has
\begin{equation}\label{eq:DET}
\begin{array}{ll}\vspace{2mm}
(a) \ \ (E^*_0,E^*_{ij}- n^i_{j+1,\nu_i}E^*_{i\nu_i})=0 \ \ & (b) \ \
g_{ij}=n^i_{j+1,\nu_i}g_{i\nu_i}\ \   (1\leq i\leq d,  \ 1\leq j\leq \nu_i)\\
(c) \ \ (E^*_i,E^*_0)=\frac{1}{\alpha_ie} \ \ & (d) \ \
(E^*_{0},E^*_{0})=\frac{1}{e} \ .
 \end{array}
\end{equation}
Part (b) explains why we do not need to insert the generators $g_{ij}$ ($j<\nu_i$) in (\ref{eq:sei1}).

For any $l'\in L'$ we set $\tc(l'):=-(E^*_0,l')$, the $E_0$-coefficient of $l'$. Furthermore,
if $l'=c_0E^*_0+\sum_{i,j} c_{ij}E^*_{ij}\in L'$, then we define its {\em reduced transform }
 by $$l'_{red}:= c_0E^*_0+\sum_{i,j} c_{ij} \cdot n^i_{j+1,\nu_i}E^*_{i}.$$
By (\ref{eq:DET}) we get
$[l']=[l'_{red}]$ in $H$, $\tc(l')=\tc(l'_{red})$, and if
 $l'_{red}=\sum_{i=0}^dc_iE^*_i$, then $\tc(l'_{red})$ is
\begin{equation}\label{eq:ntilde}
\tc:=\frac{1}{|e|}\cdot \big(\,c_0+
\sum_{i=1}^d\frac{c_i}{\alpha_i}\,\big).
\end{equation}
If $h\in H$, and $l'_h\in L'$ is any of its lifting (that is, $[l'_h]=h$),
then $l'_{h,red}$ is also a lifting of the same $h$ with $\tc(l'_h)=\tc(l'_{h,red})$.
In general, $\tc=\tc(l'_h)$ depends on the lifting, nevertheless replacing $l'_h$ by $l'_h\pm E_0$ we modify
$\tc$ by $\pm 1$, hence we can always achieve
$\tc\in [0,1)$,  where it is  determined uniquely by $h$.
 For example, since $r_h\in \square$, its $E_0$-coefficient $\tc(r_h)$ is in
$[0,1)$.

  Finally, we consider
\begin{equation}\label{eq:discr}
\gamma:=\frac{1}{|e|}\cdot \big( d-2-\sum_{i=1}^d \frac{1}{\alpha_i}\big).
\end{equation}
It has several `names'.
 Since the canonical class is given by $K=-\sum_vE_v+\sum_{v}(\delta_v-2)E^*_v$,
by (\ref{eq:DET}) we get that the $E_0$ coefficient of $-K$ is $(K,E^*_0)=\gamma+1$. The number $-\gamma$ is
sometimes called the `log discrepancy' of $E_0$, $\gamma$ the `exponent' of the weighted
homogeneous germ $(X,o)$, and $\fro\gamma$ is the Goto--Watanabe $a$--invariant of the universal abelian
cover of $(X,o)$, see  \cite[(3.1.4)]{G-W} and \cite[(3.6.13)]{c-m}; while in
\cite{neumann.abel} $e\gamma$ appears as an orbifold Euler characteristic.

\subsection{The function Z}\labelpar{ss:Z}
By the reduction Theorem \ref{eq:REDPC}, for the periodic constant computation, we can
reduce ourself to the variable of the single node, it will be denoted by $t$.

First we analyze the equivariant rational function associated with  the denominator of $Z^e$
$$Z^{/H}(t)=\prod_{i=1}^d\,\big(1-t^{-(E^*_{i}, E^*_0)}[g_i]\big)^{-1}=
\sum_{x_1,\ldots , x_d\geq 0}\, t^{\,\sum_i
x_i/(\alpha_i|e|)}\ \big[\,\sum _ix_ig_i\,\big]\in \Z[[t^{1/\fro}]][H].
$$
 The right hand side of the above expression
can be transformed  as follows (cf. \cite[\S 3]{NN2}).
If we fix a lift $\sum_{i=0}^dc_iE_i^*$ of $h$
as above, then using the presentation (\ref{eq:sei1}) one gets that
$\sum_{i=1}^dx_ig_i$ equals $h$ if and only if there exist
$\ell,\ell_1,\ldots, \ell_d\in\Z$ such that
$$\begin{array}{lrll}
(a) \ & -c_0&= \ \ell_1+\cdots +\ell_d-\ell b &\\
(b) \ & x_i-c_i&= -\omega_i \ell-\alpha_i\ell_i & \ (i=1,\ldots, d).
\end{array}$$
Since $x_i\geq 0$, from (b) we get
$\tilde{\ell}_i:=\big\lfloor \frac{c_i-\omega_i\ell}{\alpha_i}\big\rfloor -\ell_i
\geq 0$.
Moreover, if we set for $\bc=(c_0,c_1,\ldots, c_d)$
\begin{equation}\label{eq:N}
N_\bc(\ell):=1+c_0-\ell b +\sum_{i=1}^d\Big\lfloor \frac{c_i-\omega_i\ell}{\alpha_i}\Big\rfloor,
\end{equation}
then the number of realizations of $h=\sum_ic_ig_i$ in the form $\sum_ix_ig_i$ is given by the number of
integers $(\tilde{\ell}_1,\ldots,\tilde{\ell}_d)$ satisfying
$\tilde{\ell}_i\geq 0$ and $\sum_i\tilde{\ell}_i=N_\bc(\ell)-1$.  This is
$\binom{N_\bc(\ell)+d-2}{d-1}$. Moreover, the non-negative
integer $\sum_ix_i/(\alpha_i|e|)$ equals $\ell+\tc$.
Therefore,
\begin{equation}\label{eq:ZH}
Z_h^{/H}(t)=\sum_{\ell\geq -\tc}\ \binom{N_\bc(\ell)+d-2}{d-1}\
 t^{\ell+\tc}.
\end{equation}
This expression is independent of the choice of $\bc=\{c_i\}_{i=0}^d$.
Similarly, for any function $\phi$, the expression
$\sum_{\ell\geq -\tc}\phi(N_\bc(\ell))t^{\ell+\tc}$ is independent of the choice of $\bc$, it depends only
on $h=\sum_ic_ig_i$.

Furthermore, one checks that $N_\bc(\ell)\leq 1+(\ell+\tc)|e|$, hence if $\ell+\tc<0$ then $N_\bc(\ell)\leq 0$,
therefore $\binom{N_\bc(\ell)+d-2}{d-1}=0$ as well. Hence, in (\ref{eq:ZH})  the inequality $\ell+\tc\geq 0$
below the sum, in fact, is not restrictive.

Next, we consider the numerator $(1-[g_0]t^{1/|e|})^{d-2}$ of $Z^e(t)$.
A similar computation as above done for $Z^e(t)$ (see \cite{neumann.abel} and \cite[\S 3]{NN2}),
or by multiplying (\ref{eq:ZH}) by the numerator and using $\sum_{k=0}^{d-2}(-1)^k\binom{d-2}{k}
\binom{N-k+d-2}{d-1}=\binom{N}{1}$,  gives
\begin{equation}\label{eq:Zt}
Z_h(t)
=\sum_{\ell\geq -\tc}\
\max\{0, N_\bc(\ell)\} \ t^{\ell+\tc}.
\end{equation}
In order to compute the periodic constant of $Z_h(t)$ we decompose $Z_h(t)$ into its
 `polynomial and negative degree parts', cf. \ref{PC}. Namely, we write $Z_h(t)=Z^{+}_h(t)+Z^{-}_h(t)$, where
\begin{equation}\label{eq:Sp}\begin{array}{l}
Z^{+}_h(t)=\sum_{\ell\geq -\tc}\ \max\big\{0, -N_\bc(\ell)\big\} \ t^{\ell+\tc} \ \
\mbox{(finite sum with positive exponents)}
\\ \  \ \\
Z^{-}_h(t)=\sum_{\ell\geq -\tc}\ N_\bc(\ell) \ t^{\ell+\tc}.
\end{array}
\end{equation} In  $Z^{-}_h$ it is convenient to
fix a choice with $\tc\in[0,1)$, hence  the summation is  over $\ell\geq 0$. Then
a computation (left to the reader) shows that it is a rational function of negative degree
\begin{equation}\label{eq:Sp2}
Z^{-}_h(t)=\Big(\frac{1-e\tc}{1-t}-\frac{e\cdot t}{(1-t)^2}
-\sum_{i=1}^d\sum_{r_i=0}^{\alpha_i-1}\Big\{\frac{c_i-\omega_ir_i}{\alpha_i}\Big\}\, t^{r_i}\cdot
\frac{1}{1-t^{\alpha_i}}\Big)\cdot t^{\tc}.
\end{equation}
[This expression can be compared with the Laurent  expansion of $Z_h$ at $t=1$ which
was already considered
in the literature. Dolgachev, Pinkham, Neumann and Wagreich \cite{Do,pinkham,neumann.abel,wa}
determine the first two terms (the pole part), while  \cite{NN2,NOSZ} the third terms as well.
Nevertheless the above $Z_h^{+}+Z_h^{-}$ decomposition does not coincide
with the `pole+regular part' decomposition of the Laurent expansion terms, and focuses on
 different aspects.]

Since the degree of $Z^{-}_h$ is negative (or by a direct computation)
 $\mathrm{pc}(Z^{-}_h)=0$, cf. \ref{PC}.

On the other hand, since $e<0$, in $Z^{+}_h(t)$ the sum is finite.
(The degree of $Z^{+}_0$ is $\leq \gamma$,
see e.g. \cite{NO2}. Since
$N_{{\bf c}(r_{h,red})}(\ell) \geq N_0(\ell)$,
  the degree of $Z_h^{+}$ is $\leq \gamma+\tc(r_h)$ too).
By  \ref{PC},
\begin{equation}\label{eq:PCpol}
\mathrm{pc}(Z_h)=Z_h^{+}(1)=\sum_{\ell\geq -\tc}\ \max\big\{0, -N_\bc(\ell)\big\}
\end{equation}
for {\it any} lifting $\bc$ of $h=\sum_ic_ig_i$.
In this sum the bound $\ell\geq -\tc$ is really restrictive.

We consider the non-equivariant version, the projection of $Z^e\in \Z[[t^{1/\fro}]][H]$
into $\Z[[t^{1/\fro}]]$ too
$$Z_{ne}(t)=\sum_hZ_h(t)=\frac{(1-t^{1/|e|})^{d-2}}{
\prod_{i=1}^d\,(1-t^{1/(|e|\alpha_i)})}\in \Z[[t^{1/\fro}]].
$$
We can get its $Z^{+}_{ne}+Z^{-}_{ne}$ decomposition either by summation of
$Z^{+}_{h}$ and $Z^{-}_{h}$, or as follows. Write
\begin{equation}\label{eq:NE}
Z_{ne}(t)=\frac{1}{(1-t^{1/|e|})^{2}} \prod_{i=1}^d\, \frac{1-t^{1/|e|}}{
1-t^{1/(|e|\alpha_i)}}=\frac{1}{(1-t^{1/|e|})^{2}}
\sum_{0\leq x_i<\alpha_i\atop 0\leq i\leq d} t^{\frac{1}{|e|}\cdot S(x)},
\end{equation}
where $S(x):=\sum_i \frac{x_i}{\alpha_i}$. Then its decomposition into
$Z_{ne}^{+}(t)+Z_{ne}^{-}(t)$ is
\begin{equation}\label{eq:NE1}
Z_{ne}^{-}(t)=
\sum_{0\leq x_i<\alpha_i\atop 0\leq i\leq d} t^{\frac{1}{|e|}\cdot \{S(x)\}}\cdot
\Big(\frac{1}{(1-t^{1/|e|})^{2}}-\frac{\lfloor S(x)\rfloor}{(1-t^{1/|e|})}\Big)
\end{equation}
\begin{equation}\label{eq:NE2}
Z_{ne}^{+}(t)=
\sum_{0\leq x_i<\alpha_i\atop 0\leq i\leq d} t^{\frac{1}{|e|}\cdot \{S(x)\}}\cdot
\frac{t^{\frac{1}{|e|}\cdot \lfloor S(x)\rfloor}-\lfloor S(x)\rfloor
t^{\frac{1}{|e|}}+\lfloor S(x)\rfloor-1}{(1-t^{1/|e|})^{2}}.
\end{equation}
After dividing in $Z^+_{ne}(t)$ (or by L'Hospital rule), we get 
\begin{equation}\label{eq:NE3}
\mathrm{pc}(Z_{ne})=Z^+_{ne}(1)=\frac{1}{2}\cdot
\sum_{0\leq x_i<\alpha_i\atop 0\leq i\leq d} \lfloor S(x)\rfloor\cdot \lfloor S(x)-1\rfloor.
\end{equation}

\subsection{Analytic interpretations}\labelpar{ss:analy} \

\vspace{2mm}

Rational homology sphere negative definite
Seifert 3--manifolds  can be realized analytically as links of  weighted homogeneous singularities,
or by  equisingular deformations of weighted homogeneous singularities provided by splice
 quotient equations \cite{neumann.abel,nw-CIuac}.

Consider the smooth germ at the origin of  $\C^d$ with coordinate ring $\C\{z\}=\C\{z_1,\ldots, z_d\}$, where
$z_i$ corresponds to the $i^{\mathrm{th}}$ end. Then $H$ acts on it by the diagonal action
$h*z_i=\theta(g_i)(h)z_i$. Similarly, we can introduce a multidegree $deg(z_i)=E^*_i\in L'$,
hence the Poincar\'e series of $\C\{z\}$ associated with this multidegree is $\prod_i
(1-\bt^{E^*_i})^{-1}$. Moreover, considering the action of $H$ on it,
$\widetilde{Z}(\bt)=\prod_i
(1-[g_i]\bt^{E^*_i})^{-1}$ is the equivariant Poincar\'e series of $\C^d$, the invariant part
$\widetilde{Z}_0(\bt)$ being the Poincar\'e series of the corresponding quotient space $\C^d/H$.

In $\C^d$ one can consider the `splice equations' as follows. Consider a matrix $\{\lambda_{ij}\}_{ij}$
of full rank and of size $d\times(d-2)$. Then the equations
$\sum_{i=1}^d\lambda_{ij}z_i^{\alpha_i}=0$, for $j=1,\ldots, d-2$, determine in $\C^d$ an isolated
complete intersection singularity $(Y,o)$ on which the group $H$ acts as well.
Then $(X,o)=(Y,o)/H$ is a normal surface singularity with the topological type of the Seifert manifold
we started with. The equivariant Poincar\'e series of $(Y,o)$ is $Z(\bt)$ \cite{neumann.abel}.
For $(X,o)$ \cite{BN} proves the identity $P(\bt)=Z(\bt)$ mentioned in subsection \ref{FM},
hence $Z(\bt)$ is also the Poincar\'e series of the equivariant divisorial filtration associated with all
the vertices.

Theorem~\ref{th:REST} reduces the filtration to the $\Z$-filtration: the
divisorial filtration associated with the central vertex.  In the weighted homogeneous case this
filtration is also induces by the weighted homogeneous equations. Then,
$Z^{/H}(t)$ is the Poincar\'e series of $\C^d/H$,
$Z(t)$  is the equivariant Poincar\'e series of $Y$, hence $Z_0(t)$ is the Poincar\'e series of $X$,
cf. \cite{Do,neumann.abel,pinkham}.

By \ref{FM}, $\{\mathrm{pc}(Z_h)\}_{h\in H}$  are the equivariant
geometric genera of the universal abelian cover $Y$ of $X$, hence
$\mathrm{pc}(Z_0)$ and $\mathrm{pc}(Z_{ne})$ are the geometric genera
of $X$ and $Y$ respectively, cf. \cite{coho3}.

\subsection{Seiberg--Witten theoretical interpretations}\labelpar{ss:onenodeSW} \

\vspace{2mm}

Fix $h\in H$. Then, for any lifting $\sum_ic_ig_i$ of $h$, Corollary~\ref{cor:4.1} and equation
\ref{eq:PCpol} give
\begin{equation}\label{eq:pcrh}
  \mathrm{pc}(Z_h)=\sum_{\ell\geq -\tc}\ \max\big\{0, -N_\bc(\ell)\big\}=
-\frsw_{-h*\sigma_{can}}(M)
-\frac{(K+2r_h)^2+|\cV|}{8}.\end{equation}
Recall  that
$\sum_h\frsw_{-h*\sigma_{can}}(M)$ is the {\it Casson--Walker invariant}
$\lambda(M)$. Hence, for the non-equvariant version   we get
\begin{equation}\label{eq:pcrhne}
  \mathrm{pc}(Z_{ne})=\frac{1}{2}\cdot
\sum_{0\leq x_i<\alpha_i\atop 0\leq i\leq d} \lfloor S(x)\rfloor\cdot
\lfloor S(x)-1\rfloor=
-\lambda(M)-\frd\cdot\frac{K^2+|\cV|}{8}+\sum_h\chi(r_h).\end{equation}
For explicit formulae of $\lambda(M)$ and $K^2+|\cV|$ in terms of
Seifert invariants see e.g. \cite[2.6]{NN2}).

\begin{remark} (\ref{eq:pcrh})
can be compared with a known formulae of the Seiberg--Witten
invariants involving the representative $s_h$. This will also lead us to an
expression for $\chi(r_h)-\chi(r_s)$ in terms of $N_{{\bf c}}(\ell)$.
Let $\bc(s_h)=(c_0,\ldots,c_d)$ be the coefficients of $s_{h,red}$,
cf. \ref{ss:seiferjel}. The set of all reduced coefficients $\bc(s_h)$,
when $h$ runs in $H$, is characterized
in \cite[11.5]{NOSZ} by the inequalities
\begin{equation}\label{eq:shcar}
\left\{ \begin{array}{l}
c_0\geq 0, \ \ \alpha_i>c_i\geq 0 \ \ (1\leq i\leq d) \\
N_\bc(\ell)\leq 0 \ \ \mbox{for any $\ell<0$}.\end{array}\right.
\end{equation}
Moreover, for this special lifting $\bc(s_h)$ of $h$, in \cite[\S 11]{NOSZ} is proved
\begin{equation}\label{eq:pcsh}
\sum_{\ell\geq 0}\ \max\big\{0, -N_{\bc(s_h)}(\ell)\big\}=
-\frsw_{-h*\sigma_{can}}(M)
-\frac{(K+2s_h)^2+|\cV|}{8}.\end{equation}
Using the discussion from the end of \ref{ss:seiferjel}, this can be rewritten for {\it any} lifting
$\bc$ of $h$ as
\begin{equation}\label{eq:pcsh2}
\sum_{\ell\geq -\tc+\lfloor \tc(s_h)\rfloor}\ \ \max\big\{0, -N_{\bc}(\ell)\big\}=
-\frsw_{-h*\sigma_{can}}(M)
-\frac{(K+2s_h)^2+|\cV|}{8}.\end{equation}
This compared with (\ref{eq:pcrh}) gives for any lifting
$\bc$ of $h$
\begin{equation}\label{eq:pcsh3}
\sum_{-\tc+\lfloor \tc(s_h)\rfloor  >\ell\geq -\tc}\ \ \max\big\{0, -N_{\bc}(\ell)\big\}=
\chi(r_h)-\chi(s_h).\end{equation}
\end{remark}

\begin{example}\labelpar{ex:sh}
The sum in (\ref{eq:pcsh3}), in general, can be non-zero. This happens,
for example,  in the case of the link of a rational singularity whose universal abelian cover is not rational. Here is a concrete example, cf. \cite[4.5.4]{trieste}:
take the Seifert manifold with $b=-2$ and
three legs, all of them with Seifert invariants $(\alpha_i,\omega_i)=(3,1)$.
For  $ h=\sum_{i=1}^3g_i$ one has $s_h=\sum_{i=1}^3E_i^*$, the $E_0$-coefficient of $s_h$ is 1,
$r_h=s_h-E_0$, and $\chi(s_h)=0$, $\chi(r_h)=1$.
\end{example}

\subsection{Ehrhart theoretical interpretations}\labelpar{ss:Ehr}\

\vspace{2mm}

We fix $h\in H$ as above and {\it we assume that $\tc\in[0,1)$}. Note that $Z_h(t)$ has the form
$t^{\tc}\sum_{\ell\geq 0}p_\ell t^\ell$; here the exponents $\{\tc+\ell\}_{\ell\geq 0}$
are the possible $E_0$--coordinates of the elements $(r_h+L)\cap \calS'$.

Let us compute the counting  function for $Z_h$. If $S(t)=\sum_r p_rt^r$ is a series,
we write $Q(S)(r')=\sum_{r<r'}p_r$, for $r'\in\Q_{\geq 0}$.
\begin{lemma}\label{lemma:P} For any $n\in\N_{\geq 0}$ one has the following facts.

(a) \ $Q(Z_h)(n)=Q(Z_h)(n+\tc)$.

(b) \ $Q(Z_h^{+})(n)$ is a step function (hence piecewise polynomial)
with $$Q(Z_h^{+})(n)={\mathrm{pc}}(Z_h) \ \ \ \mbox{for $n> \mathrm{deg}(Z_h^{+})$}.$$

(c) \ $Q(Z_h^{-})(n)$ is a quasipolynomial:
\begin{equation}\label{eq:PSp2}
Q(Z^{-}_h)(n)=(1-e\tc)n-e\cdot \frac{n(n-1)}{2}
-\sum_{i=1}^d\sum_{r_i=0}^{\alpha_i-1}\Big\{\frac{c_i-\omega_ir_i}{\alpha_i}\Big\}\,
\Big\lceil \frac{n-r_i}{\alpha_i}\Big\rceil\end{equation}
\begin{equation*}
=-\frac{en^2}{2}+\frac{en}{2}(\gamma+1-2\tc)
-\sum_{i=1}^d\sum_{r_i=0}^{\alpha_i-1}\Big\{\frac{c_i-\omega_ir_i}{\alpha_i}\Big\}\,
\Big(\Big\{\frac{r_i-n}{\alpha_i}\Big\}-\frac{r_i}{\alpha_i}\Big).
\end{equation*}
\end{lemma}
In particular, if $n=m\alpha$  for  $m\in \Z$, and $n>\mathrm{deg}(Z^{+}_h)$,
then the double sum is zero, hence
\begin{equation}\label{eq:PP}
Q(Z_h)(n)=-\frac{en^2}{2}+\frac{en}{2}(\gamma+1-2\tc)+\mathrm{pc}(Z_h).
\end{equation}
This is compatible with the expression provided by Theorem \ref{th:JEMS} and the Reduction theorem
\ref{th:REST}. Indeed, let us fix any chamber $\calC$ such that $int(\calC\cap \calS')\not=\emptyset$, and
$\calC$ contains the ray $\calR=\R_{\geq 0}\cdot E^*_0$. Since the numerator of $f(\bt)$ is
$(1-\bt^{E^*_0})^{d-2}$, all the $b_k$--vectors belong to $\calR$. In particular, $\cap_k(b_k+\calC)$ intersects
$\calR$ along a semi-line $\calR^{\geq c}=\R_{\geq const}\cdot E^*_0$ of $\calR$. Since $Q_h(l')$
is quasipolynomial on $\cap_k(b_k+\calC)$, cf. (\ref{ex:qbar2}), and certain restriction of it is determined
by  (\ref{eq:SUM}) whose right hand side is a quasipolynomial too, we obtain that the identity
(\ref{eq:SUM}) is valid on $\calR^{\geq c}$ as well.

Recall that for any $h\in H$ and $l'\in L'$ we have a unique $q_{l',h}\in\square $ with
$l'+q_{l',h}\in r_h+L$. Hence we get
\begin{equation}\label{eq:QPpol}
Q_h(l')=
-\frsw_{-h*\sigma_{can}}(M)-\frac{(K+2l'+2q_{l',h})^2+|\cV|}{8}\ \ \ \ (l'\in \calR^{\geq c}).
\end{equation}
The term $q_{l',h}$ is responsable for the  non-polynomial behavior. Nevertheless, if we assume that
 $l'=m\fro E^*_0\in\calR^{\geq c}\cap L$, $m\in \Z$, then $q_{l',h}=r_h$, hence by  (\ref{eq:pcrh})
 \begin{equation}\label{eq:PPP}
Q_h(l')=-\frac{(l',l'+K+2r_h)}{2}+\mathrm{pc}(Z_h).\end{equation}
By the Reduction theorem \ref{th:REST} $Q_h(l')$ from (\ref{eq:PPP}) depends only on
the $E_0$-coefficient of $l'=m\fro E_0^*$, which  is exactly $m\alpha$. One sees that in fact
(\ref{eq:PPP}) agrees with (\ref{eq:PP}) if we set $n=m\alpha$.

\vspace{2mm}

The non-equivariant version can be obtained by summation of (\ref{eq:PP}). For this we need
$\sum_h\tc(r_h)$. Consider the group homomorphism $\varphi:H\to \Q/\Z$ given by
$h\mapsto [\tc(r_h)]$, or $ [E^*_v]\mapsto [-(E^*_0,E^*_v)]$. Its image is generated by the
classes of $1/(\alpha_i|e|)$, hence its order is $\fro$. Hence, $\tc(r_h)$ vanishes exactly
$\frd/\fro$ times (whenever $h\in \ker(\varphi)$). In all other cases $\tc(r_h)\not=0$, and
$\tc(r_h)+\tc(r_{-h})=1$. In particular, $2\sum_h\tc(r_h)=\frd-\frd/\fro$.
Therefore, the summation of (\ref{eq:PP}) provides
\begin{equation}\label{eq:PPne}
Q(Z_{ne})(n)=-\frac{\frd en^2}{2}+\frac{\frd en}{2}(\gamma+\frac{1}{\fro})+\mathrm{pc}(Z_{ne})
\ \ \ \ \mbox{(for $n\in\alpha\Z$)}.
\end{equation}

Next, we will identify the coefficients of (\ref{eq:PP}) and (\ref{eq:PPne})
with the first three coefficient of the
Ehrhart quasipolynomial $\calL^\calC_h(\calT)$ via the identity (\ref{ex:qbar2}).

For simplicity we will assume that $\fro=1$, in particular all the $b_k$--vectors belong to $L$.

If $l'\in\calR$, then by Reduction theorem  the counting function $\calL^\calC_h(\calT,l')$
of the polytope $P^{(l')}$ depends only on the $E_0$--coefficient of $l'$; let us denote this coefficient by $l'_0$.

Hence,  this $\calL^\calC_h(\calT,l'_0)$ is the Ehrhart quasipolynomial of the $d$-dimensional
simplicial polytope,
being its $h$--class counting function. Via (\ref{eq:DET}) the definition (\ref{eq:Pv}) of this polytope becames
\begin{equation}\label{eq:P0}
P_0=\{(x_1,\ldots,x_d)\in (\R_{\geq 0})^d\,:\, \sum_i\frac{x_i}{|e|\alpha_i}< l'_0\}.
\end{equation}
Let
\begin{equation}\label{eq:fra}
\calL^\calC_h(\calT,l'_0)=\sum_{j=0}^d \, \fra_{h,j}(l'_0)\cdot \frac{(l'_0)^j}{j!}
\end{equation}
be the coefficients of the Ehrhart quasipolynomial: each $\fra_{h,j}(l'_0)$ is a periodic function in $l'_0$ and is normalized by $1/j!$.
Since the numerator of $f$ is $(1-t^{1/|e|})^{d-2}$, by
(\ref{ex:qbar2}) we obtain for $l'\in\calR$
\begin{equation}\label{eq:fraP}
Q_h(l')=\sum_{j=0}^d\, \fra_{h,j}(l'_0)\cdot \frac{1}{j!}\ \sum_{k=0}^{d-2} (-1)^k\binom{d-2}{k}\Big(l'_0-\frac{k}{|e|}\Big)^j.
\end{equation}
This equals the expression
(\ref{eq:QPpol}) above. The non-polynomial behavior of these
two expressions indicate that $\fra_j(l'_0)$ is indeed nonconstant periodic, and can be determined explicitly.

Since we are interested primarily in the Seiberg--Witten invariant, namely in $\mathrm{pc}(Z_h)$, we perform
this explicit identification only via the expressions (\ref{eq:PP}) and (\ref{eq:PPP}). Hence, similarly
as in these cases, we take $l'=m\fro E_0^*\in\calR^{\geq c}\cap L$, and we identify (\ref{eq:PP})
with (\ref{eq:fraP}) evaluated for $l'$, whose $E_0$--coefficient is $l'_0=m\alpha=n$.
In this case $\fra_{h,j}(n)$ is a {\it constant}, denoted by $\fra_{h,j}$, and
\begin{equation}\label{eq:fraPP}
-\frac{en^2}{2}+\frac{ne}{2}(\gamma+1-2\tc)+\mathrm{pc}(Z_h)=
\sum_{j=0}^d \fra_{h,j}\cdot \frac{1}{j!}\ \sum_{k=0}^{d-2} (-1)^k\binom{d-2}{k}\Big(n-\frac{k}{|e|}\Big)^j.
\end{equation}
Here it is helpful the combinatorial expression (see e.g. \cite[p. 7-8]{PSz})
\begin{equation}\label{eq:PSz}
\frac{(-1)^d}{(d-2)!}\cdot \sum_{k=0}^{d-2}(-1)^k\binom{d-2}{k}k^j=\left\{
\begin{array}{ll}
0 & \mbox{if $j<d-2$},\\
1 & \mbox{if $j=d-2$},\\
(d-2)(d-1)/2 & \mbox{if $j=d-1$},\\
(d-2)(d-1)d(3d-5)/24 & \mbox{if $j=d$}.\end{array}\right.
\end{equation}
We obtain
\begin{equation}\label{eq:fra3a}
\begin{array}{rl}\vspace{1mm}
\frac{\fra_{h,d}}{|e|^{d}}=&
\frac{1}{|e|}\\\vspace{1mm}
\frac{\fra_{h,d-1}}{|e|^{d-1}}=&
\frac{d-2}{2|e|}-\frac{1}{2}(\gamma+1-2\tc)
\\ \vspace{1mm}
\frac{\fra_{h,d-2}}{|e|^{d-2}}=&
\mathrm{pc}(Z_h)+\frac{(d-2)(3d-7)}{24|e|}-\frac{d-2}{4}(\gamma+1-2\tc)
.
\end{array}
\end{equation}
In particular, the $\fra_{h,d-2}$ can be identified (up to
 `easy' extra terms) with  ${\rm pc}(Z_h)$ (with  analytical interpretation
$\dim (H^1(\widetilde{Y},\cO_{\widetilde{Y}})_{\theta(h)})$ and Seiberg--Witten
theoretical interpretation (\ref{eq:pcrh})).
The first coefficients can also be identified with the equivariant volume of $P_0$,
(a fact already known in the non-equivariant cases). Usually (in the non-equivariant case, and when
we count the points of all the facets)
the second coefficient can be related with the volumes of the facets.
Here we eliminate from this count some of the facets, and we are in the equivariant situation as well.
The expression of the  second coefficient is also a novelty of the present article
(to the best of author's knowledge).

In the non-equivariant case, if $\sum_{j=0}^d \fra_{j}\frac{n^j}{j!}$ is the classical Ehrhart polynomial
of $P_0$, then
\begin{equation}\label{eq:fra3b}
\begin{array}{rl}\vspace{1mm}
\frac{\fra_{d}}{|e|^d}=&
\prod_i\alpha_i\\\vspace{1mm}
\frac{\fra_{d-1}}{|e|^{d-1}}=&
\prod_i\alpha_i \cdot\left(-\frac{1}{\alpha}+\sum_i\frac{1}{\alpha_i}\right)/2
\\ \vspace{1mm}
\frac{\fra_{d-2}}{|e|^{d-2}}=&
\prod_i\alpha_i \Big(
\frac{\mathrm{pc}(Z_{ne})}{\prod_i\alpha_i}-\frac{(d-2)(3d-5)}{24}+\frac{d-2}{4}(-\frac{1}{\alpha}+
\sum_i\frac{1}{\alpha_i})\Big).
\end{array}
\end{equation}
In this non-equivariant case the identities (\ref{eq:fra3b}) are valid even without
the assumption $\fro=1$ by the same proof.

The formulae in (\ref{eq:fra3a}) and (\ref{eq:fra3b}) can be further simplified if we replace $P_0$
by $|e|P_0$, or if we substitute in the Ehrhart polynomial the new variable $\lambda:= |e|l_0'$ instead of $l_0'$; cf. Section \ref{s:Last}.

\section{The two--node case}\labelpar{s:TN}

\subsection{Notations and the group $H$}\labelpar{ss:TNC}

We consider the following graph $\Gamma$:

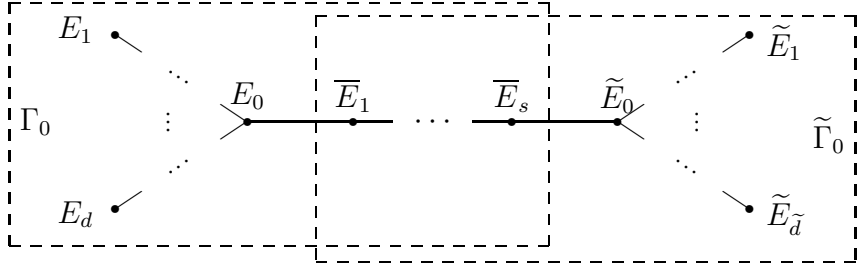
\vspace{1cm}
\begin{center}
\begin{picture}(200,60)(80,15)
\put(100,40){\circle*{3}}
\put(140,40){\circle*{3}}
\put(100,40){\line(1,0){55}}
\put(171,40){\makebox(0,0){$\cdots$}}
\put(200,40){\circle*{3}}
\put(240,40){\circle*{3}}
\put(185,40){\line(1,0){55}}
\put(50,73){\circle*{3}}
\put(100,40){\line(-3,2){10}}
\put(50,73){\line(3,-2){10}}
\multiput(77,55)(-3,2){3}%
{\circle*{1}}

\multiput(70,43)(0,-3){3}%
{\circle*{1}}

\put(50,7){\circle*{3}}
\put(100,40){\line(-3,-2){10}}
\put(50,7){\line(3,2){10}}
\multiput(77,25)(-3,-2){3}%
{\circle*{1}}

\put(290,73){\circle*{3}}
\put(240,40){\line(3,2){10}}
\put(290,73){\line(-3,-2){10}}
\multiput(263,55)(3,2){3}%
{\circle*{1}}

\multiput(270,43)(0,-3){3}%
{\circle*{1}}

\put(290,7){\circle*{3}}
\put(240,40){\line(3,-2){10}}
\put(290,7){\line(-3,2){10}}
\multiput(263,25)(3,-2){3}%
{\circle*{1}}

\put(100,50){\makebox(0,0){$E_0$}}
\put(140,50){\makebox(0,0){$\overline E_1$}}
\put(200,50){\makebox(0,0){$\overline E_s$}}
\put(240,50){\makebox(0,0){$\widetilde E_0$}}
\put(35,75){\makebox(0,0){$E_1$}}\put(35,5){\makebox(0,0){$E_d$}}
\put(303,70){\makebox(0,0){$\widetilde E_1$}}
\put(303,5){\makebox(0,0){$\widetilde E_{\widetilde{d}}$}}

\multiput(10,85)(0,-8){12}{\line(0,-1){4}}
\multiput(10,85)(8,0){26}{\line(1,0){4}}
\multiput(214,85)(0,-8){12}{\line(0,-1){4}}
\multiput(214,-7)(-8,0){26}{\line(-1,0){4}}
\put(20,40){\makebox(0,0){$\Gamma_0$}}

\multiput(330,80)(-8,0){26}{\line(-1,0){4}}
\multiput(330,80)(0,-8){12}{\line(0,-1){4}}
\multiput(126,80)(0,-8){12}{\line(0,-1){4}}
\multiput(126,-13)(8,0){26}{\line(1,0){4}}
\put(320,35){\makebox(0,0){$\widetilde\Gamma_0$}}

\end{picture}
\end{center}

\vspace{1.5cm}

The nodes $E_0$ and $\widetilde{E}_0$ have decorations $b_0$ and $\widetilde{b}_0$ respectively.
Similarly as in the one--node case, we codify the decorations of maximal chains by
continued fraction expansions. In fact, it is convenient to consider the two maximal
star--shaped graphs $\Gamma_0$ and $\widetilde{\Gamma}_0$, and the corresponding normalized
Seifert invariants of their legs. Hence, let the
normalized Seifert invariants of the legs with ends $E_i$ ($1\leq i\leq d$) be  $(\alpha_i,\omega_i)$, while of the legs with ends  $\widetilde{E}_j$ ($1\leq j\leq \widetilde{d}$)
be $(\widetilde\alpha_j,\widetilde\omega_j)$.

The chain connecting the nodes, viewed in $\Gamma_0$ has normalized
Seifert invariants $(\alpha_0,\omega_0)$, while
viewed as a leg in $\widetilde{\Gamma}_0$, it has Seifert invariants
$(\alpha_0,\widetilde{\omega}_0)$. One has $\omega_0 \widetilde \omega_0=\alpha_0 \tau +1$.
Clearly, $\alpha_0$ is the determinant of the chain, and

\begin{picture}(400,40)(30,5)
\put(40,20){\makebox(0,0)[l]{$\omega_0:= \det(\hspace{23mm})$}}
\put(100,20){\circle*{3}}
\put(150,20){\circle*{3}}
\put(100,20){\line(1,0){10}}
\put(150,20){\line(-1,0){10}}
\put(125,20){\makebox(0,0){$\cdots$}}
\put(100,30){\makebox(0,0){$\overline E_2$}}
\put(150,30){\makebox(0,0){$\overline E_s$}}
\put(190,20){\makebox(0,0)[l]{$\widetilde{\omega}_0:= \det(\hspace{25mm})$}}
\put(250,20){\circle*{3}}
\put(300,20){\circle*{3}}
\put(250,20){\line(1,0){10}}
\put(300,20){\line(-1,0){10}}
\put(275,20){\makebox(0,0){$\cdots$}}
\put(250,30){\makebox(0,0){$\overline E_1$}}
\put(300,30){\makebox(0,0){$\overline E_{s-1}$}}

\put(346,20){\makebox(0,0)[l]{$\tau:= \det(\hspace{25mm}).$}}
\put(400,20){\circle*{3}}
\put(450,20){\circle*{3}}
\put(400,20){\line(1,0){10}}
\put(450,20){\line(-1,0){10}}
\put(425,20){\makebox(0,0){$\cdots$}}
\put(400,30){\makebox(0,0){$\overline E_2$}}
\put(450,30){\makebox(0,0){$\overline E_{s-1}$}}
\end{picture}

We denote the orbifold Euler numbers of the star--shaped subgraphs $\Gamma_0$
and $\widetilde\Gamma_0$ by
$$e=b_0+ \frac{\omega_0}{\alpha_0}+\sum_i \frac{\omega_i}{\alpha_i}  \ \ \ \mbox{and} \ \ \
\widetilde{e}=\widetilde b_0+
\frac{\widetilde\omega_0}{\alpha_0}+
\sum_j \frac{\widetilde\omega_j}{\widetilde\alpha_j}.$$
Consider the  {\it orbifold intersection matrix}
$I^{orb}=\left(
\begin{array}{lr}
 e & 1/\alpha_0 \\
 1/\alpha_0 & \widetilde{e}
\end{array}
\right)$, cf. \cite[4.1.4]{BN07}.

\noindent Then, the
negative definiteness of $I$ (or $\Gamma$)  implies
that  $I^{orb}$  is negative definite too, hence
$$\varepsilon:=\det I^{orb}=e \widetilde e - \frac{1}{\alpha_0^2}>0.$$
Then the determinant of the graph is $\det(\Gamma)=\det(-I)=\varepsilon\cdot \alpha_0 \prod_i \alpha_i \prod_j \widetilde{\alpha}_j$, cf. \cite{BN07}.

Using (\ref{eq:DETsgr}) we have the following intersection number of the
dual base elements:
\begin{equation}\label{eq:DET2node}
\begin{array}{rl}
(E^*_0)^2=\frac{\widetilde e}{\varepsilon}; \ \
(\widetilde E^*_0)^2=\frac{e}{\varepsilon}; \ \
(E^*_0,\widetilde E^*_0)=-\frac{1}{\alpha_0 \varepsilon}; \ \
(E^*_0,E^*_i)=\frac{\widetilde e}{\alpha_i \varepsilon}; \\
(E^*_0,\widetilde E^*_j)=-\frac{1}{\alpha_0 \widetilde \alpha_j \varepsilon}; \ \
(\widetilde E^*_0,E^*_i)=-\frac{1}{\alpha_0 \alpha_i \varepsilon}; \ \
(\widetilde E^*_0,\widetilde E^*_j)=\frac{e}{\widetilde \alpha_j \varepsilon}.
\end{array}
\end{equation}
Similarly as in \ref{ex:lens} or \ref{ss:seiferjel}, we can write
$n^i_{k_1,k_2}$, \  $\widetilde{n}^j_{k_1,k_2}$ resp. $\overline{n}_{k_1,k_2}$
 for the the determinant of the
sub--chains  of the `left' $i^{th}$ leg, `right'  $j^{th}$ leg and connecting chain
 connecting the vertices $v_{k_1}$ and $v_{k_2}$. Let $\nu_i$ and $\tilde{\nu}_j$ be the
 number of vertices in the legs, cf. \ref{ss:seiferjel}. Then (with the standard notations, where
 $E_{i\ell}$ and $\widetilde{E}_{j\ell}$ are the vertices of the legs)
 one has the following slightly technical Lemma, but whose proof is standard based on the arithmetical
 properties of continued fractions:
\begin{lemma}\labelpar{lem:legeq}
(a) \
  $E^*_{i\ell}=n^i_{\ell+1,\nu_i} E^*_i+ \sum_{\ell<r\leq \nu_i}\frac{n^i_{1,\ell-1} n^i_{r+1,\nu_i}-
n^i_{1,r-1} n^i_{\ell+1,\nu_i}}{\alpha_i}E_{ir}$\ \ for any $1\leq \ell<\nu_i$.

(There is a similar formula for $\widetilde E^*_{j\ell}$.)

  (b) \  $\overline E^*_k=\overline{n}_{1,k-1}\overline E^*_1-\overline{n}_{2,k-1}E^*_0+
  \sum_{1\leq r<k} \frac{\overline{n}_{1,r-1} \overline{n}_{k+1,s}-\overline{n}_{1,k-1}
  \overline{n}_{r+1,s}}
{\alpha_0}\overline E_r \,$, for $1<k\leq s$.

(This is true even for $k=s+1$ with the identification $\overline{E}^*_{k+1}=\widetilde{E}^*_0$.)

\end{lemma}

Next, we give a presentation of $H=L'/ L$. Set $g_i:=[E^*_i]$ ($1\leq i\leq d$),
$\widetilde g_j:=[\widetilde E^*_j]$ ($1\leq j\leq \widetilde{d}$), $g_0:=[E^*_0]$ and
$\widetilde g_0:=[\widetilde E^*_0]$. Moreover we need to choose an additional generator
corresponding  to a vertex
sitting on the connecting chain: we choose $\overline{g}:=[\overline E^*_1]$ (this motivates
the choice in Lemma \ref{lem:legeq})(b) too).
The above lemma implies
\begin{equation}\label{eq:dual2nodes}
[E^*_{i\ell}]=n^i_{\ell+1,\nu_i}g_i, \ \ \ [\widetilde
E^*_{j\ell}]=\widetilde{n}^j_{\ell+1,\widetilde{\nu}_j}\widetilde g_j \ \ \
\mbox{and} \ \ \  [\overline{E}^*_k]=\overline{n}_{1,k-1}\overline{g}-\overline{n}_{2,k-1}g_0;
\end{equation}
and similar arguments as in the star--shaped case provides the following presentation for $H$
\begin{eqnarray}\label{eq:2Hpres}
{\textstyle  H={\mathrm{ab}}\langle \, g_0,\widetilde g_0, g_i, \widetilde g_j, \overline{g}
\,|\, g_0=\alpha_i\cdot g_i; \
\widetilde g_0 = \widetilde \alpha_j \cdot \widetilde g_j; \ \alpha_0 \cdot
\overline{g}=\omega_0 \cdot g_0 +\widetilde g_0;} \\ \nonumber {\textstyle
-\overline{g}-b_0\cdot g_0=\sum_i \omega_i \cdot g_i; \ -\widetilde \omega_0 \cdot
\overline{g} + \tau \cdot g_0
-\widetilde b_0 \cdot \widetilde g_0= \sum_j \widetilde \omega_j \cdot \widetilde g_j \rangle.}
\end{eqnarray}
Moreover,  for any  $l' \in L'$,
$$\textstyle{l'=c_0 E^*_0 + \widetilde c_0 \widetilde E^*_0+\sum_k \overline{c}_k
\overline{E}^*_k+\sum_{i,\ell} c_{i\ell}E^*_{i\ell}+\sum_{j\ell}
\widetilde c_{j\ell}\widetilde E^*_{j\ell},}$$
if we define its {\it reduced transform}  $l'_{red}$ by
\begin{equation*}
({c}_0-\sum_{k>1}
\overline{n}_{2,k-1}\overline{c}_k){E}^*_0 +\widetilde c_0 \widetilde E^*_0 + (\overline{c}_1+\sum_{k>1}\overline{n}_{1,k-1}\overline{c}_k)\overline{E}^*_1+
\sum_{i,\ell}
c_{i\ell}n^i_{\ell+1,\nu_i}E^*_i+ \sum_{j,\ell} \widetilde c_{j\ell}\widetilde n^j_{\ell+1,\widetilde{\nu}_j}\widetilde E^*_j,
\end{equation*}
then, by Lemma \ref{lem:legeq},  $[l']=[l'_{red}]$ in $H$.
Moreover,  if for any $l'\in L'$ we distinguish the $E_0$ and $\widetilde{E}_0$ coefficients,
that is,
we set $c(l'):=-(E^*_0,l')$ and $\widetilde c(l'):=-(\widetilde E^*_0,l')$,
then $c(l')=c(l'_{red})$
and $\widetilde c(l')=\widetilde c(l'_{red})$ as well. Lemma \ref{lem:legeq}(b) (applied for
$k=s+1$) provide these coefficients for $\overline{E}_1$:
\begin{equation}\label{eq:E1bar}
(\overline{E}_1^*,E_0^*)=\frac{1}{\varepsilon\alpha_0}\big(\omega_0\widetilde{e}-\frac{1}{\alpha_0}\big),
\ \ \ \
(\overline{E}_1^*,\widetilde{E}_0^*)=\frac{1}{\varepsilon\alpha_0}
\big(e-\frac{\omega_0}{\alpha_0}\big).
\end{equation}
We will use the coefficients $\bc=(c_0,\widetilde c_0,
\overline c, c_i,\widetilde c_j)$ to write an element
$l'_{red}=c_0E^*_0+ \widetilde{c}_0\widetilde E^*_0+ \overline{c}\overline{E}^*_1+
\sum_ic_iE^*_i+\sum_j\widetilde{c}_j\widetilde E^*_j$. Then (\ref{eq:DET2node}) and (\ref{eq:E1bar})
imply that
\begin{equation}\label{eq:AA}
\left( \begin{array}{c}  c\\   \widetilde{c} \end{array} \right)=
\left( \begin{array}{c}  c(l'_{red})\\   \widetilde{c}(l'_{red})
 \end{array} \right)
 = (-I^{orb})^{-1}\cdot \left( \begin{array}{c}   A\\
\widetilde A   \end{array} \right)
=\frac{1}{\varepsilon}
\left( \begin{array}{lr}
 -\widetilde{e} & 1/\alpha_0 \\
 1/\alpha_0 & -e
\end{array} \right)
\cdot \left( \begin{array}{c} A\\
\widetilde A \end{array} \right),
\end{equation}
where
\begin{equation*}
 A:=c_0+\sum_i \frac{c_i}{\alpha_i}+\frac{\omega_0}{\alpha_0} \overline{c}, \ \ \ \ \ \ \
 \widetilde A:=\widetilde c_0+\sum_j \frac{\widetilde c_j}{\widetilde \alpha_j}+\frac{1}{\alpha_0} \overline{c}.
\end{equation*}
Therefore, any
$h\in H$ has a lift of type  $l'_{h,red}$. Although the corresponding coefficients  $c$ and
$\widetilde c$ depend on the lift, by adding $\pm E_0$ and $\pm \widetilde E_0$ to $l'_{h,red}$ we can achieve
$c,\widetilde c \in [0,1)$, and these values
are uniquely determined by $h$. For example, the reduced  transform $(r_h)_{red}$ of
$r_h$ satisfies  $c((r_h)_{red})=c(r_h)\in[0,1) $ and $\widetilde c((r_h)_{red})=\widetilde c(r_h)
\in[0,1)$ since $r_h \in \square$.

As we will see, for different elements of $h\in H$,
we have to shift the rank two lattices  by vectors of type $(c,\widetilde{c})$, hence the vectors
$(c,\widetilde c)$ will play a crucial role later.

\subsection{The function Z} If we wish to compute the periodic constant of $Z^e(\bt)$, by Theorem
 \ref{th:REST} we can eliminate all the variables of
$Z^e(\bt)$ except  the variables of the nodes; these remaining two variables
are denoted by $(t,\wtt)$.
Therefore the equivariant form of reciprocal
of the denominator is
\begin{eqnarray*}
Z^{/H}(t,\wtt) &=& \prod_{i}\,\big(1-t^{-(E^*_{i}, E^*_0)}\wtt^{-(E^*_{i}, \widetilde E^*_0)}
[g_i]\big)^{-1}\cdot
\prod_{j}\,\big(1-t^{-(\widetilde E^*_{j}, E^*_0)}\wtt^{-(\widetilde E^*_{j},
\widetilde E^*_0)}[\widetilde g_j]
\big)^{-1}\\
&=& \sum_{x_i,\widetilde x_j\geq 0}\, t^{\,\frac{-\widetilde e}{\varepsilon}\sum_i
\frac{x_i}{\alpha_i} + \frac{1}{\alpha_0 \varepsilon} \sum_j \frac{\widetilde x_j}{\widetilde \alpha_j}}
\,\, \wtt^{\ \frac{1}{\alpha_0 \varepsilon}\sum_i
\frac{x_i}{\alpha_i} + \frac{-e}{\varepsilon} \sum_j \frac{\widetilde x_j}
{\widetilde \alpha_j}}\ \big[\,\textstyle{\sum _i} x_ig_i
+\textstyle{\sum_j} \widetilde x_j \widetilde g_j\big].
\end{eqnarray*}
We fix a lift $c_0 E^*_0+\widetilde c_0 \widetilde E^*_0 +\overline c\, \overline{E}^*_1+
\sum_i c_i E^*_i+\sum_j \widetilde c_j \widetilde E^*_j$ of $h$.  Then
the class of $\sum _i x_iE^*_i+\sum_j \widetilde x_j \widetilde E^*_j$
equals $h$ if and only if its  difference with
the lift is a linear combination of the relation in \ref{eq:2Hpres}. In other words,
if  there exist $\ell_0,\widetilde \ell_0, \overline \ell,\ell_i,
\widetilde \ell_j \in \Z$ such that
$$\begin{array}{lrlllrll}
(a) \ & -c_0&= \sum_i \ell_i-b_0 \ell_0 +\tau \widetilde \ell_0 +\omega_0 \overline \ell \hspace{5mm}&
(c) \ & x_i-c_i&= -\omega_i \ell_0-\alpha_i \ell_i & \ (i=1,\ldots, d)
\\
(b) \ & -\widetilde c_0&= \sum_j \widetilde \ell_j-\widetilde b_0 \widetilde\ell_0 +\overline \ell &
(d) \ & \widetilde x_j-\widetilde c_j&= -\widetilde \omega_j \widetilde\ell_0-\widetilde\alpha_j
\widetilde\ell_j & \ (j=1,\ldots, \widetilde d)
\\
(e) \ & -\overline c&= -\ell_0-\widetilde\omega_0 \widetilde\ell_0-\alpha_0 \overline \ell.
\end{array}$$
From (e) we deduce that
\begin{equation}\label{eq:CONGR}
\ell_0+\widetilde\omega_0 \widetilde\ell_0\equiv \overline c \,(\mbox{mod} \,\alpha_0).
\end{equation}
Since $x_i,\widetilde x_j\geq 0$, (c) and (d) implies $\frac{c_i-\omega_i\ell_0}{\alpha_i}\geq \ell_i$,
$\frac{\widetilde c_j-\widetilde\omega_j \widetilde\ell_0}{\widetilde\alpha_j}\geq \widetilde\ell_j$.
Recall also that $\omega_0 \widetilde \omega_0=\alpha_0 \tau +1$.
Therefore if we set $m_i:=\lfloor \frac{c_i-\omega_i\ell_0}{\alpha_i} \rfloor-\ell_i$ and $\widetilde m_j:=
\lfloor \frac{\widetilde c_j-\widetilde\omega_j \widetilde\ell_0}{\widetilde\alpha_j}\rfloor-\widetilde \ell_j$
 non-negative integers then the number of the realization of $h$ in the form $\sum_i x_ig_i+\sum_j\widetilde x_j\widetilde g_j$
is determined by the number of non-negative  integral $(d+\widetilde d)$-tuples $(m_i,\widetilde m_j)$
satisfying
$$\begin{array}{lrlr}
  & N_\bc(\ell_0,\widetilde \ell_0)&:= c_0+\frac{\omega_0}{\alpha_0}\overline c -(b_0+\frac{\omega_0}{\alpha_0})
\ell_0- \frac{1}{\alpha_0}\widetilde\ell_0 + \sum_i \lfloor \frac{c_i-\omega_i\ell_0}{\alpha_i}
\rfloor& = \sum_i m_i \, , \\
& \widetilde N_\bc(\ell_0,\widetilde \ell_0)&:= \widetilde c_0+\frac{1}{\alpha_0}\overline c -
(\widetilde b_0+\frac{\widetilde\omega_0}{\alpha_0}) \widetilde\ell_0- \frac{1}{\alpha_0}\ell_0 +
\sum_j \lfloor \frac{\widetilde c_j-\widetilde\omega_j\widetilde\ell_0}{\widetilde\alpha_i} \rfloor& =
\sum_j \widetilde m_j \, .
  \end{array}
$$
This number is $\binom{N_\bc(\ell_0,\widetilde \ell_0)+d-1}{d-1}\binom{\widetilde N_\bc(\ell_0,\widetilde \ell_0)
+\widetilde d-1}{\widetilde d-1}$ if $N_\bc$ and $\widetilde N_\bc \geq 0$, otherwise is $0$.
Note that (\ref{eq:CONGR}) guarantees that both $N_\bc$ and $\widetilde N_\bc $ are integers.
Furthermore, (c) and (d) and (\ref{eq:AA}) show  that
the exponent of $t$ and $\wtt$ in the formula of $Z_h^{/H}(t,\wtt)$
are  equal to $\ell_0+c$ and  $\widetilde \ell_0 +\widetilde c$ respectively. Hence
$$Z_h^{/H}(t,\wtt)=\sum  
\ \binom{N_\bc(\ell,\widetilde \ell)+d-1}{d-1}
\ \binom{\widetilde N_\bc(\ell,\widetilde \ell)+\widetilde d-1}{\widetilde d-1}\
 t^{\ell+c} \ \wtt^{\widetilde \ell +\widetilde c},
$$
where the sum runs over $(\ell,\widetilde \ell)\in \Z^2$ with  $\ell+\widetilde\omega_0 \widetilde\ell \equiv
\overline c \ (\mbox{mod} \ \alpha_0)$.

The numerator of $Z(t,\wtt)$ is
$\big(1-t^{-(E^*_{0}, E^*_0)}\wtt^{-(E^*_{0}, \widetilde E^*_0)}[g_0]\big)^{d-1}\cdot
\big(1-t^{-(\widetilde E^*_{0}, E^*_0)}\wtt^{-(\widetilde E^*_{0}, \widetilde E^*_0)}[\widetilde g_0]
\big)^{\widetilde{d}-1}$.
Hence we get $Z^e$ by multiplying this expression by $\sum_hZ_h^{/H}[h]$.
Recall that
 $h=c_0 g_0+\widetilde c_0 \widetilde g_0 +\overline c\, \overline{g}+
\sum_i c_i g_i+\sum_j \widetilde c_j \widetilde g_j$ is paired  with ${\bf c}$. Set
$h':=h+ k g_0+\widetilde k \widetilde g_0$ which corresponds  to ${\bf c'}={\bf c}+(k,
\widetilde k,0,0,0)$. Hence $Z_{h'}[h']$ is the next  sum according to the  decompositions
$h'=h+ k g_0+\widetilde k \widetilde g_0$:
\begin{equation*}
\begin{array}{l}
\sum_{k=0}^{d-1}(-1)^k\binom{d-1}{k}\sum_{\widetilde k=0}^{\widetilde d-1}
(-1)^{\widetilde k}\binom{\widetilde d-1}{\widetilde k} \cdot \\ \hspace{1cm}\sum_h \
\left(
\sum_{\equiv_{\overline c}}\ \binom{N_\bc(\ell,\widetilde \ell)+d-1}{d-1}
\ \binom{\widetilde N_\bc(\ell,\widetilde \ell)+\widetilde d-1}{\widetilde d-1}\
 t^{\ell+c+\frac{-\widetilde e k+\widetilde k/\alpha_0}{\varepsilon}}
  \ \wtt^{\ \widetilde \ell +\widetilde c+
 \frac{-e\widetilde k+k/\alpha_0}{\varepsilon}}\right)\ [h']=
\end{array}
\end{equation*}
$$\textstyle{
\sum_{k=0}^{d-1}(-1)^k\binom{d-1}{k}\sum_{\widetilde k=0}^{\widetilde d-1}
(-1)^{\widetilde k}\binom{\widetilde d-1}{\widetilde k} \cdot
\sum_{h}\left(
\sum_{\equiv_{\overline c}}\ \binom{N_{\bc'}(\ell,\widetilde \ell)-k+d-1}{d-1}
\ \binom{\widetilde N_{\bc'}(\ell,\widetilde \ell)-\widetilde k +\widetilde d-1}{\widetilde d-1}\
 t^{\ell+c'} \ \wtt^{\widetilde \ell +\widetilde c'}\right)\
 }[h'].$$
Changing the orders of the sums and using the combinatorial formula
$\sum_{k=0}^{d-1}(-1)^k\binom{N-k+d-1}{d-1}\binom{d-1}{k}=1$ for $N\geq 0$ and $=0$ otherwise,
we get the following.
\begin{theorem}\label{th:UJZH}
For any $h\in H$ one has
\begin{equation}\label{eq:ZHTT}
 Z_h(t,\wtt)=\sum_{(\ell,\widetilde \ell)\in \cS_\bc} t^{\ell+c} \ \wtt^{ \ \widetilde\ell+\widetilde c},
 \ \ \ \ \mbox{where}
\end{equation}
\begin{equation}\label{eq:module}
\cS_\bc:=\left\{ (\ell,\widetilde \ell) \in \Z^2 \ : \ N_\bc(\ell,\widetilde \ell),
\widetilde N_\bc(\ell,\widetilde \ell)\geq 0\ \ \mbox{and} \
\ \ell+\widetilde\omega_0 \widetilde\ell \equiv
\overline c \ (\mbox{mod} \ \alpha_0)
\right\}.
\end{equation}\end{theorem}
It is straightforward to verify that
the right hand side of (\ref{eq:ZHTT})  does not depend on the choice of $\bc$, it depends only on $h$.
The identity (\ref{eq:ZHTT}) is remarkable: it realizes the bridge between the series $Z^e$ and the equivariant Hilbert series of
{\it affine monoids and their modules}.

\subsection{The structure of $\cS_\bc$}

Recall that for any $h \in H$ and $\bc$ we consider a lift of $h$ identified by certain
${\bf c}$ which detemines the pair $(c,\widetilde c)$ (cf.  (\ref{eq:AA})), and the integers
 $N_\bc(\frl)$ and $\widetilde N_\bc(\frl)$, where  $\frl=(\ell,\well) \in \Z^2$.
We abridge  the congruence condition $\ell+\widetilde \omega_0 \well \equiv \overline c \
(\mbox{mod} \ \alpha_0)$  by $\equiv_\bc$.

If  $h=0$ then we always choose the zero lift with ${\bf c}={\bf 0}$.

If in the definition of  $N_\bc(\frl)$  and  $\widetilde N_\bc(\frl)$
we replace each $[y]$ by $y$, we get  the entries of
$$\left(\begin{array}{c}  A-e\ell_0-\widetilde \ell/\alpha_0 \\ \widetilde A-\ell_0/\alpha_0-\widetilde e \widetilde \ell_0 \end{array}\right)=
  -I^{orb}\left( \begin{array}{c} \ell+c \\   \well+\widetilde c \end{array}\right).$$
This motivates to define
\begin{equation}
\overline\calS_\bc:=\left\{ \frl\in \Z^2 \ : \ -I^{orb}
\left( \begin{array}{c}
 \ell+c \\
 \well+\widetilde c
\end{array}\right) \geq 0 \ \mbox{and} \ \frl \ \mbox{satisfies } \ \equiv_\bc
\right\}.
\end{equation}
Clearly $\calS_{{\bf c}}\subset \overline{\calS}_{{\bf c}}$.
We also consider
$\calC^{orb}$,  the real cone $\{\frl\in\R^2\ :\ -I^{orb}\cdot \frl\geq 0\}$.  Then
$\overline \calS_{\bf c}=\big(\calC^{orb}-(c,\widetilde{c})\big)\cap
\Z^2\cap (\equiv_{{\bf c}})$,
 where $(\equiv_{{\bf c}})$ means  that the elements
 satisfy the congruence $\equiv_{{\bf c}}$ too.

\begin{lemma}\labelpar{lem:fingen} (1) \ $\calS_{\bf 0}$ and $\overline \calS_{\bf 0}$ are affine monoids.
$\overline \calS_{\bf 0}$ is the normalization of  $\calS_{\bf 0}$.

(2) \  $\calS_{{\bf c}}$ and $ \overline{\calS}_{{\bf c}}$ are  finitely generated $\calS_{\bf 0}$-modules,
$\calS_{{\bf c}}$ is a submodule of $ \overline{\calS}_{{\bf c}}$.
\end{lemma}
\begin{proof} (1) is elementary. By Corollary \cite[2.12]{BG}
 $\overline \calS_\bc$ is finitely generated over $\overline\calS_0$, but
$\overline \calS_0$ itself  is finitely generated as an $\calS_0$ module.
\end{proof}
\begin{lemma}\labelpar{lem:VV} There exists $\frv_1$ and $\frv_2$ elements of $\Z^2$ with the following properties:

(a) $\frv_1$ and $\frv_2$ belong to $\calS_{{\bf 0}}$ and
 $\R_{\geq 0}\frv_1+\R_{\geq 0}\frv_2=\calC^{orb}$.

 (b) For any $\frl\in\overline{\calS}_{{\bf c}}$ one has: ($i$) $N_{{\bf c}}(\frl+\frv_1)=
 N_{{\bf c}}(\frl)$; ($ii$)   $N_{{\bf c}}(\frl+\frv_2)\geq 0$; ($\widetilde{i}$)
 $\widetilde{N}_{{\bf c}}(\frl+\frv_2)=\widetilde{N}_{{\bf c}}(\frl)$;
 and ($\widetilde{ii}$) $\widetilde{N}_{{\bf c}}(\frl+\frv_1)\geq 0$.
\end{lemma}
\begin{proof} We choose

(A)  $\frv_1=(\ell_1,\well_1)\in \Z^2\cap (\equiv _{{\bf 0}})$
such that $\{-\omega_i\ell_1/\alpha_i\}=0$ for all $i$, and  $N_{{\bf 0}} (\frv_1)=0$;

(B) $\frv_2=(\ell_2,\well_2)\in \Z^2\cap (\equiv _{{\bf 0}})$
such that $\{-\widetilde{\omega}_j\widetilde{\ell}_2/\widetilde{\alpha}_j\}=0$
for all $j$, and $\widetilde{N}_{{\bf 0}} (\frv_2)=0$.

\noindent  Then $\frv_1$ and $\frv_2$ satisfy (a), and (b)($i$), and (b)($\widetilde{i}$).
Furthermore,
note that $N_{{\bf c}}(\frl+\frv_2)\geq  N_{{\bf c}}(\frl)+N_{{\bf 0}}(\frv_2)$
and for any  $\frl\in\overline{\calS}_{{\bf c}}$
one has $N_{{\bf c}}(\frl)\geq -(d-1)$,
Hence, if we also assume  $\widetilde N_{{\bf 0}}(\frv_1)\geq \widetilde d-1$
  and $N_{{\bf 0}}(\frv_2)\geq d-1$, then all the conditions will be satisfied.
\end{proof}
Usually, the `universal restrictions'  $\widetilde N_{{\bf 0}}(\frv_1)\geq \widetilde d-1$
  and $N_{{\bf 0}}(\frv_2)\geq d-1$  in the proof of
  Lemma \ref{lem:VV}  provide  rather `large' vectors  $\frv_1$ and $\frv_2$. Nevertheless,
usually much smaller vectors also satisfy (a) and (b). Here is another choice.
Besides (A) and (B) we impose the following:

(C) Let $\Box=\Box(\frv_1,\frv_2)=
 \left\{ \frl=q_1\frv_1+q_2\frv_2 \ : \ 0\leq q_1,q_2<1 \right\}$ be
  the semi-open cube in $\calC^{orb}$. Then we require $N_{{\bf 0}}(\frv_2)\geq 0$ and
  $N_{{\bf c}}(\frl_\Box+\frv_2)\geq 0$ for any
  $\frl_\Box\in (\Box-(c,\widetilde c))\cap\Z^2\cap (\equiv_\bc)$; and symmetrically:
  $\widetilde N_{{\bf 0}}(\frv_1)\geq 0$ and
  $\widetilde N_{{\bf c}}(\frl_\Box+\frv_1)\geq 0$ for any
   $\frl_\Box\in (\Box-(c,\widetilde c))\cap\Z^2\cap (\equiv_\bc)$.

The wished inequality for any $\frl\in\overline{\calS}_{{\bf c}}$ then follows from
$N_{{\bf c}}(\frl_\Box +k_1\frv_1+k_2\frv_2+\frv_2)=N_{{\bf c}}(\frl_\Box +k_2\frv_2+\frv_2)\geq
N_{{\bf c}}(\frl_\Box +\frv_2)+k_2N_{{\bf 0}}(\frv_2)$ (and its symmetric version).

\vspace{2mm}

In the sequel the next two subsets of  $\overline{\calS}_{{\bf c}}$ will be crucial.
\begin{eqnarray*}
\calS^-_{{\bf c},1}
:=\left\{ \frl\in (\Box-(c,\widetilde c))\cap\Z^2\cap (\equiv_\bc) \ : \ N_\bc(\frl)<0\right\}\ , \\
\calS^-_{{\bf c},2}:=
\left\{ \frl\in (\Box-(c,\widetilde{c}))\cap\Z^2\cap (\equiv_\bc) \ : \ \widetilde N_\bc(\frl)
<0\right\}.
\end{eqnarray*}
Again, both sets $\calS^-_{{\bf c},1}$ and $\calS^-_{{\bf c},2}$ are independent of the choice of ${\bf c}$, they
depend only on $h$.

\begin{proposition}\labelpar{prop:str} With the above notations one has
\begin{eqnarray*}
 &(1)&  \overline\calS_\bc= \bigsqcup_{\frl\in (\Box-(c,\widetilde{c}))\cap \Z^2\cap (\equiv_\bc)} \frl+\Z_{\geq 0}\frv_1+\Z_{\geq 0}\frv_2 \\
 &(2)&  \overline\calS_\bc\setminus \calS_\bc = \big(\bigsqcup_{\frl\,\in \calS^-_{\bc,1}}
 \frl+\Z_{\geq 0}\frv_1\big) \ \bigcup\ \big(\bigsqcup_{\frl\,\in \calS^-_{\bc,2}}
 \frl+\Z_{\geq 0}\frv_2\big), \\
 & &\mbox{where} \ \ \big(\bigsqcup_{\frl\,\in \calS^-_{\bc,1}}
 \frl+\Z_{\geq 0}\frv_1\big) \ \bigcap\ \big(\bigsqcup_{\frl\,\in \calS^-_{\bc,2}}
 \frl+\Z_{\geq 0}\frv_2\big)= \bigsqcup_{\frl\,\in \calS^-_{\bc,1} \cap \calS^-_{\bc,2}} \frl \ .
\end{eqnarray*}
\end{proposition}
\begin{proof} The statements follow from the choice of $\frv_1$ and $\frv_2$ and the above properties (a) and (b).
Compare also with the structure theorem  \cite[4.36]{BG} of $\calS_{{\bf 0}}$ modules. \end{proof}

\subsection{The periodic constant and the SW invariant in the equivariant case. } \

\vspace{2mm}

Set $\bt=(t,\wtt)$.
Using (\ref{eq:ZHTT}) and Proposition \ref{prop:str}
one can write $Z_h(\bt)/\bt^{(c,\widetilde c)}$ in the next form:
$$\sum_{\frl\in (\Box-(c,\widetilde c))\cap \Z^2\cap (\equiv_\bc)}
\frac{\bt^\frl}{(1-\bt^{\frv_1})(1-\bt^{\frv_2})}
-\sum_{\frl\in \calS^-_{\bc,1}}\frac{\bt^\frl}{1-\bt^{\frv_1}} -\sum_{\frl\in \calS^-_{\bc,2}}
\frac{\bt^\frl}{1-\bt^{\frv_2}}+\sum_{\frl\in \calS^-_{\bc,1}\cap \calS^-_{\bc,2}}\bt^\frl.$$

Next, we apply the decomposition established in subsection \ref{ss:POLPART}.
Here it is important to {\it choose ${\bf c}$ in such a way that
$c\in[0,1)$ and $\widetilde c\in[0,1)$}.

Note that $\frv_1\in\R_{>0}(1/\alpha_0,-e)$ and $\frv_2\in\R_{>0}(-\widetilde e,1/\alpha_0)$,
hence $\frv_2$ sits in the cone determined by $\frv_1$ and $(1,0)$. Then, as in \ref{ss:POLPART},
we set $\Xi_1:=\{(\ell,\widetilde \ell)\,:\, 0\leq \ell < \mbox{first coordinate of $\frv_1$}\}$
and $\Xi_2:=\{(\ell,\widetilde \ell)\,:\, 0\leq \widetilde \ell < \mbox{second  coordinate of $\frv_2$}\}$, and for any $\frl\in \calS^-_{\bc,i}$  the unique
$n_{\frl,i}$ such that $\frl-n_{\frl,i}\frv_i\in\Xi_i$, $i=1,2$.
Then subsection \ref{ss:POLPART} provides the following decomposition
\begin{equation*}\label{eq:2nZ+}
\begin{array}{l}
 Z^+_h(\bt)=\bt^{(c,\widetilde c)}
 \left( \sum_{\frl\in \calS^-_{\bc,1}}\sum_{j=1}^{n_{\frl,1}}\bt^{\frl-j\frv_1} +
 \sum_{\frl\in \calS^-_{\bc,2}}
\sum_{j=1}^{n_{\frl,2}}\bt^{\frl-j\frv_2}+\sum_{\frl\in \calS^-_{\bc,1}\cap
\calS^-_{\bc,2}}\bt^\frl
\right) \ \\
Z^-_h(\bt)=\bt^{(c,\widetilde c)} \left( \sum_{\frl\in (\Box-(c,\widetilde c))\cap \Z^2\cap (\equiv_\bc)}
\frac{\bt^\frl}{(1-\bt^{\frv_1})(1-\bt^{\frv_2})} -\sum_{\frl\in \calS^-_{\bc,1}}\frac{\bt^{\frl-n_{\frl,1}\frv_1}}{1-\bt^{\frv_1}}
-\sum_{\frl\in \calS^-_{\bc,2}}\frac{\bt^{\frl-n_{\frl,2}\frv_2}}{1-\bt^{\frv_2}} \right) \ .
\end{array}
\end{equation*}
Therefore, by \ref{rem:fh} and Theorem \ref{th:POLPART} we get
$${\rm pc}^{\calC^{orb}}_h(Z)=
{\rm pc}^{\calC^{orb}}(Z_h/\bt^{(c,\widetilde c)})=Z^+_h(1,1)=
\sum_{\frl\in \calS^-_{\bc,1}}n_{\frl,1}+
\sum_{\frl\in \calS^-_{\bc,2}}n_{\frl,2}+
|\calS^-_{\bc,1} \cap \calS^-_{\bc,2}| \ .$$

\begin{corollary} Choose  ${\bf c}$ in such a way that
$c\in[0,1)$ and $\widetilde c\in[0,1)$. Then one has the following combinatorial formula for the
  normalized Seiberg--Witten invariant of $M$
 \begin{equation*}
-\frac{(K+2r_h)^2+|\cV|}{8}-\frsw_{-h*\sigma_{can}}(M)=
\sum_{\frl\in \calS^-_{\bc,1}}n_{\frl,1}+
\sum_{\frl\in \calS^-_{\bc,2}}n_{\frl,2}+
|\calS^-_{\bc,1} \cap \calS^-_{\bc,2}|.
\end{equation*}
\end{corollary}
\begin{proof}
Use Corollary \ref{cor:4.1}, the reduction Theorem \ref{th:REST} and the above computation.
\end{proof}

\begin{example} Consider the following plumbing graph
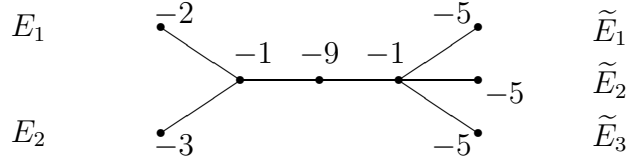
\vspace{0.5cm}
\begin{center}
\begin{picture}(140,60)(80,15)
\put(110,40){\circle*{3}}
\put(140,40){\circle*{3}}
\put(170,40){\circle*{3}}
\put(200,40){\circle*{3}}
\put(80,60){\circle*{3}}
\put(80,20){\circle*{3}}
\put(200,60){\circle*{3}}
\put(200,20){\circle*{3}}
\put(110,40){\line(1,0){90}}
\put(110,40){\line(-3,2){30}}
\put(110,40){\line(-3,-2){30}}
\put(170,40){\line(3,2){30}}
\put(170,40){\line(3,-2){30}}
\put(30,20){\makebox(0,0){$E_2$}}
\put(30,60){\makebox(0,0){$E_1$}}
\put(250,40){\makebox(0,0){$\widetilde E_2$}}
\put(250,60){\makebox(0,0){$\widetilde E_1$}}
\put(250,20){\makebox(0,0){$\widetilde E_3$}}
\put(85,65){\makebox(0,0){$-2$}}
\put(85,15){\makebox(0,0){$-3$}}
\put(210,35){\makebox(0,0){$-5$}}
\put(190,65){\makebox(0,0){$-5$}}
\put(190,15){\makebox(0,0){$-5$}}
\put(115,50){\makebox(0,0){$-1$}}
\put(140,50){\makebox(0,0){$-9$}}
\put(165,50){\makebox(0,0){$-1$}}

\end{picture}
\end{center}
\vspace{0.5cm}

\noindent
The corresponding Seifert invariants are  $\alpha_1=2$, $\alpha_2=3$, $\widetilde{\alpha}_j=5$, $\alpha_0=9$ and
$\omega_i=\widetilde\omega_j=\omega_0=\widetilde\omega_0=1$ for all $i$ and  $j$.
Hence  $e=-1/18$, $\widetilde e=-13/45$ and $\varepsilon=1/(3^3\cdot 10)$.
For $h=0$ we choose $\bc=0$. Then
\begin{equation*}
\calS_0=\left\{
\begin{array}{l} (\ell,\well)\in \Z^2 \\
 8\ell - \well +9\cdot([\frac{-\ell}{2}]+[\frac{-\ell}{3}])\geq 0  \\
 8\well -\ell +27\cdot[\frac{-\well}{5}]\geq 0\\
 \ell+\well \equiv 0 \ (\mbox{mod} \ 9 )
\end{array}
\right\} \ \ \mbox{and} \ \ \
\overline\calS_0=\left\{
\begin{array}{l} (\ell,\well)\in \Z^2 \\
 \ell-2\well \geq 0 \\
 -5\ell+13\well \geq 0 \\
 \ell+\well \equiv 0 \ (\mbox{mod} \ 9 )
\end{array}
\right\} .
\end{equation*}
If we take the generators $\frv_1=(60,30)$ and $\frv_2=(26,10)$ (via conditions (A)-(B)-(C) following
Lemma \ref{lem:VV}), one can calculate explicitly the sets
\begin{equation*}
\calS^-_{{\bf 0},1}=\left\{
\begin{array}{l}
(13,5),(19,8),(25,11), (31,14),\\
(37,17),(43,20),(49,23),\\
(55,26),(61,29),(67,32)
\end{array}
\right\} \ \mbox{and} \ \
\calS^-_{{\bf 0},2}=\left\{
\begin{array}{l}
(6,3),(19,8),(12,6),\\
(25,11),(24,12),(37,17),\\
(42,21),(55,26)
\end{array}\right\} .
\end{equation*}
This generates the next counting function of
$\overline \calS_{{\bf 0}}\setminus \calS_{{\bf 0}}$, namely
$\sum_{(\ell, \well)\in \overline \calS_{{\bf 0}}\setminus \calS_{{\bf 0}}}
\ t^\ell \widetilde{t}^{\well}\ =$
\begin{equation*}
\begin{array}{l}
(t^{13}\wtt^5+t^{19}\wtt^8+t^{25}\wtt^{11}+t^{31}\wtt^{14}+t^{37}\wtt^{17}+
t^{43}\wtt^{20}+t^{49}\wtt^{23}+t^{55}\wtt^{26}+t^{61}\wtt^{29}+t^{67}\wtt^{32})/(1-t^{60}\wtt^{30})+\\
 + (t^6\wtt^3+t^{12}\wtt^6+t^{19}\wtt^8+t^{24}\wtt^{12}+t^{25}\wtt^{11}+t^{37}\wtt^{17}
 +t^{42}\wtt^{21}+t^{55}\wtt^{26})/
 (1-t^{26}\wtt^{10})-\\
 -t^{19}\wtt^8-t^{25}\wtt^{11}-t^{37}\wtt^{17}-t^{55}\wtt^{26} \ ,
\end{array}
\end{equation*}
which by \ref{eq:2nZ+} provides  $Z^+_0(t,\wtt)=t\wtt^{-1}+t^3\wtt^{2}+t^{-2}\wtt^2+t^{-1}\wtt+t^{11}\wtt^7+t^{16}\wtt^{11}+t^{-10}\wtt+
t^{29}\wtt^{16}+t^{3}\wtt^{6}+t^{19}\wtt^8+t^{25}\wtt^{11}+t^{37}\wtt^{17}+t^{55}\wtt^{26}$.
Hence ${\rm pc}^{\calC^{orb}}_0(Z)=Z^+_0(1,1)=13$.

[It can be verified that there exists a splice quotient type normal surface singularity whose
link is given by the above graph. It is a complete intersection in $(\C^4,0)$ with equations
$z^3+(y_2+2y_3)^2-y_1y_2(2y_2+3y_3)=y_1^5+(2y_2+3y_3)y_2y_3=0$.
\end{example}

\subsection{The periodic constant in the non-equivariant case and  $\lambda(M)$.}\labelpar{ss:NEC} \

\vspace{2mm}

Though the non-equivariant $ Z_{ne}$  can be obtained by the sum $\sum_hZ_h$ treated in the previous subsection, here we
provide a more direct procedure, which leads to a new formula.
Write $J:=(-I^{orb})^{-1}$ and $\bt^{\binom{a}{b}}$ for $t^a\wtt^{b}$.
Applying the reduction \ref{th:REST} for the definition
\ref{eq:INTR} of $Z$, we get
\begin{eqnarray*}
 Z_{ne}(\bt)=\frac{(1-\bt^{J\binom{1}{0}})^{d-1}(1-\bt^{J\binom{0}{1}})^{\widetilde d-1}}
{\prod_i(1-\bt^{J\binom{1/\alpha_i}{0}})\prod_j(1-\bt^{J\binom{0}{1/\widetilde \alpha_j}})}.
\end{eqnarray*}
Set
$S(x):= \sum_i x_i/\alpha_i$ and $\widetilde S(\widetilde x):= \sum_j \widetilde{x}_j/\widetilde{\alpha}_j$.
 Similarly as in \ref{eq:NE}, $Z_{ne}(\bt)$ can be written as
\begin{eqnarray*}
\sum_{0\leq x_i<\alpha_i, 0\leq i\leq d\atop 0\leq \widetilde x_j<\widetilde\alpha_j,
0\leq j \leq\widetilde d}f(x,\widetilde x), \ \ \mbox{where} \ \ f(x,\widetilde x)=
 \frac{\bt^{J\binom{S(x)}{\widetilde S(\widetilde x)}}}{(1-\bt^{J\binom{1}{0}})
(1-\bt^{J\binom{0}{1}})}.
\end{eqnarray*}
By the substitution $u_1=\bt^{J\binom{1}{0}}$ and $u_2 =\bt^{J\binom{0}{1}}$, $f(x,\widetilde x)$ transforms into  $u_1^{S(x)}u_2^{\widetilde S(\widetilde x)}/ (1-u_1)(1-u_2)$.
The division of this fraction (with remainder) is elementary, hence $f(x,\widetilde x)$ equals
$$
\bt^{J\binom{S_{rat}}{\widetilde S_{rat}}} \left( \sum_{n=0}^{S_{int}-1}
\sum_{k=0}^{\widetilde S_{int}-1}\bt^{J\binom{n}{k}}
-\sum_{k=0}^{S_{int}-1}\frac{\bt^{J\binom{k}{0}}}{1-\bt^{J\binom{0}{1}}} -\sum_{\widetilde k=0}^{\widetilde S_{int}-1}
\frac{\bt^{J\binom{0}{\widetilde k}}}{1-\bt^{J\binom{1}{0}}}+ \frac{1}{(1-\bt^{J\binom{1}{0}})
(1-\bt^{J\binom{0}{1}})}\right)\ ,$$
where $S_{int}:=\lfloor S(x)\rfloor$, $\widetilde S_{int}:=\lfloor \widetilde
S(\widetilde x)\rfloor$,
$S_{rat}:=\{S(x)\}$ and $\widetilde S_{rat}:=\{\widetilde S(\widetilde x)\}$.

Then, by \ref{lem:d2},
${\rm pc}^{\calC^{orb}}(\bt^{J\binom{S_{rat}}{\widetilde S_{rat}}}/ (1-\bt^{J\binom{1}{0}})
(1-\bt^{J\binom{0}{1}}))=0$. Moreover, \ref{ss:POLPART} gives a unique integer $s(k)\geq 0$
for $k\in \{0,\dots,S_{int}-1\}$ such that $\bt^{J\binom{k+S_{rat}}{-s(k)+\widetilde S_{rat}}}/1-\bt^{J\binom{0}{1}}$
has vanishing periodic constant with respect to $\calC^{orb}$. It turns out that $s(k)=\lfloor
-\widetilde e \alpha_0(k+S_{rat})+\widetilde S_{rat}\rfloor$.
Similarly $s(\widetilde k)=\lfloor -e\alpha_0(\widetilde k+
\widetilde S_{rat})+S_{rat}\rfloor$ in the case of
$\bt^{J \binom{-s(\widetilde k)+S_{rat}}{\widetilde k+
\widetilde S_{rat}}}/1-\bt^{J\binom{1}{0}}$.
%
Therefore, by \ref{th:POLPART}, for
\begin{equation*}
  \mathrm{pc}(Z_{ne})=
-\lambda(M)-\frd\cdot\frac{K^2+|\cV|}{8}+\sum_h\chi(r_h)
\end{equation*}
we get
\begin{equation*}
\sum_{0\leq x_i<\alpha_i, 0\leq i\leq d\atop 0\leq \widetilde x_j<\widetilde\alpha_j,
0\leq j \leq\widetilde d} \  \Big( S_{int}\widetilde S_{int} + \sum_{k=0}^{S_{int}-1}
 \lfloor -\widetilde e
\alpha_0(k+S_{rat})+\widetilde S_{rat}\rfloor
+ \sum_{\widetilde k=0}^{\widetilde S_{int}-1} \lfloor -e\alpha_0(\widetilde k+ \widetilde S_{rat})+S_{rat}\rfloor  \ \Big).
\end{equation*}

\subsection{Ehrhart theoretical interpretation.}
In general, 
in contrast with the one--node case \ref{ss:Ehr}, the direct determination
 of
the counting function of $Z_h(\bt)$, or equivalently, of the complete equivariant
 Ehrhart quasipolynomial associated with the corresponding polytope,  is rather hard.
Nevertheless, those coefficients which are relevant to us (e.g. those ones which
contain the information about the Seiberg--Witten invariants of the 3--manifold)
can be identified using  the right hand side of \ref{eq:SUM}.
The computation is more transparent
when $L'=L$. In  that case, the two-variable Ehrhart polynomial has degree $d+\widetilde d$,
and a specific $d+\widetilde d-2$ degree coefficient
 is exactly the normalized Seiberg--Witten invariant of the 3--manifold.
We will not provide here the formulae, since
this  identification will be established for any
negative definite plumbing graph with arbitrary number of nodes, see
section  \ref{s:Last}, where several other coefficients will be computed as well.

\section{Ehrhart theoretical interpretation of the SW invariant (the general case)}\labelpar{s:Last}

\subsection{}
Let $\Gamma$ be a negative definite plumbing graph, a connected tree as in \ref{ss:11}.
Let  $\calN$ and $\cale$ be  the set of
nodes and  end--vertices as above. We assume that $\calN\not=\emptyset$.
If  $\delta_n$ denotes the valency of a node $n$, then $|\cale|=
2+\sum_{n\in \calN}( \delta_n-2)$.

We consider the
matrix $J$ with entries  $J_{nm}:=-(E^*_n,E^*_m)$ for $n,m\in\calN$.
By (\ref{ss:11}) it is a principal minor of $-I^{-1}$ (with rows and columns corresponding to the nodes).

Another incarnation of the matrix $J$ already appeared in subsection \ref{ss:NEC},
as the negative of the  inverse of the orbifold intersection matrix.
Indeed, let for any $n\in\calN$ take that component of $\Gamma\setminus
\cup_{m\in\calN\setminus n}\{m\}$
which contains $n$. It is a star shaped graph,
let $e_n$ be its orbifold Euler number. Furthermore, for any
two nodes $n$ and $m$ which are connected by a chain, let $\alpha_{nm}$ be the determinant of
that chain (not including the nodes). Then define the orbifold intersection  matrix
(of size $|\calN|$)
as  $I^{orb}_{nn}=e_n$, $I^{orb}_{nm}=1/\alpha_{nm}$ if the two nodes $n\not=m$ are
connected by a chain,
and  $I^{orb}_{nm}=0$ otherwise; cf. \cite[4.1.4]{BN07} or \ref{ss:TNC}.
One can show (see \cite[4.1.4]{BN07}) that $I^{orb}$ is invertible, negative definite, and
$\det(-I^{orb})$ is the product of $\det(-I)$ with the determinants of all (maximal)
chains and legs of
$\Gamma$. This fact and  \ref{eq:DETsgr} imply that
$J=(-I^{orb})^{-1}$.

\subsection{The Ehrhart polynomial}
In the sequel we assume that $L=L'$, that is $H=0$.

By  \ref{ss:cor}, $P^{(l)}$ sits in $\R^{|\cale|}$. Moreover, by the
reduction theorem \ref{ss:REST}, we can take
 $l$ of the form $l=\sum_{n\in \calN} \lambda_n E_n^*$, from the subcone of the Lipman cone
 generated by $\{E^*_n\}_{n\in\calN}$.

Then \ref{ss:REST} guarantees that  the associated polytope is $P^{(l)}=\bigcup_{n\in \calN}
P_n^{(l_n)}$, $P_n^{(l_n)}$ depending only of the component $l_n=-(l,E_n^*)$. Note that the
coefficients $\{\lambda_n\}_n$ and the entries $\{l_n\}_n$ are connected exactly by the
transformation law
$\left(l_n\right)_n=J \left(\lambda_n\right)_n$.

Take any chamber $\calC$ such that $int(\calC \cap \calS)\neq \emptyset$, as in \ref{cor:4.1}.
Let $\widehat \calL^{\calC}(P,\calT,(\lambda_n)_n)$
be the Ehrhart quasipolynomial  $\calL^{\calC}(P,\calT,(l_n)_n)$, associated with the
 denominator of $Z$, after changing the variables to $(\lambda_n)_n$ via $\left(l_n\right)_n=J \left(\lambda_n\right)_n$. It is convenient to normalize the coefficient of $\prod_n \lambda_n^{m_n}$
 by a factor $\prod _n m_n!$, hence we write
$$\widehat \calL^{\calC}(P,\calT,(\lambda_n)_n)=\sum_{\sum_n m_n\leq |\cale|
\atop m_n\geq 0; \ n\in \calN}  \widehat \fra^{\calC}_{(m_n)_n}\prod_n\frac{\lambda_n^{m_n}}{m_n!},$$
for certain  periodic functions $\widehat \fra^{\calC}_{(m_n)_n}$ in variables $(\lambda_n)_n$.
By \ref{eq:SUM}, \ref{cor:Taylor} and \ref{th:REST}
\begin{equation}\label{eq:Ehrcoeff}
 \chi(\sum_{n\in \calN}\lambda_n E_n^*)+{\rm pc}^{\calS}(Z)=\Delta((\lambda_n)_n),
\end{equation}
where
$$\Delta((\lambda_n)_n)= \sum_{0\leq k_n\leq \delta_n-2\atop \forall n\in\calN}
(-1)^{\sum_n k_n}\prod_n \binom{\delta_n-2}{k_n} \
\widehat \calL^{\calC}(P,\calT,(\lambda_n-k_n)_n)=$$
$$\sum_{\sum_n m_n\leq |\cale| \atop m_n\geq 0; \ n\in \calN}\ \left(\sum_{0\leq p_n\leq m_n \atop
n\in\calN} (-1)^{\sum_n p_n}\cdot \prod_n\binom{m_n}{p_n}
\left( \sum_{k_n=0}^{\delta_n-2}(-1)^{k_n} \binom{\delta_n-2}{k_n} k_n^{p_n}\right)
\right)\cdot \widehat \fra^{\calC}_{(m_n)_n}\prod_n\frac{\lambda_n^{m_n-p_n}}{m_n!}.$$
On the other hand, since
$\chi(l)=-(K+l,l)/2$,  the left hand side of (\ref{eq:Ehrcoeff})
is the quadratic function
$$\sum_{n,m\in\calN}(J_{nm}/2) \lambda_n\lambda_m+\sum_{n\in\calN}(-(K,E_n^*)/2)\lambda_n+{\rm pc}^{\calS}(Z).$$
Now we identify these coefficients with those of $\Delta((\lambda_n)_n)$ above. The additional
ingredient is the combinatorial formula (\ref{eq:PSz}), which also shows that for the non--zero
summands one necessarily has $p_n\geq \delta_n-2$ for any $n$. One gets the following result.
%
\begin{theorem}\label{th:MAINLAST}
\begin{eqnarray*}
 \widehat \fra^{\calC}_{(\delta_n,(\delta_m-2)_{m\neq n})}&=&J_{nn}\\
 \widehat \fra^{\calC}_{(\delta_n-1,\delta_m-1,(\delta_q-2)_{q\neq n,m})}&=&J_{nm} \ \
\mbox{for $n\not=m$}\\
 \widehat \fra^{\calC}_{(\delta_n-1,(\delta_m-2)_{m\neq n})}&=&
 \textstyle {
 -\frac{1}{2}(K,E_n^*)+\frac{1}{2}
 \sum_{m\in \calN}(\delta_m-2)J_{nm}}\\
\end{eqnarray*}
\begin{equation*}
\widehat \fra^{\calC}_{(\delta_n-2)_n}={\rm pc}^{\calS}(Z)-
\textstyle {
\sum_{n\in \calN} \frac{(\delta_n-2)(K,E_n^*)}{4}+\sum_{n\in \calN}
\frac{(\delta_n-2)(3\delta_n-7)J_{nn}}{24}
+ \sum_{n,m\in \calN \atop m\neq n}\frac{(\delta_n-2)(\delta_m-2)J_{nm}}{8}.}
\end{equation*}
\end{theorem}

Recall that   ${\rm pc}^{\calS}(Z)=-(K^2+|\calv|)/8-\lambda(M)$, where $\lambda(M)$ is the
Casson  invariant of $M$. Hence $\widehat \fra^{\calC}_{(\delta_n-2)_n}$ equals
the normalized Casson invariant modulo some `easy terms'.

We emphasize that these formulae also  show  that the above coefficients are constants (as periodic functions
in $(\lambda_n)_n$) and independent of the chosen chamber $\calC$.

\end{document}